\documentclass[12pt]{article}

\usepackage[margin=1in]{geometry}

\usepackage{amsthm,amsmath,amsfonts,amssymb}

\usepackage[authoryear]{natbib}
\usepackage{bibunits}
\defaultbibliographystyle{plainnat}
\defaultbibliography{references}

\usepackage[colorlinks,citecolor=blue,urlcolor=blue]{hyperref}

\usepackage{graphicx}

\usepackage{mathrsfs}
\usepackage{mathtools}
\usepackage{enumitem}
\usepackage{cleveref}
\crefname{assumption}{Assumption}{Assumptions}
\Crefname{assumption}{Assumption}{Assumptions}
\usepackage{bm}
\usepackage{booktabs}
\usepackage{placeins}
\usepackage{float}

\theoremstyle{plain}
\newtheorem{theorem}{Theorem}[section]
\newtheorem{proposition}[theorem]{Proposition}
\newtheorem{lemma}[theorem]{Lemma}
\newtheorem{corollary}[theorem]{Corollary}

\newtheorem*{theorem*}{Theorem}

\theoremstyle{definition}

\newtheorem{example}[theorem]{Example}

\theoremstyle{definition}
\newtheorem{remark}[theorem]{Remark}

\DeclareMathOperator{\Var}{Var}
\DeclareMathOperator{\diam}{diam}
\DeclareMathOperator{\inj}{inj}
\DeclareMathOperator{\Sec}{Sec}

\DeclareMathOperator{\tr}{tr}

\DeclareMathOperator*{\argmin}{arg\,min}

\DeclareMathOperator{\supp}{supp}

\newcommand{\cA}{\mathcal{A}}
\newcommand{\cB}{\mathcal{B}}
\newcommand{\cX}{\mathcal{X}}
\newcommand{\cF}{\mathcal{F}}
\newcommand{\cH}{\mathcal{H}}
\newcommand{\cG}{\mathcal{G}}

\newcommand{\cE}{\mathcal{E}}

\newcommand{\cD}{\mathcal{D}}
\DeclareMathOperator{\Dir}{Dir} 
\newcommand{\bR}{\mathbb{R}}
\newcommand{\bE}{\mathbb{E}}
\newcommand{\bP}{\mathbb{P}}

\newcommand{\bZ}{\mathbb{Z}}
\newcommand{\bH}{\mathbb{H}}
\newcommand{\Id}{\mathrm{Id}}
\newcommand{\dd}{\,\mathrm{d}}
\newcommand{\norm}[1]{\left\lVert #1 \right\rVert}
\newcommand{\abs}[1]{\left\lvert #1 \right\rvert}
\newcommand{\ip}[2]{\left\langle #1,\, #2 \right\rangle}

\begin{document}

\title{Harmonic Map Regression: Rate-Optimal Nonparametric Estimation on Manifolds with Topological Recovery}

\author{Xiaoyu Chen\\[4pt]
  \textit{Department of Industrial and Systems Engineering,
  University at Buffalo}\\[2pt]
  }
\date{}
\maketitle

\begin{abstract}
We study harmonic map regression, a nonparametric estimator for manifold-valued responses, that penalizes the empirical Fr\'echet risk by the Dirichlet energy.  By connecting penalized regression to the theory of harmonic maps, the estimator acquires a structural theory that parallels the classical Euclidean smoothing spline.  The Euler--Lagrange equation characterizes the solution as a piecewise-geodesic spline, an equivalent kernel controls pointwise risk at the rate $n^{-2/3}$, and the infinite-dimensional variational problem reduces exactly to a finite-dimensional optimization.  Such newly established connection reveals a topological phenomenon that has no analogue in Euclidean nonparametric regression and, to our knowledge, has not been studied in the manifold regression literature.  On manifolds whose regression curves can wrap around in topologically distinct ways, maps in distinct homotopy classes are separated by energy barriers intrinsic to the geometry of the target, and the Dirichlet penalty makes the estimator sensitive to this structure, recovering the correct topological class with probability tending to one, a phase transition we call topological recovery.  A curvature-dependent oracle inequality yields the minimax rate $n^{-2s/(2s+1)}$ for Sobolev order~$s$, matching the Euclidean constant on non-positively curved targets, while five geometric obstructions show that the full structural theory is unique to the Dirichlet energy ($s=1$).  Simulations on $S^2$, $\mathbb{H}^2$, $SO(3)$, $\mathrm{Sym}^+(2)$, and~$T^2$ corroborate the theory, and an application to wind-direction data on~$S^1$ illustrates practical advantages.
\end{abstract}



\begin{bibunit}

\section{Introduction}\label{sec:intro}

Nonparametric regression with a Sobolev penalty is one of the most
thoroughly understood estimation problems in statistics.  When
the response is real-valued, the smoothing spline estimator that
minimizes $R_n(f)+\lambda\int|f^{(s)}|^2$ achieves the minimax
optimal rate $n^{-2s/(2s+1)}$ \citep{Stone1982}.  Beyond rate
optimality, the estimator possesses a structural theory.  The
Euler--Lagrange equation produces a polynomial spline characterization,
the Green's function of the associated differential operator serves as
an equivalent kernel, and together they yield an explicit pointwise
bias-variance decomposition \citep{Wahba1990, Silverman1984}.  This
structural chain, from the variational penalty to pointwise risk control
via PDE theory explains the mechanism by which the estimator achieves its performance.
This paper shows that penalizing by the Dirichlet energy connects
manifold-valued regression to the theory of harmonic maps, and that
this connection produces the Riemannian analogue of the entire
structural chain together with a topological phenomenon that has no
Euclidean counterpart.

In many modern applications the response takes values not in~$\bR$ but
in a Riemannian manifold $(M,g)$.  For example, directional data lie on spheres
$S^d$, covariance matrices on the manifold $\mathrm{Sym}^+(k)$ of
symmetric positive-definite matrices, and rotational measurements on
the Lie group $\mathrm{SO}(3)$.  A natural approach to nonparametric
regression in this setting is to replace the squared-error loss by
the squared geodesic distance $d_M^2$ and the Sobolev seminorm by an
intrinsic analogue, such as the Dirichlet energy
$\Dir(F)=\frac{1}{2}\int_0^1|F'|_g^2\,dt$.  However, the passage from
$\bR^D$ to a curved manifold introduces several difficulties, and it is
not clear a priori whether the Euclidean structural chain survives.

First, the Riemannian metric introduces curvature, represented by the sectional curvature.  Comparison
inequalities for geodesic triangles depend on the sign of the sectional
curvature, so standard variance bounds in~$\bR^D$ acquire a
curvature-dependent correction factor on
positively curved targets \citep{Karcher1977}.
Second, when $M$ has non-trivial fundamental group, the
Dirichlet energy develops homotopy-class-dependent lower bounds.  A map
from $S^1$ to $S^1$ with winding number~$k$ must have energy at least
$2\pi^2 k^2/L$, and this topological constraint creates a statistical
challenge with no Euclidean analogue.
\Cref{fig:winding} illustrates such topological difficulty.
When the winding number is small ($k=0.5$), all methods succeed;
but when the data wrap fully around $S^1$ ($k=2$), log-map-based
methods collapse to near-constant fits because of cut-locus
discontinuities, whereas the Dirichlet energy penalty correctly
recovers the winding number.

\begin{figure}[t]
  \centering
  \includegraphics[width=1\linewidth]{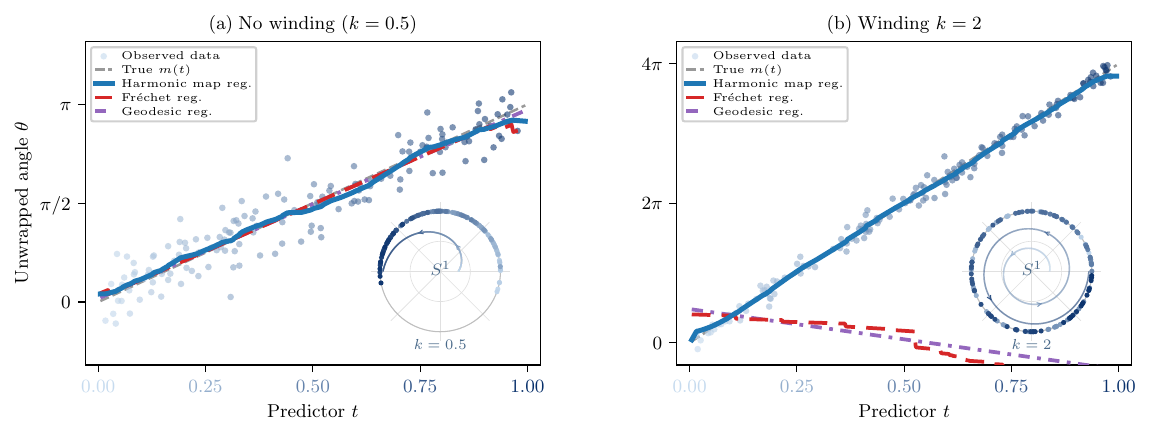}
  \caption{Winding ambiguity in circle-valued regression ($n=150$, $\sigma=0.3$).
  Main panels plot the unwrapped angle $\theta$ against the predictor~$t$,
  with colored dots showing noisy observations, the dashed gray line
  the true regression function $m(t)=2\pi k t$, and solid/dashed curves
  the fitted estimators.  Insets visualize the winding on~$S^1$ by plotting
  the angle $m(t)$ at increasing radius to reveal the number of
  wraps, with color varying from
  light blue (early~$t$) to dark blue (late~$t$).
  (a)~Low winding ($k=0.5$): all methods succeed.
  (b)~High winding ($k=2$): Fr\'echet and geodesic regression
  collapse to near-constant fits, and only the proposed estimator
  (harmonic map regression, see~\eqref{eq:estimator-intro})
  recovers the correct winding number.}
  \label{fig:winding}
\end{figure}

Several approaches to manifold-valued regression have been developed.
Geodesic regression \citep{Fletcher2013, MuralidharanFletcher2012,
Shi2012} fits $t\mapsto\exp_{y_0}(tv)$, which cannot represent
nonlinear regression functions.
Intrinsic regression models for Riemannian symmetric spaces
\citep{CorneaZhu2017} exploit the algebraic structure of
symmetric spaces to build semiparametric models, but their
framework does not extend to
general Riemannian targets.  Local polynomial methods for
symmetric positive-definite matrices \citep{YuanZhu2012}
are similarly restricted in scope.
The Fr\'echet regression framework, initiated by
\citet{Petersen2019} and developed extensively by M\"uller and
collaborators provides
a systematic theory of regression for random objects in general
metric spaces.
Fr\'echet regression \citep{Petersen2019, Schotz2022}
achieves $\sqrt{n}$-consistency but imposes no smoothness structure.
\citet{Schotz2022nonparametric} extends this to kernel-smoothed
Fr\'echet regression with convergence rates on NPC targets,
a construction related to the kernel-smoothed estimator in \S6
in the Appendix.
\citet{LinMuller2021TV} introduced TV-regularized Fr\'echet regression
in Hadamard spaces, establishing the rate $n^{-2/3}$. However,
the $L^1$-type penalty does not admit an Euler--Lagrange equation,
pointwise risk bounds are limited to the local averaging rate
$n^{-1/2}$, and simple connectivity of Hadamard spaces excludes
targets with non-trivial fundamental group.
Work on Riemannian spline interpolation
\citep{Noakes1989, Machado2010, Kim2021spline},
intrinsic polynomial splines with higher-order covariant
penalties \citep{HinkleFletcherJoshi2014},
trajectory models on rotation groups
\citep{SuJuddWuFletcherSrivastava2014},
and graph-Laplacian estimation \citep{SteinkeHein2010}
develop geometric structure without establishing convergence rates.
The higher-order penalties of \citet{HinkleFletcherJoshi2014}
correspond to $s\ge 2$, for which the structural theory does not
extend (\Cref{sec:obstructions}).
In the Euclidean setting, the smoothing spline owes its effectiveness
not merely to rate optimality but to a structural chain. Specifically, the
Euler--Lagrange equation yields a spline characterization, whose
Green's function serves as an equivalent kernel and delivers pointwise
risk control. Such structural chain explains why the estimator works.
However, none of the above manifold regression formulations produces such structural understanding.

In this paper, we study the estimator
\begin{equation}\label{eq:estimator-intro}
  \hat F_n \;=\; \argmin_{F\in\cH^s(I,M)}\;
  \bigl\{R_n(F) + \lambda_n\, J_s(F)\bigr\},
\end{equation}
where $I=[0,1]$ without loss of generality, $R_n(F)=n^{-1}\sum_{i=1}^n d_M^2(Y_i,F(t_i))$ is the
empirical Fr\'echet risk for observations $(t_i,Y_i)\in I\times M$,
$\cH^s(I,M)$ is the Sobolev space of maps into~$M$ (defined via an
isometric Nash embedding), and $J_s(F)=|F|_{H^s}^2$ is the $s$-th order
Sobolev energy.  When $s=1$,
$J_1(F)=2\Dir(F)=\int_0^1|F'|_g^2\,dt$ is twice the Dirichlet
energy, and the critical points of~$\Dir$ are geodesics
\citep{EellsSampson1964}.  We call~\eqref{eq:estimator-intro}
\emph{harmonic map regression}.

At $s=1$ the Dirichlet penalty connects the regression problem to the
theory of harmonic maps
\citep{EellsSampson1964, Schoen1984, Helein2002, KorevaarSchoen1993},
and this connection yields a structural theory that parallels the
Euclidean smoothing spline.  Four families of results make this precise.

\emph{Euler--Lagrange equation and geodesic spline.}
On the geometric side (\Cref{sec:regularity}), the estimator
satisfies a perturbed harmonic map equation (\Cref{thm:EL}).  Between
consecutive design points $t_i$ it is a geodesic in~$M$, with derivative
jumps proportional to the data residuals (\Cref{thm:regularity}).

\emph{Equivalent kernel and pointwise risk.}
The geodesic spline representation yields an equivalent-kernel pointwise
risk bound of order $n^{-2/3}$ (\Cref{thm:pointwise}).
Moreover, the geodesic spline structure implies an exact
finite-dimensional reduction (\Cref{prop:exact-reduction}), under which the
infinite-dimensional continuous problem~\eqref{eq:estimator-intro} is
equivalent to a finite-dimensional optimization on $M^n$ with
zero discretization error, closing the theory-computation gap
completely.

\emph{Topological phase transition.}
On the topological side (\Cref{sec:topology}), when the target manifold
has non-trivial fundamental group ($\pi_1(M)\neq 0$, e.g., $SO(3)$, $T^2$),
maps in distinct homotopy classes are separated by energy barriers intrinsic
to the geometry of~$M$.  The minimum Dirichlet energy in each class is
strictly positive and proportional to the squared length of the shortest
geodesic representative (\Cref{thm:topological}).  The Dirichlet penalty
makes the estimator sensitive to this structure.  Any estimator confined to
the wrong homotopy class incurs excess risk of order at least $\sqrt{\lambda_n}$,
strictly larger than the minimax rate (\Cref{thm:topo-cost}), while the
penalized estimator recovers the correct class with probability tending to
one (\Cref{thm:topo-phase}).

\emph{Rate optimality.}
The geometric penalty does not sacrifice statistical efficiency.
For any integer $s\ge 1$, we prove the oracle inequality
\begin{equation}\label{eq:oracle-intro}
  \bE\bigl[\cE(\hat F_n)\bigr]
  \;\le\; C_1\inf_{F\in\cH^s(I,M)}
  \bigl\{\cE(F)+\lambda_n J_s(F)\bigr\}
  \;+\;\frac{C_2}{n\,\lambda_n^{1/(2s)}},
\end{equation}
where $R(F)=\bE[d_M^2(Y,F(t))]$ is the population risk,
$m(t)=\argmin_{y\in M}\bE[d_M^2(Y,y)\mid t]$ is the regression
function, $\cE(F)=R(F)-R(m)$ is the excess risk, and the constants
depend on the curvature of~$M$ through a curvature factor
$\eta\in(0,1]$ (defined in \eqref{eq:curvature-factor}).
For non-positively curved targets $\eta=1$ and the constant
matches the Euclidean theory
(\Cref{thm:NPC-sharp}).  The resulting rate $n^{-2s/(2s+1)}$ is minimax
optimal (\Cref{thm:minimax}), and a high-probability version is given
in \Cref{cor:hp-oracle}.
The minimax rate itself holds for any Sobolev-type penalty satisfying
standard entropy conditions.  What distinguishes the Dirichlet energy
is the structural theory.  The geodesic spline, the equivalent kernel,
and the topological phase transition are all consequences of the
harmonic map connection and do not extend to higher-order penalties.
\Cref{sec:obstructions} makes this precise through five geometric
obstructions, including topological blindness of the bienergy
(\Cref{prop:bienergy-blind}), showing that the restriction to $s=1$
is necessary.

The formal theory assumes compact target manifolds. The noncompact
targets $\bH^d$ and $\mathrm{Sym}^+(k)$ in \Cref{sec:numerics} are
treated in bounded-region regimes where the results apply locally.
\Cref{sec:prelim} fixes notation, \Cref{sec:model} states the model,
\Cref{sec:main} establishes rate optimality, \Cref{sec:regularity}
develops the geodesic spline, equivalent kernel, and exact
finite-dimensional reduction,
\Cref{sec:topology} presents the topological phase transition,
\Cref{sec:obstructions} identifies the obstructions for higher-order
penalties, \Cref{sec:numerics} presents numerical experiments, and
\Cref{sec:discussion} discusses open problems.  Full proofs appear
in the Appendix.

\section{Geometric and statistical preliminaries}\label{sec:prelim}

Standard Riemannian notation follows \citet{doCarmo1992} and
\citet{Lee2018}.  We write $(M,g)$ for a smooth compact connected
Riemannian $d$-manifold without boundary, $d_M$ for the geodesic
distance, $\exp_y\colon T_yM\to M$ for the exponential map,
$\log_y$ for its local inverse, $\Sec_M$ for the sectional
curvature, $\inj(M)$ for the injectivity radius, and
$B_M(y,r)$ for the open geodesic ball of radius~$r$.
Informally, $\exp_y(v)$ follows the shortest path (geodesic)
starting at~$y$ with velocity~$v$ for unit time, and $\log_y(p)$
returns the initial velocity of the geodesic from~$y$ to~$p$.
The injectivity radius $\inj(M)$ is the largest radius~$r$ such
that, for every $y\in M$, the exponential map $\exp_y$ is a
diffeomorphism on the ball of radius~$r$ in~$T_yM$; equivalently,
within distance~$\inj(M)$ of any point, shortest paths are unique
and the logarithmic map $\log_y$ is smooth.
The sectional curvature $\Sec_M$ measures how geodesics spread:
positive curvature (e.g., spheres) makes nearby geodesics converge,
negative curvature (e.g., hyperbolic space) makes them diverge, and
zero curvature corresponds to flat geometry.

We fix the following additional notation.  The second fundamental form of
$M\hookrightarrow\bR^D$ is denoted~$A$, which measures how $M$ curves
within~$\bR^D$ and vanishes when $M$ is an affine subspace.
The orthogonal projection onto $T_yM$ is $\Pi_{T_yM}$.  For $I=[0,1]$ and integer $s\ge 0$,
$H^s(I,\bR^D)=W^{s,2}(I,\bR^D)$ is the $L^2$-based Sobolev space of
functions possessing $s$ weak derivatives, with norm
$\norm{F}_{H^s}^2=\sum_{\abs\alpha\le s}\norm{D^\alpha F}_{L^2}^2$ and
seminorm $\abs{F}_{H^s}^2=\sum_{\abs\alpha=s}\norm{D^\alpha F}_{L^2}^2$.
Here $\alpha$ is a multi-index, $\abs{\alpha}$ is its order, and
$D^\alpha$ denotes the corresponding weak partial derivative; on the
one-dimensional domain~$I$, a multi-index reduces to a single
non-negative integer~$k$, and $D^\alpha F=F^{(k)}$ is the $k$-th
weak derivative.
We write $a_n\lesssim b_n$ if $a_n\le Cb_n$ for a constant~$C$ independent
of~$n$, and $a_n\asymp b_n$ if $a_n\lesssim b_n\lesssim a_n$.

\subsection{Riemannian geometry}\label{subsec:riem}

The \emph{convexity radius} of~$M$ is
\[
  \rho_M \;=\; \min\!\Bigl\{\tfrac{1}{2}\inj(M),\;
  \tfrac{\pi}{2\sqrt{\kappa_+}}\Bigr\},
\]
where $\kappa_+=\max(\sup\Sec_M,0)$ is the positive part of the
sectional curvature upper bound and $\pi/(2\sqrt{0})=+\infty$.
Within a ball of radius~$\rho_M$, the squared distance
$d_M^2$ is jointly smooth and convex, so the Fr\'echet mean is
unique and the manifold behaves much like Euclidean space
\citep{Karcher1977, Afsari2011}.

By the Nash embedding theorem \citep{Nash1956}, we fix a smooth isometric
embedding $\iota\colon M\hookrightarrow\bR^D$ and identify $M$ with
$\iota(M)$.  For a smooth curve $F\colon I\to M\subset\bR^D$, the
intrinsic Dirichlet energy coincides with the extrinsic one:
\begin{equation}\label{eq:energy-extrinsic}
  \Dir(F) \;=\;\frac{1}{2}\int_0^1|F'(t)|_g^2\,dt
  \;=\;\frac{1}{2}\int_0^1|F'(t)|_{\bR^D}^2\,dt,
\end{equation}
since the embedding is isometric.

A \emph{geodesic} in~$M$ is a constant-speed curve of zero
intrinsic acceleration, the manifold analogue of a straight line.
A smooth curve $F\colon I\to M$ is called a \emph{harmonic map} if it
is a critical point of the Dirichlet energy, i.e., it satisfies the
Euler--Lagrange equation $\tau(F)=0$, where the \emph{tension field}
$\tau(F)=F'' + A(F)(F',F')$ measures the failure of~$F$ to be a
geodesic \citep{EellsSampson1964}.  Here $F''$ is the ambient
acceleration in~$\bR^D$, and $A(y)(u,v)\in(T_yM)^\perp$ is the
second fundamental form evaluated at $y\in M$ on tangent
vectors $u,v\in T_yM$; the term $A(F)(F',F')$ is the normal
component of the acceleration, so $\tau(F)=0$ says that all
acceleration is absorbed by the curvature of~$M$.
In one dimension, $\tau(F)=0$ if and only if $F$ is a geodesic, so
harmonic maps from an interval into~$M$ are precisely geodesics.
When $M=\bR^D$, $A$ vanishes and this reduces to $F''=0$.  The existence and
regularity theory of harmonic maps, developed by
\citet{EellsSampson1964}, \citet{Schoen1984}, and
\citet{Helein2002}, plays a central role in the structural analysis of
our estimator (Sections~\ref{sec:regularity}--\ref{sec:topology}).

The next lemma is a specialization of
variance-convexity results in NPC spaces
\citep[Proposition~4.4]{Sturm2003} and local Hessian-comparison
results under positive curvature
\citep[Theorem~1.2]{Karcher1977}.
We state it in the form needed for the statistical analysis.

\begin{lemma}[Curvature-dependent variance inequality]
\label{lem:variance-ineq}
  Let $\mu$ be a probability measure on~$M$ with Fr\'echet mean~$\bar y$,
  supported in $B_M(\bar y,r_0)$ with
  $r_0<\rho_M$.  Then for
  every $y$ strictly within the open ball $B_M(\bar y,\rho_M)$,
  \begin{equation}\label{eq:variance-decomp}
    \int_M d_M^2(z,y)\dd\mu(z)
    \;\ge\; \eta(\kappa_+,r_0)\;d_M^2(\bar y,y)
    \;+\;\int_M d_M^2(z,\bar y)\dd\mu(z),
  \end{equation}
  where the \emph{curvature factor} is
  \begin{equation}\label{eq:curvature-factor}
    \eta(\kappa,r) \;=\;
    \begin{cases}
      \dfrac{\sqrt\kappa\, r}{\tan(\sqrt\kappa\, r)}
      & \text{if }\kappa>0,\\[6pt]
      1 & \text{if }\kappa\le 0.
    \end{cases}
  \end{equation}
  For $\kappa>0$, $\eta(\kappa,r)\in(0,1)$ and
  $\eta(\kappa,r)=1-\frac{\kappa r^2}{3}+O(\kappa^2 r^4)$ as $\kappa r^2\to 0$.
\end{lemma}

\begin{proof}
  See Appendix~\ref{supp:lemma21}.
\end{proof}

\subsection{Sobolev spaces of manifold-valued maps}\label{subsec:sobolev}

The Nash embedding allows us to define Sobolev spaces of curves in $M$.
For integer $s\ge 1$, define
\begin{equation}\label{eq:sobolev-def}
  \cH^s(I,M) \;=\;
  \bigl\{F\in H^s(I,\bR^D):\; F(t)\in M\;\text{for all }t\in I\bigr\}.
\end{equation}
For integer $s\ge 1$, the Sobolev embedding $H^s(I,\bR^D)\hookrightarrow C(\bar I,\bR^D)$ ensures that every $H^s$ curve has a continuous representative and the pointwise constraint $F(t)\in M$ is meaningful.  The \emph{$s$-th order energy} is
$J_s(F)=\abs{F}_{H^s}^2=\sum_{\abs\alpha=s}\norm{D^\alpha F}_{L^2}^2$.
When $s=1$, $J_1(F)=\norm{F'}_{L^2}^2=2\Dir(F)$, where
$\Dir(F)$ is the Dirichlet energy.

\begin{proposition}[Compactness (Sobolev-Rellich)]
\label{prop:compactness}
  For any $E_0>0$, $\cF_{E_0}=\{F\in\cH^s(I,M):J_s(F)\le E_0\}$ is
  compact in $C(\bar I,\bR^D)$ and sequentially weakly compact in
  $H^s(I,\bR^D)$.
\end{proposition}

\begin{proof}
  See Appendix~\ref{supp:prop22}.
\end{proof}

Let $(t,Y)\sim P$ on $I\times M$.  The \emph{conditional Fr\'echet
mean}, \emph{population risk}, and \emph{conditional variance} are
\begin{gather*}
  m(t)=\argmin_{y\in M}\bE\bigl[d_M^2(Y,y)\mid t\bigr],\qquad
  R(F)=\bE\bigl[d_M^2(Y,F(t))\bigr],\\
  \sigma^2(t)=\bE\bigl[d_M^2(Y,m(t))\mid t\bigr].
\end{gather*}

\section{The harmonic map regression model}\label{sec:model}

We now state the formal assumptions and define the estimator.

\subsection{Assumptions}\label{subsec:assumptions}

The predictor domain is a compact interval $I\subset\bR$. Without loss
of generality we take $I=[0,1]$ throughout.  The target $(M,g)$ is a
smooth, compact, connected Riemannian manifold without boundary, with
$\dim M=d$.  Write
$\sigma_0^2=\sup_{t\in I}\sigma^2(t)$ for the supremum of the
conditional variance. On a compact target with bounded noise support,
$\sigma_0^2\le r_0^2<\infty$.
The following conditions are imposed on the model.

\begin{enumerate}[label=\textup{(A\arabic*)},ref=\textup{A\arabic*},leftmargin=*,nosep]
  \item\label{ass:curvature}
    \emph{Curvature bound.}\;
    $\abs{\Sec_M}\le\kappa$ for some $\kappa\ge 0$.
    Write $\kappa_+=\max(\sup\Sec_M,0)\le\kappa$ for the positive part;
    $\kappa_+$ enters the curvature
    factor~$\eta(\kappa_+,r_0)$ defined in \eqref{eq:curvature-factor}.
  \item\label{ass:density}
    \emph{Design density.}\;
    The covariate $t$ has Lebesgue density $f$ with $0<c_f\le f\le C_f<\infty$ on $I$.
  \item\label{ass:concentration}
    \emph{Concentration.}\;
    There exists $0<r_0<\rho_M$ such that $Y\mid t$ is supported in
    $B_M(m(t),r_0)$ for every $t\in I$.
  \item\label{ass:smoothness}
    \emph{Sobolev regularity.}\;
    $m\in\cH^s(I,M)$ for an integer $s\ge 1$, with $J_s(m)\le B$.
\end{enumerate}

\begin{remark}[On (\ref{ass:concentration})]\label{rem:NPC}
  Since $\rho_M \le \frac{1}{2}\inj(M)$, (\ref{ass:concentration}) places the data within
    the injectivity radius, ensuring smoothness of the Riemannian
    $\log$ map. When $\Sec_M\le 0$ (e.g., Hadamard manifolds, SPD matrices),
  (\ref{ass:concentration}) can be replaced by
  $\bE[d_M^2(Y,y_0)\mid t]<\infty$ for some $y_0\in M$, and the
  central results hold with $\eta=1$.
  For positively curved targets, the bounded-support condition can be
  relaxed to a sub-Gaussian tail condition via a sample-splitting
  truncation procedure that preserves the minimax rate exponent.
  The estimation constant acquires a multiplicative correction
  $1/\eta(\kappa_+, 3r_n/2)=1+O(\kappa_+\sigma^2\log n)$,
  where $r_n=\sigma\sqrt{2c\log n}$ is the effective truncation radius.
  This correction grows with~$n$ but does not affect the rate exponent
  $n^{-2s/(2s+1)}$. See Appendix~\ref{supp:subgauss} for the precise statement.

  The effective restriction depends on the target geometry.  For
  $S^2$ ($\kappa_+=1$), $\rho_M=\pi/2\approx 1.57$, so the noise
  radius must satisfy $r_0<\pi/2$.  For $\mathrm{SO}(3)$ with the
  bi-invariant metric ($\kappa_+=1$, $\inj=\pi$), $\rho_M=\pi/2$.
  For flat tori $T^d=\bR^d/\Lambda$, $\kappa_+=0$ and
  $\rho_M=\frac{1}{2}\min_i L_i$ where $L_i$ are the side lengths.
  For NPC targets ($\mathrm{Sym}^+(k)$, $\bH^d$), $\eta=1$
  unconditionally and any finite $r_0$ suffices.
\end{remark}

\subsection{Estimator}\label{subsec:estimator}

With the model and assumptions in place, we can state the estimator.
Given observations $(t_i,Y_i)_{i=1}^n$ with $t_i\in I$ and $Y_i\in M$,
define the \emph{empirical risk}
$R_n(F)=n^{-1}\sum_{i=1}^n d_M^2(Y_i,F(t_i))$, which measures the average
squared geodesic distance between the observations and the fitted values.
The \emph{harmonic map regression estimator} minimizes the sum of data fit
and a regularity penalty:
\begin{equation}\label{eq:estimator}
  \hat F_n \;=\; \argmin_{F\in\cH^s(I,M)}\;
  \bigl\{R_n(F) + \lambda_n\, J_s(F)\bigr\}.
\end{equation}
The first term pulls $\hat F_n$ toward the data, while the second,
$\lambda_n J_s(F)$, controls the Sobolev regularity of the map.  The
parameter $\lambda_n>0$ balances the two objectives.  Large $\lambda_n$
produces a smooth but potentially biased fit, while small $\lambda_n$ tracks
the data closely at the expense of roughness.  The optimal choice
$\lambda_n\asymp n^{-2s/(2s+1)}$, determined by the oracle inequality in
\Cref{sec:main}, balances the bias-variance trade-off at the minimax rate.

The optimization is carried out over the manifold-valued Sobolev space
$\cH^s(I,M)$ rather than a space of continuous or measurable functions.
When $s\ge 1$, the Sobolev embedding guarantees that elements of
$\cH^s(I,M)$ are continuous, so the fidelity term $R_n(F)$ is
well defined.
As in the Euclidean case, no boundary conditions are imposed on $\hat F_n$, meaning that
the minimization is over all of~$\cH^s(I,M)$.  The estimator therefore
satisfies natural (Neumann) boundary conditions automatically,
with $\hat F_n'=0$ at the domain endpoints
(\Cref{thm:regularity}\ref{item:reg-neumann}).

\section{Oracle inequality and minimax optimality}
\label{sec:main}

This section establishes existence, a curvature-dependent oracle
inequality, and minimax optimality of the proposed estimator for any integer Sobolev order
$s\ge 1$.

\subsection{Existence}\label{subsec:existence}

The first question is whether the minimum in~\eqref{eq:estimator} is
attained.  Existence is not automatic because the Sobolev space $\cH^s(I,M)$ is
not a linear space, so the classical theory for quadratic functionals on
Hilbert spaces does not apply directly.  A minimizing sequence could, in
principle, develop oscillations that converge weakly to a map leaving~$M$.
The key ingredients are the compactness of~$M$ (which keeps values bounded),
the Rellich--Kondrachov theorem
\citep[Theorem~6.3]{Adams2003} (which, for $s\ge 1$, upgrades weak
Sobolev convergence to uniform convergence), and the weak lower
semicontinuity of the energy $J_s$ on $H^s(I,\bR^D)$
\citep{KorevaarSchoen1993}.  Together, these
guarantee that the limit of any minimizing sequence lies in $\cH^s(I,M)$
and achieves the minimum.

\begin{theorem}[Existence]\label{thm:existence}
  Under (\ref{ass:curvature}) with $s\ge 1$:
  \begin{enumerate}[label=(\alph*)]
    \item\label{item:exist-emp}
      For any $\lambda>0$ and data $(t_i,Y_i)_{i=1}^n$, a minimizer $\hat F_n$
      of $R_n+\lambda J_s$ over $\cH^s(I,M)$ exists.
    \item\label{item:exist-pop}
      Under (\ref{ass:density})--(\ref{ass:concentration}), for any $\lambda>0$,
      a minimizer of $R+\lambda J_s$ over $\cH^s(I,M)$ exists.
  \end{enumerate}
\end{theorem}

\begin{proof}
  See Appendix~\ref{supp:thm41}.
\end{proof}

The theorem asserts existence only.  This restriction is essential.  Even
when $\Sec_M\le 0$, a compact target need not be simply connected, and the
quotient geometry can prevent global strict convexity of the squared
distance.  Flat tori provide the basic example, where the distance on the quotient
is obtained by minimizing over deck-transformation branches on the universal
cover, and the resulting Fr\'echet functional can have more than one global
minimizer.  For the statistical results developed below, existence together
with measurable selection is sufficient.

\begin{remark}[Measurable selection]\label{rem:measurable-selection}
  When the minimizer is not unique, the symbol
  $\hat F_n$ denotes any element of the (nonempty, compact)
  minimizer set.  A measurable selection exists by the
  Kuratowski--Ryll-Nardzewski theorem, since the objective
  $(t_1,Y_1,\ldots,t_n,Y_n,F)\mapsto R_n(F)+\lambda_n J_s(F)$ is
  jointly measurable (as a composition of the continuous squared-distance
  function and point evaluation, which is continuous on $\cH^s(I,M)$ by
  the Sobolev embedding), the minimizer set
  $\Gamma(\omega)=\argmin_{F\in\cH^s(I,M)}\{R_n(F)+\lambda_n J_s(F)\}$
  is nonempty (\Cref{thm:existence}) and closed-valued (by lower
  semicontinuity of $J_s$ and continuity of $R_n$ on the compact sublevel
  sets), hence admits a universally measurable selection.
  The oracle inequality (\Cref{thm:oracle}) holds
  for every minimizer, so all statistical conclusions (convergence rates,
  excess risk bounds, and the topological phase
  transition) are independent of which minimizer is selected.
\end{remark}

\subsection{Curvature-dependent excess risk bound}

The oracle inequality in \Cref{subsec:oracle-ineq} below controls the
excess risk $\cE(\hat F_n)=R(\hat F_n)-R(m)$.  Converting this to the
estimation distance $\bE[d_M^2(\hat F_n,m)]$ requires a comparison
between $\cE(F)$ and $\bE[d_M^2(F,m)]$.  In Euclidean spaces the two
coincide by the bias-variance decomposition. On curved manifolds, the
comparison acquires a curvature-dependent factor.

\begin{lemma}[Excess risk]\label{lem:excess-risk}
  Under (\ref{ass:curvature}) and (\ref{ass:concentration}), define
  $\cE(F)=R(F)-R(m)$.  For every measurable $F\colon I\to M$ with
  $d_M(F(t),m(t))<\rho_M$ for all~$t$,
  \begin{equation}\label{eq:excess-lower}
    \cE(F) \;\ge\;
    \eta(\kappa_+,r_0)\;\bE\bigl[d_M^2(F(t),m(t))\bigr],
  \end{equation}
  where $\eta(\kappa_+,r_0)$ is the curvature factor \eqref{eq:curvature-factor}.
\end{lemma}

The factor $\eta(\kappa_+,r_0)<1$ when $\kappa_+>0$ means that
positive curvature degrades the conversion.  For non-positively curved
targets, the next result shows that this degradation disappears
entirely, and the excess risk decomposes into estimation distance plus a
non-negative remainder, exactly as in the Euclidean case.

\begin{theorem}[Sharp excess risk identity for NPC targets]\label{thm:NPC-sharp}
  Suppose $\Sec_M\le 0$ everywhere (non-positively curved target).  Then
  the excess risk satisfies the identity
  \begin{equation}\label{eq:NPC-identity}
    \cE(F)
    \;=\;\bE\bigl[d_M^2(F(t),m(t))\bigr]
    \;+\;h(F),
  \end{equation}
  where $h(F)\ge 0$ is a non-negative remainder that depends on the curvature
  and the noise distribution but not on the estimation error:
  \[
    h(F) \;=\;\bE\Bigl[d_M^2(Y,F(t))-d_M^2(Y,m(t))
    -d_M^2(F(t),m(t))\Bigr] \;\ge\; 0.
  \]
  In particular, $\cE(F)\ge\bE[d_M^2(F(t),m(t))]$, which is
  \eqref{eq:excess-lower} with $\eta=1$.  For the oracle inequality, this
  gives, with $\lambda_n\asymp n^{-2s/(2s+1)}$,
  \begin{equation}\label{eq:NPC-rate}
    \bE\biggl[\int_I d_M^2(\hat F_n,m)\,f\,dt\biggr]
    \;\le\; C\,n^{-\frac{2s}{2s+1}},
  \end{equation}
  with no curvature degradation in the excess risk lower bound (i.e., $\eta=1$), though the absolute constant $C$ retains a mild dependence on the manifold through the metric entropy and Lipschitz properties of the squared distance function.
\end{theorem}

\begin{proof}
  See Appendix~\ref{supp:thm44}.
\end{proof}

For non-positively curved targets (e.g., Hadamard manifolds, SPD
matrices), the estimation constant is exactly that of the Euclidean
theory, with no curvature penalty.  The minimax lower bound
(\Cref{thm:minimax}) matches with the same exponent and a
curvature-independent constant.

\subsection{Oracle inequality}\label{subsec:oracle-ineq}

The excess risk comparison of \Cref{lem:excess-risk} converts
estimation distance into excess risk with a curvature-dependent
constant.  The next result controls $\bE[\cE(\hat F_n)]$ itself
via an oracle inequality.  The proof combines a basic inequality,
metric entropy bounds via the Birman--Solomjak theorem
\citep[Theorem~1]{Birman1967}, and a dyadic peeling argument.

\begin{theorem}[Oracle inequality and convergence rate]\label{thm:oracle}
  Under (\ref{ass:curvature})--(\ref{ass:smoothness}) with $s\ge 1$, there exist
  constants $C_1,C_2>0$ depending on $s,M,c_f,C_f,\sigma_0^2$ such that
  for every $\lambda_n>0$ and $n\ge 1$:
  \begin{enumerate}[label=(\alph*)]
    \item\label{item:oracle-excess}
      \textbf{(Excess risk.)} Unconditionally,
      \begin{equation}\label{eq:oracle}
        \bE\bigl[\cE(\hat F_n)\bigr]
        \;\le\; C_1\inf_{F\in\cH^s(I,M)}
        \bigl\{\cE(F)+\lambda_n J_s(F)\bigr\}
        \;+\;\frac{C_2}{n\,\lambda_n^{1/(2s)}}.
      \end{equation}
    \item\label{item:oracle-estimation}
      \textbf{(Estimation.)}
      The estimation bound
      \begin{equation}\label{eq:oracle-estimation}
        \bE\biggl[\int_I d_M^2(\hat F_n,m)\,f\,dt\biggr]
        \;\le\;\frac{1}{\eta(\kappa_+,r_0)}\biggl(
        C_1\inf_{F}\bigl\{\cE(F)+\lambda_n J_s(F)\bigr\}
        +\frac{C_2}{n\lambda_n^{1/(2s)}}\biggr)
      \end{equation}
      holds unconditionally when $\Sec_M\le 0$ (with $\eta=1$,
      by \Cref{thm:NPC-sharp}), and for $n\ge n_1$ when
      $\Sec_M>0$, where $n_1=n_1(s,M,B)$ is a finite threshold
      (see \Cref{rem:locality}).
    \item\label{item:oracle-rate}
      \textbf{(Convergence rate.)}
      With $\lambda_n\asymp n^{-2s/(2s+1)}$,
      \begin{equation}\label{eq:rate}
        \bE\biggl[\int_I d_M^2(\hat F_n,m)\,f\,dt\biggr]
        \;\le\;\frac{C}{\eta(\kappa_+,r_0)}\;n^{-\frac{2s}{2s+1}}.
      \end{equation}
  \end{enumerate}
\end{theorem}

\begin{proof}
  See Appendix~\ref{supp:thm46}.
\end{proof}

\begin{remark}[Locality condition for positively curved targets]
\label{rem:locality}
  Part~\ref{item:oracle-estimation} converts excess risk to
  estimation distance via \Cref{lem:excess-risk}, which requires
  $d_M(\hat F_n(t),m(t))<\rho_M$ for all~$t$.
  When $\Sec_M\le 0$ this holds vacuously.
  For $\Sec_M>0$, the basic inequality
  $\lambda_n J_s(\hat F_n)\le\diam(M)^2+\lambda_n B$
  bounds the Sobolev norm, and the embedding
  $H^s\hookrightarrow C^{0,s-1/2}$ yields equicontinuity.
  Combined with $L^2$-convergence from \eqref{eq:oracle}, an
  Arzel\`a--Ascoli argument gives
  $\sup_t d_M(\hat F_n(t),m(t))\to 0$ in probability
  (Appendix~\ref{supp:cor49}), so a finite threshold~$n_1$ exists.
\end{remark}

The oracle inequality controls the expected excess risk.
Sharper tail control is also available.  The next result shows that
$\cE(\hat F_n)$ concentrates around its rate with sub-exponential
tails, which is needed for the $L^\infty$-consistency used in
\Cref{sec:regularity}.

\begin{corollary}[High-probability oracle inequality]
\label{cor:hp-oracle}
  Under the conditions of \Cref{thm:oracle} with
  $\lambda_n\asymp n^{-2s/(2s+1)}$, there exist constants
  $C,c>0$ such that for every $t\ge 1$ and $n\ge 1$,
  \begin{equation}\label{eq:hp-oracle}
    \bP\Bigl(\cE(\hat F_n)
    \;\ge\;C\,t\,n^{-\frac{2s}{2s+1}}\Bigr)
    \;\le\;\exp\bigl(-c\,t\,n^{\frac{1}{2s+1}}\bigr).
  \end{equation}
  This gives sub-exponential concentration
  $\cE(\hat F_n)=O_P(n^{-2s/(2s+1)})$ and, via \Cref{rem:locality},
  uniform consistency
  $\sup_t d_M(\hat F_n(t),m(t))\to 0$ in probability.
\end{corollary}

\begin{proof}
  See Appendix~\ref{supp:cor49}.
\end{proof}

\subsection{Minimax optimality}

The oracle inequality establishes that the harmonic map regression estimator
converges at rate $n^{-2s/(2s+1)}$.  The next result shows that no estimator
can do better uniformly over the Sobolev ball, confirming that this rate is
optimal in the minimax sense.

\begin{theorem}[Minimax lower bound]\label{thm:minimax}
  Under (\ref{ass:curvature}) and (\ref{ass:density}) with $\dim M\ge 1$ and
  $s\ge 1$, there exist $c,\sigma_1>0$ such that
  \begin{equation}\label{eq:minimax}
    \inf_{\tilde F_n}\sup_{m\in\cF_B}
    \bE\biggl[\int_I d_M^2(\tilde F_n,m)\,f\,dt\biggr]
    \;\ge\; c\,n^{-\frac{2s}{2s+1}},
  \end{equation}
  where $\cF_B=\{F\in\cH^s(I,M):J_s(F)\le B\}$, using the same notation
  for the Sobolev ball as in the oracle inequality.
\end{theorem}

\begin{proof}
  See Appendix~\ref{supp:thm411}.
\end{proof}

The proof follows the Assouad scheme
\citep[Theorem~2.12]{Tsybakov2009} by embedding a hypercube of bump
functions on~$I$ into~$M$ via the exponential map, reduces the testing
problem to pairwise Kullback--Leibler divergence control, and invokes
Assouad's lemma to obtain the lower bound.  The rate $n^{-2s/(2s+1)}$
coincides with the classical Stone-Wahba rate for Euclidean smoothing splines
\citep{Stone1982, Wahba1990}, confirming that the Riemannian geometry of~$M$
does not affect the minimax rate. However, the
lower bound constant~$c$ in~\eqref{eq:minimax} is curvature-free because the bump
amplitudes shrink with~$n$, making curvature corrections negligible.  Whether
the minimax constant depends on curvature for positively curved targets
remains an open problem, discussed further in \Cref{sec:discussion}.

\section{The harmonic map equation and regularity}\label{sec:regularity}

The oracle inequality of the previous section controls the
integrated risk $\bE[\cE(\hat F_n)]$.  This section refines that
bound to pointwise control by exploiting the PDE structure available
at $s=1$, paralleling the Euclidean smoothing-spline theory.
The Euler--Lagrange equation (\Cref{thm:EL}) identifies the estimator
as a solution of the perturbed harmonic map equation.
Elliptic regularity (\Cref{thm:regularity}) then shows that
the estimator is a \emph{geodesic spline}, piecewise geodesic with
prescribed derivative jumps at the design points.
Finally, an equivalent-kernel argument
(\Cref{thm:pointwise}) delivers a pointwise risk bound of order
$n^{-2/3}$ that separates the local noise $\sigma^2(t)$ from the
design density $f(t)$, with a non-perturbative proof for NPC targets
whose variance constant is at most the Euclidean constant.

\subsection{The Euler-Lagrange equation}

The variational identities in this subsection are standard in harmonic-map
theory. We adapt them to the penalized empirical objective
\citep[Chapter~8]{Jost2017}.

\begin{theorem}[Harmonic map equation for the estimator]\label{thm:EL}
  Let $s=1$, and let
  $\hat F_n\in\cH^1(I,M)$ be the empirical minimizer of $R_n+\lambda_n J_1$.
  Suppose that
  \begin{equation}\label{eq:log-condition}
    d_M\bigl(\hat F_n(t_i),\,Y_i\bigr)\;<\;\inj(M)
    \qquad\text{for all }i=1,\ldots,n.
  \end{equation}
  Then $\hat F_n$ satisfies the \emph{perturbed harmonic map equation} in the
  distributional sense:
  \begin{equation}\label{eq:EL}
    \hat F_n''(t) + A\bigl(\hat F_n(t)\bigr)
    \bigl(\hat F_n'(t),\hat F_n'(t)\bigr)
    \;=\;
    -\frac{1}{n\lambda_n}\sum_{i=1}^n\delta_{t_i}(t)\;
    \log_{\hat F_n(t_i)}(Y_i),
  \end{equation}
  where $A(y)\colon T_yM\times T_yM\to (T_yM)^\perp$ is the second
  fundamental form of $M\hookrightarrow\bR^D$ at~$y$, and
  $\log_y(q)\in T_yM$ denotes the Riemannian logarithmic map (the
  initial velocity of the minimizing geodesic from $y$ to~$q$, defined
  whenever $d_M(y,q)<\inj(M)$).
\end{theorem}

\begin{remark}[Validity of condition~\eqref{eq:log-condition}]
\label{rem:log-condition}
  Condition~\eqref{eq:log-condition} ensures that the logarithmic map
  on the right-hand side of~\eqref{eq:EL} is well-defined and smooth.
  It is satisfied automatically when $\diam(M)\le\inj(M)$, which
  holds for all simply connected compact manifolds of non-negative
  curvature (e.g., $S^d$) and more generally for compact Lie groups
  with bi-invariant metrics (e.g., $\mathrm{SO}(3)$, which is not
  simply connected but satisfies $\diam(\mathrm{SO}(3))=\inj(\mathrm{SO}(3))=\pi/\sqrt{2}$).  For quotient manifolds where
  $\diam(M)>\inj(M)$ (e.g., flat tori, lens spaces),
  \eqref{eq:log-condition} must be verified.
  Under (\ref{ass:concentration}), we have
  $d_M(\hat F_n(t_i),Y_i)\le d_M(\hat F_n(t_i),m(t_i))+r_0$,
  so \eqref{eq:log-condition} holds whenever
  $\sup_t d_M(\hat F_n(t),m(t))<\inj(M)-r_0\ge\inj(M)/2$.
  By the $L^\infty$-consistency established in
  \Cref{cor:hp-oracle}, this is satisfied for all $n\ge n_2$
  with a finite threshold~$n_2$ depending on
  $s,M,I,\sigma_0^2,B$.
  In particular, the Euler-Lagrange equation holds for all
  sufficiently large~$n$ under the standing assumptions.
\end{remark}

\begin{proof}
  See Appendix~\ref{supp:thm51}.
\end{proof}

Equation \eqref{eq:EL} is a second-order ODE
with a measure-valued right-hand side supported on the design points.
The left-hand side, $\tau(F)=F''+A(F)(F',F')$, is the
\emph{tension field} of $F$, whose vanishing characterizes harmonic maps
\citep{EellsSampson1964}.  The estimator is therefore a harmonic map
perturbed by data-driven point sources.

\subsection{Regularity of the estimator}

The Euler--Lagrange equation identifies $\hat F_n$ as a solution of a
second-order ODE with point-source forcing.  Classical regularity theory
for harmonic maps then determines the smoothness of $\hat F_n$
between the design points and the form of the singularities at the
design points.  The key conclusion is that $\hat F_n$ is a
\emph{geodesic spline}, piecewise geodesic with explicit derivative
jumps.

\begin{theorem}[Derivative jump identity]\label{thm:regularity}
  Under the conditions of \Cref{thm:EL}, assume
  $(M,g)$ is a $C^k$ Riemannian manifold with $k\ge 2$.
  \begin{enumerate}[label=(\alph*)]
    \item\label{item:reg-smooth}
      $\hat F_n\in C^k(I\setminus\{t_1,\ldots,t_n\},M)$, i.e., the estimator
      inherits the regularity of~$M$ on each open interval between
      consecutive design points.  In particular, if $M$ is
      $C^\infty$ (resp.\ real analytic), then $\hat F_n$ is $C^\infty$
      (resp.\ $C^\omega$) on each such interval.
    \item\label{item:reg-jump}
      At each design point $t_i$, the estimator is continuous but has a jump
      discontinuity in the first derivative.  The first derivative is bounded
      ($\|\hat F_n'\|_{L^\infty} < \infty$, see proof below),
      and the jump satisfies:
      \begin{equation}\label{eq:jump}
        \hat F_n'(t_i^+)-\hat F_n'(t_i^-)
        \;=\;-\frac{1}{n\lambda_n}\;
        \log_{\hat F_n(t_i)}(Y_i).
      \end{equation}
      The one-sided derivatives $\hat F_n'(t_i^\pm)$ are evaluated in the
      ambient space~$\bR^D$ as one-sided limits, and their difference lies in
      $T_{\hat F_n(t_i)}M$ since $\log_y(q)\in T_yM$.
    \item\label{item:reg-spline}
      On each interval $(t_{(i)},t_{(i+1)})$ (where $t_{(1)}<\cdots<t_{(n)}$
      are the order statistics), $\hat F_n$ is a harmonic map from the one-dimensional interval
      to~$M$, which implies it is a constant-speed geodesic.  For $M=\bR^d$, this reduces to an affine function (i.e.,
      $\hat F_n$ is a piecewise linear spline); for $M=S^d$, it is an arc of a
      great circle traversed at constant speed.
    \item\label{item:reg-neumann}
      The estimator satisfies the natural (Neumann) boundary conditions
      $\hat F_n'(0^+)=0$ and $\hat F_n'(1^-)=0$, so the velocity vanishes
      at the domain endpoints.
  \end{enumerate}
\end{theorem}

\begin{proof}
  See Appendix~\ref{supp:thm54}.
\end{proof}

Part~\ref{item:reg-spline} is a Riemannian analogue of the
classical result that the Euclidean first-order smoothing spline is
piecewise linear.  The estimator is a geodesic spline,
piecewise geodesic with prescribed derivative jumps \eqref{eq:jump}
at the design points.
The $C^0$ junction conditions follow from the second-order
Euler--Lagrange equation with point-source forcing, matching the
regularity of the Sobolev class
$H^1(I,M)\hookrightarrow C^{0,1/2}(I,M)$.
Higher-order covariant-derivative penalties
\citep{Kim2021spline} produce smoother junctions ($C^1$ or $C^2$)
at the cost of more complex PDE structure and the obstructions
analyzed in \Cref{sec:obstructions}.

\subsection{Pointwise risk via the equivalent kernel}\label{subsec:pointwise}

The geodesic spline structure (\Cref{thm:regularity}) enables a pointwise
risk analysis that goes beyond the integrated risk bounds of \Cref{sec:main}
and requires the harmonic map ODE.  This is possible because
the explicit solution of the harmonic map equation on each inter-design-point
interval provides a representation of the estimator in terms of a Green's
function.

\begin{theorem}[Pointwise risk bound]\label{thm:pointwise}
  Let $s=1$.  Suppose
  (\ref{ass:curvature})--(\ref{ass:smoothness}) hold with $s=1$, and the
  regression function $m$ is Lipschitz with
  $\sup_t\abs{m'(t)}_g\le L_m$.
  For all sufficiently large~$n$ and every interior point
  $t\in(\delta_n,1-\delta_n)$ with $\delta_n=C_0\lambda_n^{1/2}\log n$,
  the following hold.
  \begin{enumerate}[label=(\alph*)]
    \item\label{item:pointwise-npc}
      \textbf{(NPC targets.)}  If\/ $\Sec_M\le 0$, then
      \begin{equation}\label{eq:pointwise-npc}
        \bE\bigl[d_M^2(\hat F_n(t),\,m(t))\bigr]
        \;\le\;
        C_1\,\lambda_n\,L_m^2
        \;+\;
        \frac{C_2\,\sigma^2(t)}{n\,\lambda_n^{1/2}\,f(t)},
      \end{equation}
      where $C_1,C_2>0$ satisfy $C_1\le C_1^{\bR^d}$ and
      $C_2\le C_2^{\bR^d}$, the corresponding constants for the
      Euclidean smoothing spline on~$\bR^d$.  With
      $\lambda_n\asymp n^{-2/3}$, this gives
      $\bE[d_M^2(\hat F_n(t),m(t))]
      =O\bigl((L_m^2+\sigma^2(t)/f(t))\,n^{-2/3}\bigr)$.

    \item\label{item:pointwise-general}
      \textbf{(General targets.)}  Suppose additionally that the Nash
      embedding $\iota\colon M\hookrightarrow\bR^D$ is $C^3$.  Then
      \begin{equation}\label{eq:pointwise}
        \bE\bigl[d_M^2(\hat F_n(t),\,m(t))\bigr]
        \;\le\;
        \frac{C_1\,\lambda_n\,L_m^2}{\eta(\kappa_+,r_0)}
        \;+\;
        \frac{C_2\,\sigma^2(t)}{n\,\lambda_n^{1/2}\,f(t)\,\eta(\kappa_+,r_0)},
      \end{equation}
      where $C_1,C_2>0$ depend only on~$M$.  With
      $\lambda_n\asymp n^{-2/3}$, this yields
      \begin{equation}\label{eq:pointwise-rate}
        \bE\bigl[d_M^2(\hat F_n(t),m(t))\bigr]
        \;=\;O\!\left(\frac{L_m^2+\sigma^2(t)/f(t)}
        {\eta(\kappa_+,r_0)}\;n^{-2/3}\right).
      \end{equation}
  \end{enumerate}
\end{theorem}

\begin{proof}
  See Appendix~\ref{supp:thm56}.
\end{proof}

\begin{remark}[Proof methods and regularity conditions]\label{rem:pointwise-methodology}
  Part~\ref{item:pointwise-npc} is proved by an intrinsic argument
  that works in a parallel transport frame
  $\{e_a^{(t)}\}$ along the regression function~$m$.
  The proof directly compares the quadratic form $Q^M$ of the
  penalized objective with its Euclidean counterpart~$Q^0$,
  exploiting two features of non-positive curvature.
  First, the index form of the Dirichlet energy gains
  a non-negative curvature term $-\langle R(V,m')m',V\rangle\ge 0$.
  Second, the Hessian of the squared distance satisfies
  $\operatorname{Hess}_p(d_M^2(\cdot,y)/2)\ge g_p$ by
  Toponogov comparison \citep[Section~6.5]{Jost2017}.
  Together these yield
  $Q^M\ge Q^0$ and, after discretization via
  \Cref{prop:exact-reduction}, the Green matrix comparison
  $\mathbf{G}^M\preceq\mathbf{G}^0$ in Loewner order.
  The variance constant of the NPC estimator is therefore
  at most the Euclidean constant, with strict inequality
  when the sectional curvature is strictly negative.
  The nonlinear remainder is controlled by geodesic
  convexity of the discrete objective on NPC targets
  \citep{Zhang2016riemannian}, and no Nash embedding or
  normal-coordinate expansion is needed.

  On positively curved targets, the curvature matrix
  $\mathbf{K}(t)$ can have negative eigenvalues, and the
  Hessian bound~$\operatorname{Hess}_p(d_M^2(\cdot,y)/2)\ge g_p$
  fails, so the comparison $Q^M\ge Q^0$ is no longer available.
  Part~\ref{item:pointwise-general} therefore takes a
  different route, approximating the Riemannian objective by
  the Euclidean smoothing spline in normal coordinates and
  treating curvature as a perturbation.
  The proof verifies three intermediate conditions:
  (P1)~local chart control (via $L^\infty$-consistency, \Cref{cor:hp-oracle}),
  (P2)~Fr\'echet differentiability of the curvature perturbation
  $\Delta\Phi$ in the RKHS semi-norm $\|\cdot\|_n$,
  and (P3)~uniform denominator concentration.
  Condition~(P2) requires a $C^3$ Nash embedding
  $\iota\colon M\hookrightarrow\bR^D$, which ensures that the
  metric perturbation in normal coordinates is $C^1$-small and
  admits a bounded Fr\'echet derivative via the Sobolev
  embedding $\|v\|_{L^\infty}\le C_S\|v\|_{H^1}$.
  The $C^3$ condition is satisfied by all manifolds in
  \Cref{sec:numerics}: $S^d\hookrightarrow\bR^{d+1}$,
  $SO(3)\hookrightarrow\bR^9$, and
  $\mathrm{Sym}^+(2)\hookrightarrow\bR^4$ are all $C^\infty$;
  for $\bH^d$, Nash's theorem \citep{Nash1956} provides a $C^\infty$
  isometric embedding.
  Although the perturbative argument also applies to NPC targets,
  it introduces the factor $1/\eta(\kappa_+,r_0)$ in the risk
  bound and requires the additional $C^3$ Nash embedding
  assumption, both of which are avoided by the intrinsic proof
  of Part~\ref{item:pointwise-npc}.

  In both parts, the effective bandwidth
  $h_{\mathrm{eff}}=\sqrt{\lambda_n}$ arises from the Green's
  function of the linearized tension-field operator, so the locality
  of the estimator is a consequence of the ODE structure.
\end{remark}

\subsection{Exact finite-dimensional reduction}\label{subsec:reduction}

The geodesic spline structure of \Cref{thm:regularity} has a further consequence
beyond the pointwise risk bound: it implies that the continuous
infinite-dimensional problem~\eqref{eq:estimator} can be solved exactly by
finite-dimensional means, with no discretization error.
For $s=1$, the harmonic map regression estimator is a geodesic
spline (\Cref{thm:regularity}), completely determined by its values
$f_i=\hat F_n(t_{(i)})\in M$ at the order statistics.  We now show that the
continuous infinite-dimensional problem~\eqref{eq:estimator} reduces
exactly to a finite-dimensional optimization on $M^n$.

\begin{proposition}[Exact finite-dimensional reduction]
\label{prop:exact-reduction}
  Let $s=1$ and $I=[0,1]$.  Write $t_{(1)}<\cdots<t_{(n)}$
  for the order statistics and $\Delta_i=t_{(i+1)}-t_{(i)}$ for
  the empirical spacings ($i=1,\ldots,n-1$).  Define
  \begin{equation}\label{eq:discrete-obj}
    \Phi_{\mathrm{disc}}(f_1,\ldots,f_n)
    \;=\;\frac{1}{n}\sum_{i=1}^n d_M^2(Y_{(i)},f_i)
    \;+\;\lambda_n\sum_{i=1}^{n-1}
    \frac{d_M^2(f_i,f_{i+1})}{\Delta_i},
  \end{equation}
  where $Y_{(i)}$ is the response paired with $t_{(i)}$.
  Then:
  \begin{enumerate}[label=(\alph*)]
    \item\label{item:red-value}
      \textbf{(Value-level equivalence.)}
      For any $(f_1,\ldots,f_n)\in M^n$, the infimum of
      $R_n(F)+\lambda_n J_1(F)$ over all
      $F\in\cH^1([0,1],M)$ with $F(t_{(i)})=f_i$ equals
      $\Phi_{\mathrm{disc}}(f_1,\ldots,f_n)$, attained by
      any piecewise constant-speed minimizing-geodesic
      interpolant with the Neumann condition $F'=0$ on the
      boundary intervals (see\
      \Cref{thm:regularity}\ref{item:reg-neumann}).
      When $d_M(f_i,f_{i+1})<\inj(M)$ for all~$i$,
      the interpolant is unique.
    \item\label{item:red-minimizer}
      \textbf{(Minimizer equivalence.)}
      If $\hat F_n$ minimizes
      $R_n(F)+\lambda_n J_1(F)$ over $\cH^1([0,1],M)$,
      then $\hat F_n$ is a piecewise-geodesic interpolant
      and its nodal values
      $f_i^*=\hat F_n(t_{(i)})$ minimize
      $\Phi_{\mathrm{disc}}$ over $M^n$.
    \item\label{item:red-equiv}
      In particular,
      $\min_F\bigl(R_n(F)+\lambda_n J_1(F)\bigr)
      =\min_{(f_i)\in M^n}
      \Phi_{\mathrm{disc}}(f_1,\ldots,f_n)$,
      with no discretization error.
  \end{enumerate}
\end{proposition}

\begin{proof}
  See Appendix~\ref{supp:prop81}.
\end{proof}

Parts~\ref{item:red-value}--\ref{item:red-equiv} hold on any
compact Riemannian manifold (by the Hopf--Rinow theorem;
\citealt[Theorem~6.19]{Lee2018}), with
no injectivity-radius or curvature restriction.
The continuous estimator~\eqref{eq:estimator} and the
finite-dimensional problem~\eqref{eq:discrete-obj} have
identical optimal values, with no discretization error.
The negative Riemannian gradient (the direction of steepest descent) has the explicit form
\begin{equation}\label{eq:grad-flow}
  -\operatorname{grad}_{f_i}\Phi_{\mathrm{disc}}
  = \frac{2}{n}\log_{f_i}(Y_{(i)})
  + \frac{2\lambda_n}{\Delta_i}\log_{f_i}(f_{i+1})
  + \frac{2\lambda_n}{\Delta_{i-1}}\log_{f_i}(f_{i-1}),
\end{equation}
with natural boundary modifications for $i=1$ and $i=n$.

\section{Topological energy barrier}\label{sec:topology}

When the target manifold has non-trivial fundamental group, the
Dirichlet energy acquires homotopy-class-dependent lower bounds that
create a new statistical phenomenon in which the convergence behavior of the
estimator depends on the homotopy class of the regression function.
The oracle inequality (\Cref{thm:oracle}) transfers the resulting
deterministic class-separation gap to the estimator, producing a
phase transition for homotopy class recovery.
We first establish the geometric energy lower bound for loops
(\Cref{thm:topological}), then extend it to boundary conditions
(\Cref{thm:topo-bdry}).  The statistical consequences, an irreducible
wrong-class estimation error (\Cref{thm:topo-cost}) and a phase
transition for homotopy class recovery (\Cref{thm:topo-phase}), are
developed in \Cref{subsec:topo-stat}.

\subsection{Energy lower bounds}\label{subsec:energy-lower-bounds}

We present two families of energy lower bounds, for loops
($\cX=S^1$) and for paths with prescribed boundary conditions
($I=[0,1]$).

\begin{theorem}[Topological energy barrier]\label{thm:topological}
  Let $\cX=S^1$ (the circle of circumference~$L$) and $(M,g)$ a compact
  Riemannian manifold with $\pi_1(M)\neq 0$.  Let
  $[F]\in\pi_1(M)\cong[S^1,M]$ denote the free homotopy class of a
  continuous map $F\colon S^1\to M$.  Then:
  \begin{enumerate}[label=(\alph*)]
    \item\label{item:topo-bound}
      For any $F\in H^1(S^1,M)$,
      \begin{equation}\label{eq:topo-bound}
        \Dir(F) \;=\;\frac{1}{2}\int_{S^1}\abs{F'}^2\dd t
        \;\ge\;\frac{\ell([F])^2}{2L},
      \end{equation}
      where $\ell([F])$ is the length of the shortest closed geodesic in the
      free homotopy class~$[F]$.

    \item\label{item:topo-minimizer}
      Equality in \eqref{eq:topo-bound} holds if and only if $F$ is a
      constant-speed closed geodesic in class~$[F]$.

    \item\label{item:topo-oracle}
      If the regression function $m$ has homotopy class $[m]\neq 0$ (the
      trivial class), then the oracle cost within the same homotopy
      class satisfies
      \[
        \inf_{F\in[m]}\bigl\{\cE(F)+\lambda_n J_1(F)\bigr\}
        \;\ge\;\lambda_n\cdot\frac{\ell([m])^2}{L}
        \;\ge\;\lambda_n\cdot\frac{\operatorname{sys}(M)^2}{L},
      \]
      where the infimum is over $F\in[m]\cap\cH^1(S^1,M)$ and
      $\operatorname{sys}(M)$ is the \emph{systole} of~$M$ (length of
      the shortest non-contractible closed geodesic).
      The unrestricted infimum $\inf_F\{\cE(F)+\lambda_n J_1(F)\}$
      may be smaller (e.g., attained by a constant map in the trivial
      class with $J_1=0$).
  \end{enumerate}
\end{theorem}

\begin{proof}
  See Appendix~\ref{supp:thm61}.
\end{proof}

The topological energy barrier means that if the true regression
function~$m$ ``winds around'' the manifold (e.g., $m\colon S^1\to S^1$
with winding number $k\neq 0$, or $m$ wrapping around a handle of a
surface of genus~$\ge 1$), then any smooth approximation to~$m$ must
pay an irreducible energy cost $\ge\ell([m])^2/L$ in the Dirichlet
energy.  This cost enters the bias of the oracle inequality as
$\lambda_n\cdot\ell([m])^2/L$, which cannot be reduced by
choosing~$\lambda_n$ smaller without increasing the variance.
No such barrier exists in the Euclidean setting
($M=\bR^d$), where every map is null-homotopic.  Thus the
topology of~$M$ has a quantifiable effect on the estimation problem,
beyond the metric considerations captured by
$\eta(\kappa_+,r_0)$.

\begin{example}[Winding number on the circle]\label{ex:circle}
  Let $M=S^1$ (the circle of circumference~$2\pi$) and $\cX=S^1$ of
  circumference~$L$.  A map $F\colon S^1\to S^1$ has a winding number
  $\deg(F)\in\bZ$.  The shortest closed geodesic in the class of winding
  number~$k$ has length $2\pi\abs{k}$.  Thus
  $\Dir(F)\ge 2\pi^2 k^2/L$ for $F$ of degree~$k$.

  If $m$ has degree $k\neq 0$, the oracle inequality gives
  \[
    \bE\biggl[\int d_{S^1}^2(\hat F_n,m)\dd x\biggr]
    \;\le\; C\Bigl(\lambda_n\cdot\frac{4\pi^2 k^2}{L}
    +\frac{1}{n\lambda_n^{1/2}}\Bigr),
  \]
  showing an explicit dependence on the winding number~$k$ in the bias.
\end{example}

\begin{example}[Surfaces of higher genus]
  For $M$ a closed surface of genus $g\ge 1$, $\operatorname{sys}(M)>0$
  by Gromov's systolic inequality \citep{Gromov1983}.  Any non-contractible
  regression function on $\cX=S^1$ incurs a minimum energy cost of
  $\operatorname{sys}(M)^2/(2L)$.
\end{example}

\label{subsec:topo-bdry}
The topological barrier extends to non-periodic domains when
boundary conditions are imposed.

\begin{theorem}[Topological barrier with boundary conditions]
\label{thm:topo-bdry}
  Let $I=[0,1]$ and $(M,g)$ a compact Riemannian manifold with
  $\pi_1(M)\neq 0$.  Fix boundary values $y_0,y_1\in M$ and define
  \[
    \cH^1_{y_0,y_1}([0,1],M)
    \;=\;\{F\in H^1([0,1],M) : F(0)=y_0,\;F(1)=y_1\}.
  \]
  The space $\cH^1_{y_0,y_1}$ decomposes into path-connected components
  indexed by the set of relative homotopy classes of paths from $y_0$
  to~$y_1$, which forms a torsor over $\pi_1(M,y_0)$.  For each class
  $\alpha$, the following hold.
  \begin{enumerate}[label=(\alph*)]
    \item\label{item:bdry-bound}
      Every $F$ in class $\alpha$ satisfies
      \begin{equation}\label{eq:bdry-bound}
        \Dir(F) \;\ge\; \frac{\ell(\alpha)^2}{2},
      \end{equation}
      where $\ell(\alpha)$ is the infimum of arc lengths over all paths
      from~$y_0$ to~$y_1$ in class~$\alpha$.

    \item\label{item:bdry-eq}
      Equality holds if and only if $F$ is a constant-speed geodesic in
      class~$\alpha$.

    \item\label{item:bdry-oracle}
      If the regression function $m\in\cH^1_{y_0,y_1}$ has relative homotopy
      class~$\alpha\neq\alpha_0$ (where $\alpha_0$ is the class of the length
      minimizer), then the bias satisfies
      \[
        \inf_{F\in[\alpha]}\bigl\{\cE(F)+\lambda_n J_1(F)\bigr\}
        \;\ge\;\lambda_n\bigl(\ell(\alpha)^2-\ell(\alpha_0)^2\bigr).
      \]
  \end{enumerate}
\end{theorem}

\begin{proof}
  See Appendix~\ref{supp:thm65}.
\end{proof}

\begin{example}[Non-trivial path class on a surface of genus $g\ge 1$]
  Let $M$ be a closed surface of genus~$g\ge 1$ and fix $y_0=y_1$ (a loop).
  Then $\cH^1_{y_0,y_0}$ decomposes by $\pi_1(M,y_0)$, and for each
  non-trivial $\alpha\in\pi_1(M,y_0)$, the minimum energy is the energy of
  the shortest closed geodesic in class~$\alpha$.  If $m$ loops around one
  handle, the barrier is positive and depends on the hyperbolic geometry of
  the surface.
\end{example}

\subsection{Statistical consequences of topological barriers}
\label{subsec:topo-stat}

The energy lower bounds in
\Cref{thm:topological,thm:topo-bdry} are deterministic geometric
statements.  We now derive their statistical consequences.  The
argument proceeds in two steps.  First, \Cref{thm:topo-cost} shows
that any estimator in the wrong homotopy class incurs an irreducible
excess risk of order $\sqrt{\lambda_n}$. Second,
\Cref{thm:topo-phase} combines this lower bound with the oracle
inequality to show that the estimator recovers the correct class with
probability tending to one.

\begin{theorem}[Statistical cost of wrong homotopy class]
\label{thm:topo-cost}
  Let $\cX=S^1$ of circumference~$L$, $(M,g)$ compact with convexity
  radius~$\rho_M>0$, and $f\ge c_f>0$.  Let $m\in H^1(S^1,M)$ with
  $[m]=\alpha$ and $\Dir(m)\le B$.
  \begin{enumerate}[label=(\alph*)]
    \item\label{item:cost-uniform}
      Any continuous $F\colon S^1\to M$ with $[F]\neq\alpha$ satisfies
      \begin{equation}\label{eq:uniform-separation}
        \sup_{t\in S^1}d_M\bigl(F(t),\,m(t)\bigr)\;\ge\;\rho_M.
      \end{equation}

    \item\label{item:cost-oracle}
      Suppose additionally that the noise support radius
      (\ref{ass:concentration}) satisfies $r_0\le\rho_M/4$, and define
      $\tilde\eta=1-4r_0/\rho_M\in(0,1]$.
      For any homotopy class $\alpha'\neq\alpha$, there exist
      $c_{\mathrm{topo}}>0$ and $\lambda_0>0$ depending on
      $\rho_M,c_f,B,\tilde\eta$ such that for all
      $0<\lambda\le\lambda_0$,
      \begin{equation}\label{eq:wrong-class-lb}
        \inf_{F\in[\alpha']\cap H^1(S^1,M)}
        \bigl\{\cE(F)+\lambda\,J_1(F)\bigr\}
        \;\ge\;c_{\mathrm{topo}}\,\sqrt{\lambda}.
      \end{equation}
      Explicitly,
      $c_{\mathrm{topo}}=\frac{1}{2}\rho_M^2
      \sqrt{\tilde\eta\,c_f/8}$
      and
      $\lambda_0=\tilde\eta\,c_f\,\rho_M^4/(512\,B^2)$.
      When $\Sec_M\le 0$, the condition $r_0\le\rho_M/4$ can be
      removed and $\tilde\eta=1$, by the NPC identity
      (\Cref{thm:NPC-sharp}).
  \end{enumerate}
\end{theorem}

\begin{proof}
  See Appendix~\ref{supp:thm67}.
\end{proof}

The $\sqrt\lambda$ scaling arises because a wrong-class map must
concentrate its topological deviation in a narrow region, creating an
energy--fidelity trade-off that balances at
$\epsilon\sim\sqrt\lambda$.  This phenomenon has no Euclidean analogue.

The mechanism behind the phase transition is a deterministic
\emph{class-separation gap}.  The oracle value in the correct class satisfies
$V_n(\alpha)\le 2\lambda_n\Dir(m)=O(\lambda_n)$, while the oracle in
any wrong class $V_n(\beta)\ge c_{\mathrm{topo}}\sqrt{\lambda_n}$.
Since $\sqrt{\lambda_n}\gg\lambda_n$ as $n\to\infty$, the gap
$\Delta_n\asymp n^{-1/3}$ eventually dominates the empirical process
fluctuation, which is also $O_P(n^{-1/3})$ but with a smaller
constant.
To place these rates in context, in Euclidean nonparametric
regression the estimation error and empirical process fluctuation are
of the same order, leaving no room for a separation gap.  Here the
topological energy barrier opens a gap at the $O(n^{-1/3})$ scale
whose \emph{constant} exceeds that of the fluctuation when the
penalty constant $c_\lambda$ is large enough, yielding exponentially
fast class recovery, a phenomenon absent in Euclidean settings.
The following theorem makes this precise.

\begin{theorem}[Phase transition for homotopy class recovery]
\label{thm:topo-phase}
  Under the conditions of \Cref{thm:topo-cost}, let $\hat F_n$ be the
  harmonic map regression estimator with $\lambda_n>0$.
  \begin{enumerate}[label=(\alph*)]
    \item\label{item:phase-recovery}
      \textbf{(Class recovery.)}  If
      $\lambda_n=c_\lambda n^{-2/3}$ with
      $c_\lambda\ge c_\lambda^*$ for an explicit finite threshold
      $c_\lambda^*=c_\lambda^*(M,I,c_f,\sigma_0)$ (given in the
      Appendix~\ref{supp:thm610}), then
      \begin{equation}\label{eq:class-recovery}
        \bP\bigl([\hat F_n]=\alpha\bigr)\;\to\;1
        \qquad\text{as }n\to\infty,
      \end{equation}
      and the estimator selects
      class~$\alpha$ for all $n\ge n_0$, where
      \begin{equation}\label{eq:n-threshold}
        n_0 \;=\;
        C_1^3 c^{3/2}\,\biggl(\frac{2\Dir(m)}{c_{\mathrm{topo}}}\biggr)^3
      \end{equation}
      depends polynomially on the Dirichlet energy~$\Dir(m)$.

    \item\label{item:phase-break}
      \textbf{(Over-regularization breakdown.)}  If
      $\lambda_n\ge\diam(M)^2 L/\ell(\alpha)^2$, the oracle in the trivial
      class~$\alpha_0$ achieves objective $V(\alpha_0)\le\diam(M)^2$
      (attained by any constant map), while the oracle in
      class~$\alpha$ satisfies
      $V(\alpha)\ge\lambda_n\,\ell(\alpha)^2/L$
      (\Cref{thm:topological}\ref{item:topo-oracle}).  Hence
      $V(\alpha_0)<V(\alpha)$, so the oracle, and for large enough~$n$
      the estimator, prefer a topologically simpler class.
  \end{enumerate}
\end{theorem}

\begin{proof}
  The proof combines the class-selection margin analysis
  (Proposition~\ref{prop:class-margin}) with empirical process
  concentration. See Appendix~\ref{supp:thm610}.
\end{proof}

\begin{remark}[Structure of the phase transition]
  \Cref{thm:topo-phase} identifies a window of correct behavior:
  $\lambda_n$ must satisfy
  $C/(c_{\mathrm{topo}}^2 n)
  \ll\lambda_n
  \ll\diam(M)^2 L/\ell(\alpha)^2$.
  The lower bound ensures that the stochastic error does not swamp the
  topological signal, while the upper bound prevents over-regularization from
  penalizing the topologically necessary energy of~$m$.  The optimal
  rate choice $\lambda_n\asymp n^{-2/3}$ lies within this window for
  all $n\ge n_0$.
  For $M=S^1$ and $m$ of winding number~$k$:
  $n_0\sim k^6/(c_f^{3/2}\rho_M^6)$, growing as the sixth
  power of the winding number.
  The condition $r_0\le\rho_M/4$ is a mild signal-to-noise
  requirement satisfied on all compact Riemannian manifolds with
  non-trivial fundamental group.
  On NPC targets, this condition can be removed entirely
  (Appendix~\ref{supp:thm67}--\ref{supp:thm610}).
\end{remark}

As a concrete illustration, consider $M=S^1$ of circumference $2\pi$,
$\cX=S^1$ of circumference $L=1$,
and $m(t)=2\pi k t$ (mod $2\pi$) with winding number~$k$.  For $k=1$,
$\Dir(m)=2\pi^2$ and $\rho_M=\pi$.  With the oracle-optimal
$\lambda_n\asymp n^{-2/3}$, the correct-class cost is
$V(\alpha)=O(n^{-2/3})$ while the wrong-class cost is
$V(\alpha_0)\ge c\,n^{-1/3}$.  For moderate~$n$ (say, $n\ge 50$
for $k=1$), the estimator recovers the correct winding number with
high probability.  For $k=5$, the threshold $n_0\sim 5^6\approx
15{,}000$ illustrates the rapid growth of the threshold with topological complexity.
\Cref{fig:phase-transition} maps the empirical recovery probability
across a grid of $(n,\lambda_n)$ values for winding numbers $k=1,2,3$.
Each cell reports the fraction of $30$ replications in which the
estimator recovers the correct winding number ($\sigma=0.3$).
For each~$k$, there is a sharp transition from failure (red) to
success (green) as $n$ increases, and the transition shifts to
larger~$n$ as $k$ grows, consistent with the $n_0\sim k^6$ scaling.

\begin{figure}[t]
  \centering
  \includegraphics[width=1\linewidth]{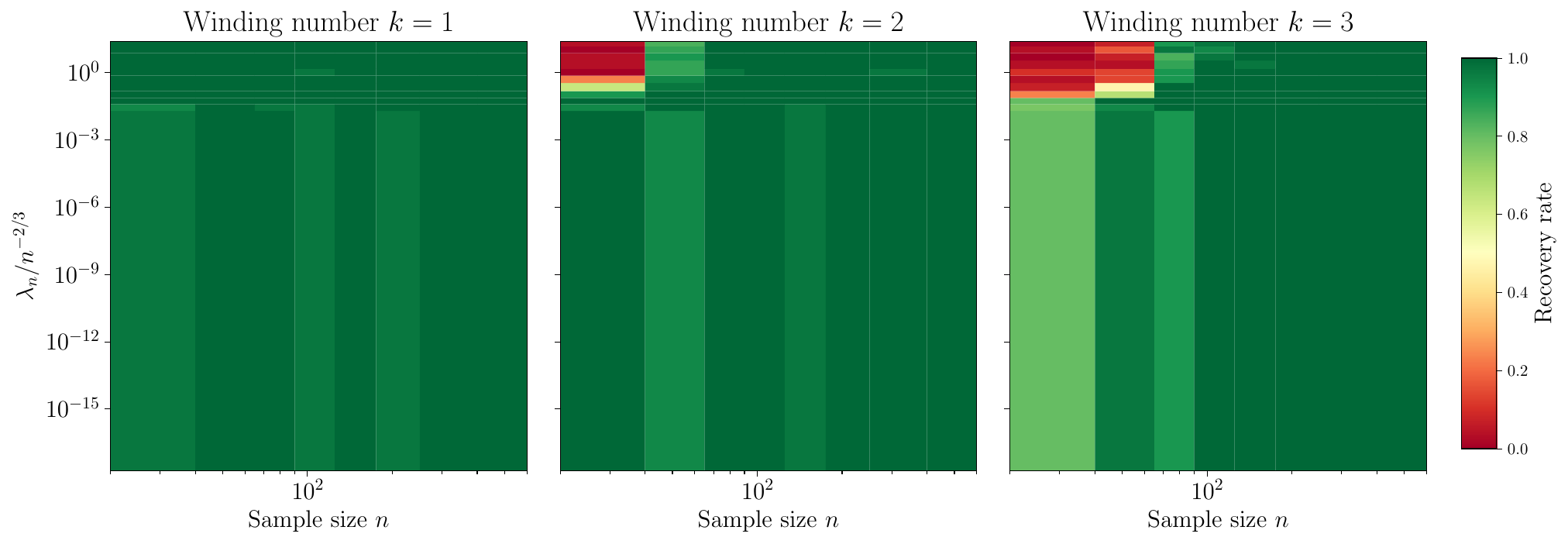}
  \caption{Phase transition for homotopy class recovery on
  $S^1\to S^1$ (\Cref{thm:topo-phase}).
  Higher winding numbers require larger~$n$ for recovery.}
  \label{fig:phase-transition}
\end{figure}

\begin{remark}[Higher homotopy groups]
  The energy barriers in this section exploit the fundamental group
  $\pi_1(M)$.  When the domain $\cX=S^p$ is a higher-dimensional sphere
  ($p\ge 2$), the same mechanism applies via $\pi_p(M)$: any map whose
  homotopy class has non-zero degree incurs a Dirichlet energy cost
  bounded below by the $p$-systole.  This extension and its proof are
  given in Appendix~\ref{supp:higherhomotopy}.
\end{remark}

\section{Obstructions to higher-order intrinsic regularization}
\label{sec:obstructions}

The structural theory of
Sections~\ref{sec:regularity}--\ref{sec:topology} holds at
$s=1$. We further investigate whether the same geometric-analytic structure extends to higher-order penalties $s\ge 2$. Five obstructions show that this is not the case, and the restriction to $s=1$ is a mathematical boundary. Obstruction~3 is proved in this paper. The remaining four are classical results in geometric analysis that we state for completeness.
We use the term \emph{obstruction} to mean a concrete counterexample or divergence result
showing that a specific component of the structural theory fails for $s\ge 2$;
together, the five obstructions establish that the full structural chain, including
Euler--Lagrange characterization, geodesic spline representation,
equivalent kernel, and topological recovery,
is unique to the Dirichlet energy ($s=1$).

\emph{Obstruction 1 (Full nonlinearity of the biharmonic equation).}
\label{prop:biharmonic-nonlinear}
  For $s=2$, the Euler--Lagrange equation for the estimator is a
  fourth-order PDE of the form
  $\nabla_t^2(\nabla_t F')+\text{lower order}=\text{(data)}$,
  where the highest-order term involves the curvature tensor
  nonlinearly through $\nabla_t F'$ and is fully nonlinear in the
  highest-order covariant derivatives \citep{Jiang1986}.  In particular,
  the equation cannot be written as $L(F)=\text{(data)}$ for a
  linear elliptic operator~$L$, and no general existence or regularity
  theory analogous to \citet{Schoen1984} is available for compact
  targets of arbitrary dimension \citep{LammRiviere2008}.

  \emph{Obstruction 2 (Maximum principle failure).}
\label{prop:max-principle}
  The fourth-order operator $\lambda\Delta^2+\Id$ on $[0,1]$ violates the
  maximum principle, and there exist non-negative boundary data for which the
  solution takes negative values in the interior.  Consequently, comparison
  arguments (e.g., showing that the estimator remains within a geodesic ball)
  that are available for the second-order harmonic-map operator
  $-\lambda\Delta+\Id$ do not extend to $s=2$.

\begin{proposition}[Topological blindness of the bienergy]
\label{prop:bienergy-blind}
  Let $(M,g)$ be a compact Riemannian manifold with
  $\pi_1(M)\neq 0$, and let $\cX=S^1$ of circumference~$L$.  For every
  free homotopy class $[\alpha]\in[S^1,M]$,
  \begin{equation}\label{eq:bienergy-zero}
    \inf_{F\in[\alpha]\cap H^2(S^1,M)}
    \Dir_2(F) \;=\; 0,
  \end{equation}
  where $\Dir_2(F)=\frac{1}{2}\int_{S^1}\abs{\nabla_t F'}^2\dd t$ is
  the bienergy.  In particular, the phase transition of
  \Cref{thm:topo-phase} has no analogue for biharmonic map
  regression.
\end{proposition}

\begin{proof}
  See Appendix~\ref{supp:prop73}.
\end{proof}

The reason is that the Dirichlet energy measures speed
($\int|F'|^2$), which is positive for any non-contractible loop,
while the bienergy measures acceleration
($\int|\nabla_t F'|^2$), which vanishes for geodesics, precisely
the maps that carry the topological information.  The Eells energy coincides with the bienergy, so every
intrinsic second-order penalty on loops is topologically blind.

\emph{Obstruction 4 (Green's function sign change).}
\label{prop:green-sign}
  The Green's function $G_2(t,t')$ of the operator
  $\lambda\Delta^2+\Id$ on $[0,1]$ with natural boundary conditions
  changes sign, with $G_2(t,t')<0$ for some $t,t'$.  Consequently,
  the equivalent kernel $K_n(t,\cdot)$ in the $s=2$ smoothing-spline
  problem takes negative values, and the estimator $\hat F_n(t)$ cannot
  be interpreted as a weighted Fr\'echet mean of the observations.
  This invalidates the convexity-based arguments used in the
  pointwise analysis of \Cref{thm:pointwise}.

\emph{Obstruction 5 (Non-uniqueness of biharmonic maps).}
\label{prop:biharmonic-nonunique}
  The uniqueness of biharmonic maps with prescribed Dirichlet or
  Navier boundary data remains open for general compact Riemannian
  targets \citep{LammRiviere2008}.  Even for maps from an interval
  to~$S^2$, biharmonic equations can admit multiple solutions at
  identical boundary data, preventing the construction of a
  well-defined representer theorem.

Taken together,
Obstructions~1--5 identify $s=1$ as the unique Sobolev
order at which the full structural theory
(Sections~\ref{sec:regularity}--\ref{sec:topology}) holds.
Extending it to $s\ge 2$ requires, at minimum, a regularity
theory for biharmonic maps into general compact targets, a major
open problem in geometric analysis \citep{LammRiviere2008}.

\section{Numerical experiments}
\label{sec:numerics}

Having established that the structural theory is unique to $s=1$,
we now validate its predictions numerically across targets representing
the full range of geometric and topological regimes covered by the theory.
The theoretical framework applies to any compact Riemannian target once
the curvature bounds and injectivity radius are specified.
The five manifolds used in the experiments below are
the sphere~$S^2$ ($\kappa=+1$),
the hyperbolic plane~$\bH^2$ ($\kappa=-1$),
the symmetric positive-definite matrices~$\mathrm{Sym}^+(2)$ ($\kappa\le 0$),
the rotation group~$SO(3)$ ($\pi_1=\bZ/2$), and
the flat torus~$T^2$ ($\kappa=0$, $\pi_1=\bZ^2$).
Explicit geometric data (second fundamental forms, geodesic and
exponential map formulae, curvature factors, and injectivity radii) for
these manifolds are collected in Appendix~\ref{supp:geom-data}.

The experiments below have four goals. First, we confirm the geodesic-spline
structure and the exact finite-dimensional reduction of
\Cref{prop:exact-reduction}, exhibit the topological phase transition
predicted by \Cref{thm:topo-phase}, verify the convergence rates of
\Cref{sec:main} across manifolds with diverse curvature and topology,
and demonstrate the estimator on real wind-direction data where
wrap-around artefacts are unavoidable for Euclidean methods.  The
comparisons should be interpreted as mechanism demonstrations
rather than a comprehensive benchmark, since the implemented estimators
differ in their tuning protocols and the theory-computation alignment
varies across methods.

\subsection{Algorithm for the geodesic spline}\label{subsec:algorithm}

By \Cref{prop:exact-reduction}, computing the harmonic map regression
estimator reduces to minimizing the discrete
objective~\eqref{eq:discrete-obj} over $M^n$, with Riemannian
gradient~\eqref{eq:grad-flow}.
We minimize~\eqref{eq:discrete-obj} by \emph{Riemannian Polak-Ribi\`ere
conjugate gradient (CG)} with Armijo backtracking.
At iteration $k$, the CG direction $d^{(k)}$ satisfies
\[
  d^{(k)} = -g^{(k)} + \beta_k\,\tau_{k-1\to k}(d^{(k-1)}),
  \qquad
  \beta_k = \frac{\langle g^{(k)},\,g^{(k)} - \tau_{k-1\to k}(g^{(k-1)})\rangle}
  {\langle g^{(k-1)},g^{(k-1)}\rangle}\vee 0,
\]
where $g^{(k)}=\operatorname{grad}\Phi(f^{(k)})$ and
$\tau_{k-1\to k}$ denotes vector transport (approximated by tangent
projection).  The initial step size is set by a CFL condition,
$\alpha_0=c\,\bigl(2\lambda_n\max_i(\Delta_i^{-1}+\Delta_{i-1}^{-1})
+2/n\bigr)^{-1}$, and reduced by Armijo backtracking to guarantee
monotone decrease.

Initializing at $f_i^{(0)}=Y_i$ sets the data residual to zero,
and the optimizer needs only to balance data fidelity against
smoothness.  For $\Sec_M\le 0$ (NPC targets), the objective
\eqref{eq:discrete-obj} is geodesically convex and the CG iterates
converge to the unique global minimizer \citep{Zhang2016riemannian, Boumal2023}.
For general~$M$, local convergence holds from the initialization.
CG reduces the effective condition number from
$\kappa\approx\lambda_n\, n^2$ (for gradient descent) to
$\sqrt\kappa$, resolving the convergence plateaus observed on
$\bH^2$ with first-order methods.
Each CG iteration evaluates~\eqref{eq:grad-flow}
in $O(nd)$ time (one $\exp$ and $\log$ per point), giving
a total cost of $O(Tnd)$. In our experiments $T\le 600$ suffices
for all manifolds and sample sizes.
Once the values $\{f_i\}$ are obtained, the estimator at any $t\in I$ is
evaluated by geodesic interpolation between the two flanking design
points.

\subsection{Simulation study}\label{subsec:numerical}

We evaluate the proposed estimator on five Riemannian manifolds spanning
different geometric and topological regimes, namely $S^2$ (constant positive curvature
$\kappa=+1$), $\bH^2$ (constant negative curvature $\kappa=-1$, NPC),
$\mathrm{Sym}^+(2)$ (non-positively curved, NPC),
$SO(3)$ with a topologically non-trivial winding curve
($\pi_1=\bZ/2$, rotation angle $0\to 2\pi$), a setting motivated
by trajectory analysis on rotation groups
\citep{SuJuddWuFletcherSrivastava2014}, and the flat torus
$T^2$ ($\pi_1=\bZ^2$, $\kappa=0$) with a curve that winds $1.5$ times
around a generator of~$\pi_1$.
Among these, $S^2$, $SO(3)$, and~$T^2$ are compact and fall directly
under the theory of \Cref{sec:main}--\Cref{sec:topology}. For the
non-compact targets $\bH^2$ and $\mathrm{Sym}^+(2)$, we restrict the
test curves and noisy observations to a bounded geodesic ball.

The predictor $t_i$ are drawn i.i.d.\ from the uniform distribution on $[0,1]$,
and sample sizes range over $n\in\{100,200,400,800\}$.
Responses are generated as $Y_i=\exp_{m(t_i)}(\sigma\,\varepsilon_i)$, where
$\varepsilon_i\sim N(0,I_d)$ is a standard Gaussian in the tangent space
$T_{m(t_i)}M$ (coordinates taken in any orthonormal basis of $T_{m(t_i)}M$),
so that $Y_i$ is a random point on~$M$ at geodesic distance approximately $\sigma$
from the true value $m(t_i)$.
The noise level is $\sigma=0.25$ for $S^2$, $\bH^2$, and $\mathrm{Sym}^+(2)$,
and $\sigma=0.05$ for $SO(3)$ and~$T^2$; the smaller value on the latter two is chosen
so that the noise radius stays well below the convexity radius $\rho_M=\pi/2$.

The test curves are chosen to represent qualitatively different geometric
and topological regimes.
On $S^2$, the curve $m(t)=\exp_{p_0}(t\,v_0+t(1-t)\,w_0)$ is a
non-geodesic arc from the south pole neighborhood toward the north,
with the quadratic perturbation $t(1-t)w_0$ (for a fixed tangent vector $w_0\perp v_0$)
introducing curvature relative to a geodesic.
On $\bH^2$, $m(t)$ follows a logarithmic spiral in the Poincar\'e disk,
sweeping outward from the origin as $t$ increases and remaining within a
geodesic ball of radius $1.5$ to satisfy the bounded-domain restriction.
On $\mathrm{Sym}^+(2)$, $m(t)=P_0^{1/2}\exp(t\,S)P_0^{1/2}$ is a
one-parameter geodesic deformed by a quadratic perturbation, where $P_0$ is a
fixed SPD matrix and $S$ is a symmetric matrix controlling the direction and speed.
On $SO(3)$, $m(t)=\exp_{I}(2\pi t\,\hat\Omega)$ is the one-parameter subgroup
generated by a fixed unit-norm skew-symmetric matrix $\hat\Omega$,
tracing a complete rotation from $I$ back to~$I$ and thereby representing
the non-trivial element of $\pi_1(SO(3))=\bZ/2$.
On $T^2=\bR^2/(2\pi\bZ)^2$, $m(t)=(3\pi t\bmod 2\pi,\;2\pi t\bmod 2\pi)$
winds $1.5$ times around the first generator and once around the second,
producing a curve whose homotopy class $(3,2)\in\pi_1(T^2)=\bZ^2$
is non-trivial and non-contractible.
Explicit formulas for the exponential and logarithmic maps, second
fundamental forms, and curvature factors for all five manifolds are given
in Appendix~\ref{supp:geom-data}.

We compare five methods.  The proposed estimator is
the harmonic map regression estimator solved by Riemannian CG
(\Cref{subsec:algorithm}) with $\lambda_n=c_\lambda\,n^{-2/3}$.
\emph{Extrinsic spline regression} fits a cubic smoothing spline
coordinate-wise in the ambient Euclidean space and projects the result
onto~$M$, with smoothing parameter $s=c_e\,n$.
\emph{Fr\'echet regression} \citep{Petersen2019} computes the
weighted Fr\'echet mean at each evaluation point, with weights
$w_i(t)=n^{-1}+(t-\bar t)\hat\Sigma_t^{-1}(t_i-\bar t)/n$ derived from
the linear model.
\emph{Geodesic regression} \citep{Fletcher2013} fits the parametric model
$m(t)=\exp_{y_0}(t\,v)$ by gradient descent, initialized at the Fr\'echet mean
with OLS slope in the tangent space. No winding-number search or other
prior information is used.
\emph{TV-regularized Fr\'echet regression}
\citep{LinMuller2021TV} minimizes
$(1/n)\sum d_M^2(Y_i,f_i)+\lambda\sum d_M(f_i,f_{i+1})$.
The objective is solved by the cyclic proximal point algorithm of
\citet{LinMuller2021TV} (Algorithm~1 therein), where each proximal
sub-problem reduces to geodesic interpolation, and the iterates are
initialized at the data.
Off-grid evaluation uses piecewise-constant (nearest-neighbor)
prediction, consistent with the piecewise-constant structure of TV
estimators.
The theoretical guarantees of \citet{LinMuller2021TV} require NPC
targets for geodesic convexity, but we include the method on all five
manifolds to evaluate its empirical performance beyond the NPC
setting.

Each method with a tunable hyperparameter
uses its theoretically optimal rate as a function of~$n$,
$\lambda_n=c_\lambda\,n^{-2/3}$ for the proposed estimator
(\Cref{thm:oracle})
and $\lambda_n=c_\lambda\,n^{-2/3}$ for TV Fr\'echet.
The multiplicative constant for each method is selected
by 5-fold cross-validation on each data realization.
The design points are assigned to folds in round-robin order
(i.e., observation $i$ goes to fold $i\bmod 5$), so that each
fold is scattered throughout $[0,1]$ rather than forming a
contiguous block.  The constant is chosen from a
logarithmically spaced grid of six candidates.
For the proposed estimator,
$c_\lambda\in\{0.01,0.03,0.1,0.3,1.0,3.0\}$;
for extrinsic spline regression, we use the
\texttt{UnivariateSpline} implementation in \texttt{scipy} with smoothing
parameter $s\in\{0.05,0.1,0.3,0.5,1.0,2.0\}$;
for TV Fr\'echet, $c_\lambda\in\{0.01,0.03,0.1,0.3,1.0,3.0\}$.
Fr\'echet regression and geodesic regression have no tunable
hyperparameters.
The CV loss is the average held-out prediction error
$n_{\mathrm{val}}^{-1}\sum d_M^2(\hat f(t_j),Y_j)$. After
selecting the best hyperparameter, the method is refit on the full
data set.

For each sample size $n\in\{100,200,400,800\}$, the mean integrated
squared error $\mathrm{MISE}=\int d_M^2(\hat F(x),m(x))\,dx$ is
approximated on a 50-point evaluation grid and averaged over 15
independent replications.  Standard errors are reported in parentheses.
Full numerical results are in Table~\ref{tab:comprehensive}.

Figure~\ref{fig:manifold_viz} shows single-realization fitted curves
on three representative manifolds ($S^2$, $\bH^2$, $T^2$) at
$n=400$; each column is one method and the MISE appears above each panel.
On $S^2$, the proposed estimator and the extrinsic spline both trace the
true arc closely, while Fr\'echet and geodesic regression collapse to
near-constant or linear fits that cannot represent the non-geodesic shape.
On $\bH^2$ (Poincar\'e disk, middle row), the gap widens: the extrinsic
spline produces a severely distorted path that overshoots the disk boundary,
reflecting the breakdown of the ambient-to-manifold projection under strong
negative curvature, whereas the proposed estimator follows the spiral
throughout; Fr\'echet and geodesic regression collapse to a point or short
segment near the center, missing the spiral entirely.
On $T^2$, the proposed estimator reproduces the $1.5$-winding curve cleanly;
TV-Fr\'echet regression also recovers the winding but with a rougher path
near the wrap-around point; the extrinsic spline incurs a jump discontinuity
at the fundamental domain boundary; and Fr\'echet and geodesic regression
fail to capture the homotopy class, producing a self-crossing path and a
single geodesic segment, respectively.

Figure~\ref{fig:comprehensive} plots log-log MISE versus sample
size on all five manifolds, averaged over 15 replications, with
reference slopes $n^{-2/3}$ (oracle rate at $s=1$) and $n^{-1/3}$
shown as dotted lines.
Table~\ref{tab:comprehensive} reports the full MISE values with standard errors.
Figure~\ref{fig:comprehensive} confirms that on all five manifolds the
proposed estimator's log-log MISE slope is consistent with the oracle
rate $n^{-2/3}$ of \Cref{thm:oracle}.
By contrast, Fr\'echet regression and geodesic regression plateau,
reflecting the bias inherent in global constant or linear fits applied
to nonlinear regression functions.
Among the simply-connected targets ($S^2$, $\bH^2$, $\mathrm{Sym}^+(2)$),
TV-Fr\'echet regression converges at the correct rate but with MISE
roughly $1.5$--$1.6\times$ larger than the proposed estimator at $n=800$.
The extrinsic spline does not share this property; the project-and-normalize
step introduces a persistent geometry-induced error floor
(MISE $\approx 0.009$ on~$S^2$, $\approx 0.12$ on~$\bH^2$)
that does not diminish as $n$ grows.
The sharpest differentiation appears on the manifolds with $\pi_1\neq 0$.
On $SO(3)$-wind and~$T^2$, the proposed estimator reaches
MISE $\approx 0.0002$ by $n=400$--$800$, while the extrinsic spline
plateaus because of coordinate singularities (quaternion sign fold on
$SO(3)$, chart discontinuity on~$T^2$), and geodesic regression diverges
on $SO(3)$ because the parametric model $t\mapsto\exp_{y_0}(tv)$ cannot
represent a non-trivial homotopy class.
These empirical failures are the direct manifestation of the phase
transition in \Cref{thm:topo-phase}.

\begin{figure}[t]
\centering
\includegraphics[width=\textwidth]{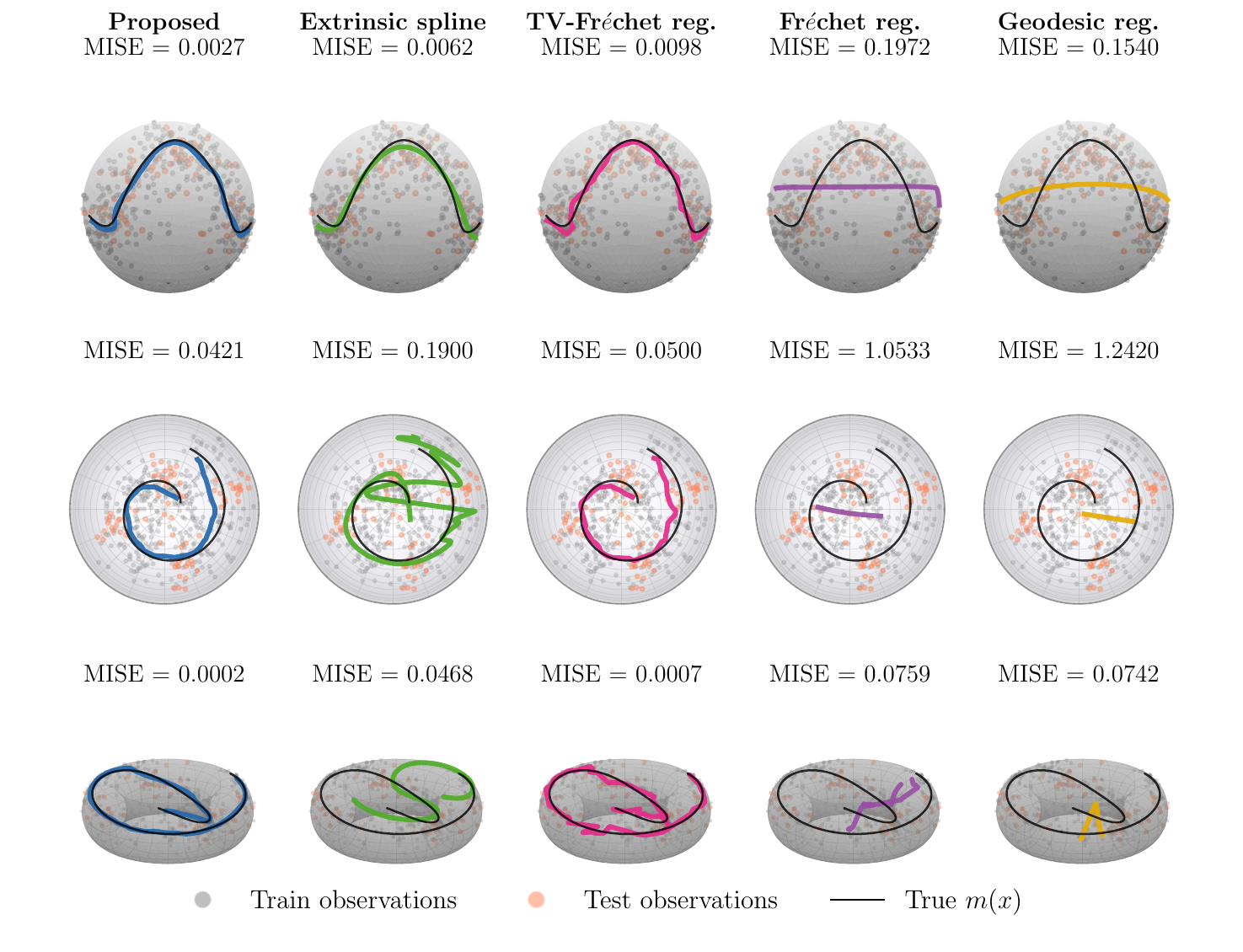}
\caption{A demonstration of the fitted curves on $S^2$, $\bH^2$, and~$T^2$ ($n=400$).}\label{fig:manifold_viz}
\end{figure}

\begin{figure}[t]
\centering
\includegraphics[width=\textwidth]{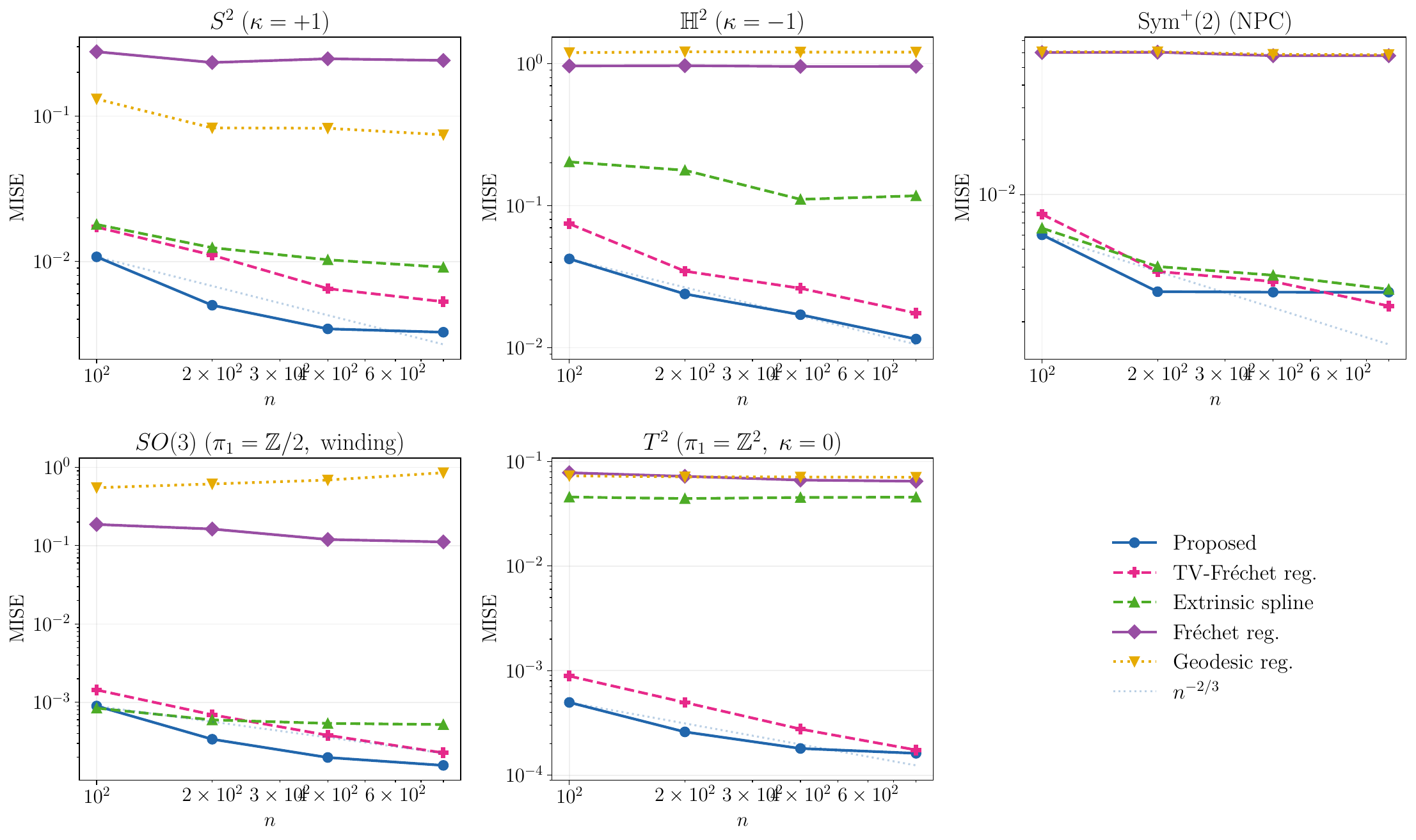}
\caption{Log-log MISE versus sample size on five manifolds.
  Reference slope $n^{-2/3}$ is shown as dotted lines.}
\label{fig:comprehensive}
\end{figure}

Additional experiments in Appendix~\ref{sec:additional-numerics} include a curvature
sensitivity study on spheres $S^2(R)$ with varying~$\kappa$, confirming
the asymptotically curvature-free constant of \Cref{thm:oracle}, and a
winding experiment on $S^2$ with a five-times equatorial wrap.

\subsection{Real-data case study: wind direction on $S^1$}
\label{subsec:wind}

We apply the harmonic map regression estimator to hourly wind-direction
observations from the NOAA Integrated Surface Database
\citep{NOAAISD2023} (station USAF 725300, Chicago O'Hare, year 2023).
Wind direction is recorded as an
angle in $[0^\circ,360^\circ)$ relative to true north, providing a natural
$S^1$-valued response.  The wrap-around at $0^\circ/360^\circ$ is a
well-known pathology for Euclidean smoothers. When the true direction is
near north, observations at $350^\circ$ and $10^\circ$ are geodesically
close on $S^1$ but numerically distant in $\bR$, so any method that
averages raw degree values is pulled toward $180^\circ$.

We use $n=649$ non-calm hourly observations from June, with
sequential time normalized to $[0,1]$.  June exhibits the strongest
directional variation (mean resultant length $R=0.38$), with the
prevailing direction near north and 40 wrap-around events.
To evaluate predictive accuracy, we adopt a scattered-block design:
the month is divided into $20$ blocks of ${\approx}\,1.5$ days, with
every fifth block held out for testing (${\approx}\,20\%$).
Hyperparameters are tuned by $4$-fold cross-validation on the
training set, where the $16$ training blocks are assigned to folds
in round-robin order so that each validation fold consists of $4$
non-adjacent blocks spread throughout the month.
The CV loss is the average held-out geodesic prediction error
$n_{\mathrm{val}}^{-1}\sum d_{S^1}^2(\hat f(t_j),Y_j)$.

\begin{figure}[t]
\centering
\includegraphics[width=\textwidth]{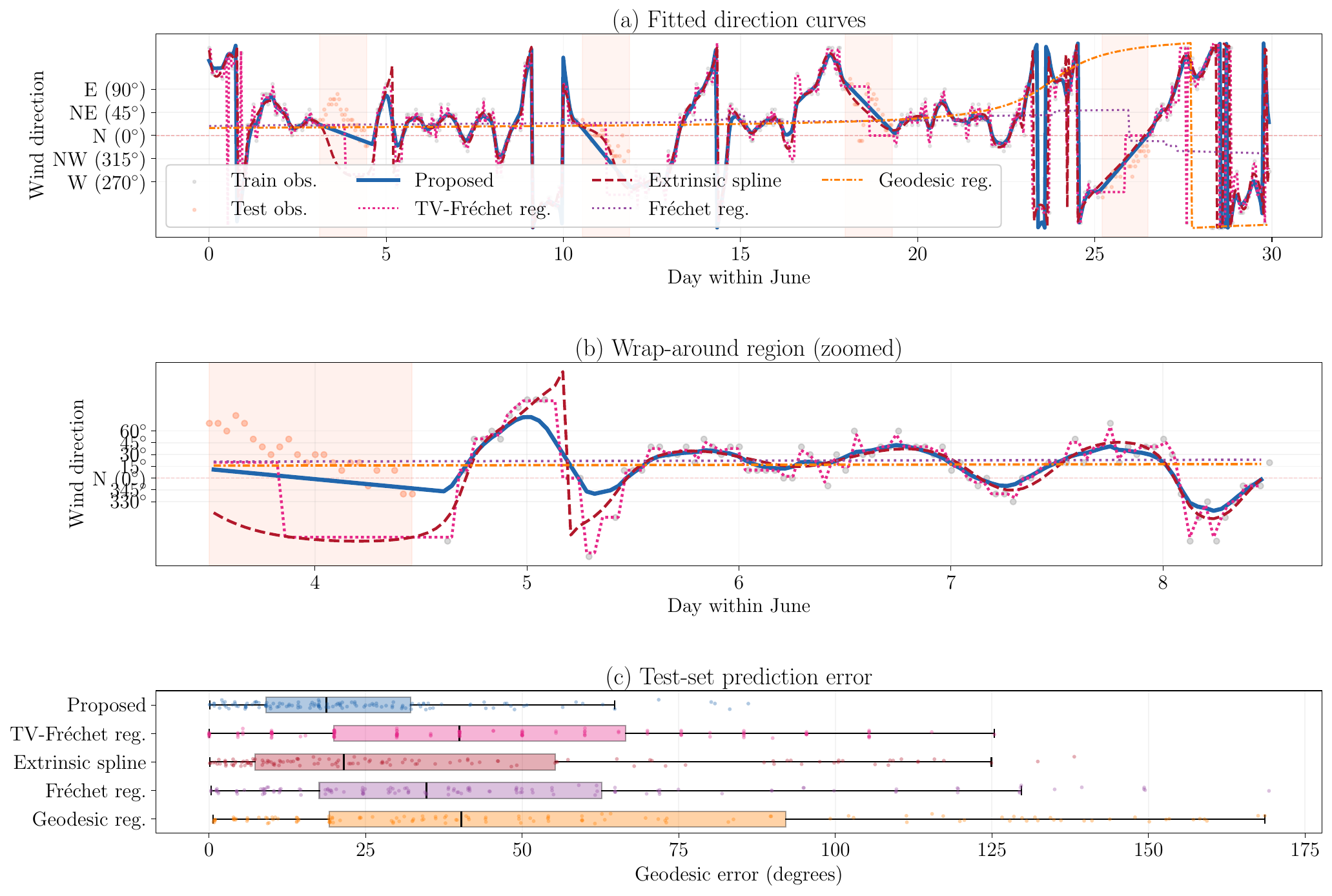}
\caption{Wind direction regression on~$S^1$ ($n=649$ hourly observations, June).
  (a)~Fitted curves for all methods.
  (b)~Zoom near wrap-around events.
  (c)~Boxplot of test-set geodesic errors.}
\label{fig:wind}
\end{figure}

Figure~\ref{fig:wind} shows the results.  In panel~(a), the harmonic
map spline (blue) closely tracks the observed direction trajectory
throughout the month, including the scattered test blocks (coral shading).
The extrinsic spline (red) captures broad trends but smooths over rapid
direction changes.
TV-Fr\'echet regression (pink) produces a piecewise-constant fit whose
total-variation penalty respects the $S^1$ geometry, but the
staircase structure cannot represent the continuously varying wind
direction.  Cross-validation selects the smallest regularization parameter
in the grid, reflecting the fundamental tension between the piecewise-constant
function class and this smooth signal.
Fr\'echet regression and geodesic regression yield near-constant
fits.
Panel~(b) zooms into a five-day stretch containing several wrap-around
events. The harmonic map spline crosses $0^\circ$ without artefacts,
whereas the TV-Fr\'echet staircase deviates substantially.

Table~\ref{tab:wind-supp} reports the full test-set geodesic prediction errors.
The harmonic map spline achieves the lowest MSGE ($0.275$ rad$^2$,
RMGE $= 30.1^\circ$), reducing the error substantially compared to the
next-best competitor, the extrinsic spline ($0.832$ rad$^2$,
RMGE $= 52.3^\circ$), which is penalized by projecting Euclidean-smoothed
ambient coordinates back onto~$S^1$.
TV-Fr\'echet regression \citep{LinMuller2021TV} ($0.897$ rad$^2$,
RMGE $= 54.3^\circ$) illustrates the cost of piecewise-constant estimation
for a smooth signal. Although its total-variation penalty respects
the intrinsic geometry of~$S^1$, the resulting staircase fit cannot
represent continuous directional variation, and cross-validation
consistently selects the least regularized candidate.
Fr\'echet regression and geodesic regression produce RMGE
$\ge 61^\circ$, reflecting the limitations of constant or linear fits.
The
harmonic map spline benefits from its piecewise-geodesic structure:
between flanking training points it interpolates along constant-speed
geodesics on~$S^1$, which avoids
wrap-around artefacts by construction.  Scattered-block CV selects
$\lambda = 1\times 10^{-5}$, and the resulting spline smoothly resolves
the rapid directional oscillations throughout June.

\section{Discussion}\label{sec:discussion}

The main result of this paper is that the first-order Dirichlet
regime $s=1$ is the setting in which manifold-valued nonparametric
regression acquires a structural theory analogous to that of the
Euclidean smoothing spline.  The Dirichlet energy penalty converts
the regression problem into a perturbed harmonic map equation,
and the resulting PDE-regularity and topological-obstruction theory
has no analogue in existing manifold-valued regression frameworks.
The obstructions of
\Cref{sec:obstructions} show that this boundary reflects intrinsic mathematical constraints:
bienergy topological blindness and absence of higher-order existence
theory.  Extending the structural
theory to higher smoothness will
require progress on open problems in geometric analysis, including
biharmonic map regularity and higher-order topological sensitivity.

This paper leaves several open questions. First,
the interplay between curvature and the pointwise constant is only partially
resolved.  For non-positively curved targets,
\Cref{thm:pointwise}\ref{item:pointwise-npc} provides a non-perturbative
pointwise bound whose variance constant is at most the Euclidean constant,
and \Cref{thm:NPC-sharp} shows that the oracle constant matches the Euclidean
theory exactly.  For positively curved targets, the perturbative bound
(\Cref{thm:pointwise}\ref{item:pointwise-general}) carries the factor
$1/\eta(\kappa_+,r_0)>1$, while the Assouad lower bound
(\Cref{thm:minimax}) is curvature-free.  Asymptotically, localization
implies $\eta(\kappa_+,r_n)\to 1$, recovering the Euclidean constant.



Second, the structural results of Sections~\ref{sec:regularity}--\ref{sec:topology}
are intrinsically one-dimensional.  The geodesic spline characterization
(\Cref{thm:regularity}), the equivalent-kernel pointwise risk bound
(\Cref{thm:pointwise}), and the topological phase transition
(\Cref{thm:topo-phase}) all rely on the Sobolev embedding
$H^1(I,\bR^D)\hookrightarrow C^0(\bar I,\bR^D)$, which ensures
that the pointwise fidelity term $d_M^2(Y_i,F(t_i))$ is well defined
for $H^1$~maps.  For covariate domains $\cX\subset\bR^p$ with $p\ge 2$,
this embedding fails since $H^1(\cX,\bR^D)$ does not embed into $C^0(\cX,\bR^D)$,
and the pointwise data-fidelity term becomes undefined.

Third, the Dirichlet energy $\Dir(F)=\frac{1}{2}\int_\cX|\nabla F|_g^2$ remains
well-defined for maps $F\colon\cX\to M$ with $\cX\subset\bR^p$,
and the oracle inequality (\Cref{thm:oracle}) extends to the
kernel-smoothed estimator (Appendix~\ref{supp:ks-exist}--\ref{supp:ks-oracle}), achieving the
minimax rate $n^{-2/(2+p)}$ at $s=1$.  The appendix also establishes
smoothness of the kernel-smoothed minimizer (Appendix~\ref{supp:highdim}) and a pointwise
risk bound of order $n^{-1/2}$ for surface-valued regression at $p=2$
(Appendix~\ref{supp:surface}).  However, this kernel-smoothed estimator
replaces the pointwise empirical risk with a smoothed version
$R_n^h(F)=n^{-1}\sum_i\int K_h(t-t_i)\,d_M^2(Y_i,F(t))\,dt$,
which is a fundamentally different estimator with different properties.
The geodesic spline structure is lost (the Euler--Lagrange equation becomes
a PDE rather than an ODE with point sources), and the topological
phase transition has no known analogue.

Overall, the theoretical and numerical results in this paper identify harmonic map regression as
a natural integration of nonparametric statistics, Riemannian geometry,
and homotopy theory. The Dirichlet energy is not merely a convenient
analogue of the Euclidean Sobolev penalty; it is the unique first-order
energy that transfers the full structural chain of the smoothing spline
to the manifold setting, with rate optimality, geometric interpretability,
and topological sensitivity emerging from a single variational principle.

\FloatBarrier
\putbib
\end{bibunit}

\appendix


\part*{Supplementary Material}
\addcontentsline{toc}{part}{Supplementary Material}

\renewcommand{\thetheorem}{S\arabic{section}.\arabic{theorem}}
\renewcommand{\theequation}{\thesection.\arabic{equation}}

This supplementary material contains the complete proofs of the main results stated in
the paper, as well as extensions to higher-dimensional covariate domains
and higher-order Sobolev penalties.  Theorem, lemma, and equation numbers
refer to the main text unless prefixed with ``S''.

\begin{bibunit}

\section{Proofs for Sections 2--3 (Preliminaries)}

\subsection{Proof of Lemma~2.1 (Variance inequality)}\label{supp:lemma21}

\label{proof:variance-ineq}

\begin{proof}
  For $\kappa\le 0$: the squared-distance function $z\mapsto d_M^2(z,y)$ is
  convex, and the standard variance decomposition holds with equality; see
  \citet[Proposition~4.4]{Sturm2003}.
  For $\kappa>0$: the Hessian comparison theorem
  \citep[Theorem~1.2]{Karcher1977},
  \citep[Chapter~1, Theorems~1.28-1.29]{CheegerEbin2008}
  gives, for every $z$ with $d_M(z,y)<\pi/\sqrt\kappa$,
  \[
    \mathrm{Hess}_y\bigl(d_M^2(z,\cdot)/2\bigr)
    \;\ge\;\eta\bigl(\kappa,\,d_M(z,y)\bigr)\,g_y,
  \]
  where the radial eigenvalue is exactly~$1$ and the orthogonal eigenvalues
  are $\ge\eta(\kappa,d_M(z,y))$; the bound $\eta$ is the minimum.
  Since $\supp(\mu)\subset B_M(\bar y,r_0)$, every
  $z\in\supp(\mu)$ satisfies $d_M(z,\bar y)\le r_0$.
  Define $\varphi(y)=\int_M d_M^2(z,y)\dd\mu(z)$.  Since $\bar y$
  minimizes~$\varphi$, the Riemannian gradient vanishes:
  $\mathrm{grad}\,\varphi(\bar y)
  =-2\int_M\log_{\bar y}(z)\dd\mu(z)=0$.
  Integrating the pointwise Hessian bound over~$\mu$ gives
  $\mathrm{Hess}_y(\varphi/2)\ge\eta(\kappa,r_0)\,g_y$
  for $y$ with $\sup_{z\in\supp(\mu)}d_M(z,y)\le r_0+d_M(\bar y,y)
  <\pi/\sqrt\kappa$, since the Hessian of $d_M^2(z,\cdot)/2$
  at~$\bar y$ is bounded below by $\eta(\kappa,d_M(z,\bar y))
  \ge\eta(\kappa,r_0)$.
  Combining $\mathrm{grad}\,\varphi(\bar y)=0$ and the strong
  convexity $\mathrm{Hess}(\varphi/2)\ge\eta(\kappa,r_0)\,g$ yields
  the variance inequality.
\end{proof}


\subsection{Proof of Proposition~2.2 (Compactness)}\label{supp:prop22}
\label{proof:compactness}

\begin{proof}
  Compactness of~$M$ gives $\norm{F}_{L^\infty}\le\diam(M)$.
  Combined with $J_s(F)\le E_0$ and the Poincar\'e inequality \citep[Theorem~6.30]{Adams2003} on the
  bounded domain~$I$, this yields a uniform bound
  $\norm{F}_{H^s}\le C(\diam(M), E_0, I)$.  The Rellich-Kondrachov
  theorem ($H^s\hookrightarrow\hookrightarrow C^0$ for $s\ge 1$;
  \citealt[Theorem~6.3]{Adams2003}) and
  closedness of~$M\subset\bR^D$ give the result.
\end{proof}


\subsection{Sub-Gaussian extension via sample splitting (Remark~3.1)}\label{supp:subgauss}
\label{proof:subgaussian}

\begin{proof}
  We condition on $\cD_1$ throughout and treat
  $\hat m^{(1)}$ as a fixed function.
  Let $r_n=\sigma\sqrt{2c\log n}$ for a constant $c>0$
  to be chosen large enough below.

  We begin by controlling the truncation event.
  Define $\delta_n=\sup_t d_M(\hat m^{(1)}(t),m(t))$
  and the good event $\cA_n=\{\delta_n\le r_n/2\}$.
  For $i\in\cD_2$, define
  $\tilde Y_i=\Pi_{r_n}(Y_i\mid\hat m^{(1)}(t_i))$, the
  metric projection of $Y_i$ onto
  $\bar B_M(\hat m^{(1)}(t_i),r_n)$.

  On $\cA_n$, the triangle inequality gives
  \[
    B_M(m(t_i),r_n/2)\subset B_M(\hat m^{(1)}(t_i),r_n),
  \]
  so if $d_M(Y_i,m(t_i))\le r_n/2$, then $Y_i$ is
  not truncated: $\tilde Y_i=Y_i$.
  By the sub-Gaussian tail,
  \[
    \bP(d_M(Y_i,m(t_i))>r_n/2\mid t_i)
    \le 2\exp(-r_n^2/(8\sigma^2))
    =2n^{-c/4}.
  \]
  By a union bound over $i\in\cD_2$ (of size $n_2=\lceil n/2\rceil$):
  \[
    \bP\bigl(\exists\,i\in\cD_2:\tilde Y_i\neq Y_i
    \mid\cD_1,\cA_n\bigr)
    \le n_2\cdot 2n^{-c/4}
    \le 2n^{1-c/4}.
  \]
  For $c>4$, this vanishes.

  Next, we derive an oracle inequality on the truncated data.
  Conditional on $\cD_1$ and the event $\cA_n$:
  each $\tilde Y_i$ is supported in
  $B_M(\hat m^{(1)}(t_i),r_n)
  \subset B_M(m(t_i),r_n+\delta_n)
  \subset B_M(m(t_i),3r_n/2)$.
  Thus (A3) holds with $r_0'=3r_n/2$ and the curvature
  factor is $\eta(\kappa_+,3r_n/2)$.

  The oracle inequality (Theorem~4.5(a)
  of the main text) applies to the estimator
  $\hat F_n^{\mathrm{trunc}}$ computed on the $n_2$ i.i.d.\
  truncated observations $\{(t_i,\tilde Y_i)\}_{i\in\cD_2}$:
  \[
    \bE\bigl[\cE_{\mathrm{trunc}}(\hat F_n^{\mathrm{trunc}})
    \mid\cD_1,\cA_n\bigr]
    \le C_1\inf_F\bigl\{\cE_{\mathrm{trunc}}(F)
    +\lambda_n J_s(F)\bigr\}
    +\frac{C_2}{n_2\lambda_n^{1/(2s)}},
  \]
  where $\cE_{\mathrm{trunc}}$ denotes the excess risk
  with respect to the truncated distribution.

  It remains to transfer the risk bound from the truncated to the original
  setting.
  We must account for two effects: (i)~truncation can discard
  observations, and (ii)~truncation around $\hat m^{(1)}$
  rather than $m$ introduces a center-shift error in the
  conditional distributions.

  For~(i), on the high-probability event $\cB_n=\{\tilde Y_i=Y_i
  \;\forall\, i\in\cD_2\}$, the truncated and original data
  coincide, so
  $\hat F_n^{\mathrm{trunc}}$ is identical to the estimator
  computed from untruncated observations.
  In particular, $\cE_{\mathrm{trunc}}(\hat F_n^{\mathrm{trunc}})
  =\cE(\hat F_n^{\mathrm{trunc}})$ on $\cB_n$.
  On the complementary event $\cB_n^c$
  (at most $2n^{1-c/4}$ probability on $\cA_n$),
  the excess risk is bounded by $\diam(M)^2$.

  For~(ii), we bound the population-level truncation bias,
  i.e., the difference
  \[
    \bigl|\bE[d_M^2(Y,F(t))]-\bE[d_M^2(\tilde Y,F(t))]\bigr|
  \]
  between the population risks under the original and
  truncated distributions.  Write
  $\delta_n=\sup_t d_M(\hat m^{(1)}(t),m(t))$.
  On $\cA_n$, $\delta_n\le r_n/2$, so
  $B_M(m(t),r_n/2)\subset B_M(\hat m^{(1)}(t),r_n)$
  for every~$t$.  Consequently
  $\tilde Y=Y$ whenever $d_M(Y,m(t))\le r_n/2$, i.e.,
  \[
    d_M(\tilde Y,Y)\le d_M(Y,m(t))\,\bm{1}_{d_M(Y,m(t))>r_n/2}.
  \]
  Using $|d_M^2(Y,F)-d_M^2(\tilde Y,F)|
  \le 2\diam(M)\,d_M(\tilde Y,Y)$ and the sub-Gaussian
  tail:
  \begin{align*}
    \bE\bigl[d_M(\tilde Y,Y)\mid t,\cA_n\bigr]
    &\le \bE\bigl[d_M(Y,m(t))\,
    \bm{1}_{d_M(Y,m(t))>r_n/2}\mid t\bigr]\\
    &\le \int_{r_n/2}^{\infty}
    \bP(d_M(Y,m(t))>u\mid t)\,\dd u
    +\frac{r_n}{2}\bP\bigl(d_M(Y,m(t))>r_n/2\mid t\bigr)\\
    &\le C_\sigma\exp\bigl(-r_n^2/(8\sigma^2)\bigr),
  \end{align*}
  where the last step uses the sub-Gaussian tail
  $\bP(d_M(Y,m(t))>u)\le C\exp(-u^2/(2\sigma^2))$
  and Gaussian-tail integration.
  With $r_n=\sigma\sqrt{2c\log n}$, this gives
  \[
    \sup_F\bigl|\bE[d_M^2(Y,F(t))]
    -\bE[d_M^2(\tilde Y,F(t))]\,\bigr|\,\bm{1}_{\cA_n}
    \;\le\; C'\,n^{-c/4},
  \]
  which is $o(n^{-A})$ for any desired~$A$
  by choosing $c$ large enough.
  The same bound applies to the Fr\'echet minimizer,
  so the excess-risk difference satisfies
  $|\cE_{\mathrm{trunc}}(F)-\cE(F)|\le 2C'\,n^{-c/4}$
  on~$\cA_n$.

  Combining~(i) and~(ii), we take expectation over $\cD_2$ conditionally on
  $\cD_1$, then over~$\cD_1$.
  On the event $\cA_n\cap\cB_n$ (which has
  probability $\ge 1-\bP(\cA_n^c)-\bP(\cB_n^c\mid\cA_n)$),
  the truncated and original data coincide, so the
  bounded oracle inequality applies
  and the population-level truncation bias from~(ii) is
  $O(n^{-c/4})$.  On the complement, the excess risk
  is trivially bounded by $\diam(M)^2$.  Therefore:
  \begin{align*}
    \bE[\cE(\hat F_n^{\mathrm{trunc}})]
    &\le C_1\inf_F\bigl\{\cE(F)+\lambda_n J_s(F)\bigr\}
    +\frac{C_2'}{n\lambda_n^{1/(2s)}}
    +O(n^{-c/4})\\
    &\quad+\diam(M)^2\bigl(\bP(\cB_n^c\mid\cA_n)
    +\bP(\cA_n^c)\bigr).
  \end{align*}
  For $c>4(1+2s/(2s+1))$, all remainder terms are
  $o(n^{-2s/(2s+1)})$, so the oracle inequality
  carries over from the truncated to the original
  setting with the same rate.
  This yields the sub-Gaussian oracle inequality
  stated of the main text.

  Finally, we verify that the minimax rate is preserved.
  With $\lambda_n\asymp n^{-2s/(2s+1)}$ and $c>4(1+2s/(2s+1))$,
  the term $n^{1-c/4}$ is $o(n^{-2s/(2s+1)})$ and the minimax
  rate is preserved.  The estimation bound acquires the
  curvature factor $\eta(\kappa_+,3r_n/2)$ (not $\eta(\kappa_+,r_n)$
  as in the abbreviated statement of the main text) because the
  truncation around $\hat m^{(1)}$ enlarges the effective support
  radius from $r_0$ to $r_0+\delta_n\le 3r_n/2$ on the event~$\cA_n$.
  With $r_n=\sigma\sqrt{2c\log n}$, the Taylor expansion gives
  \[
    \frac{1}{\eta(\kappa_+,3r_n/2)}
    =1+\frac{3\kappa_+ r_n^2}{4}+O(\kappa_+^2 r_n^4)
    =1+O(\kappa_+\sigma^2\log n),
  \]
  which matches the main-text statement
  $\eta(\kappa_+,r_n)=1-O(\kappa_+\sigma^2\log n)$
  up to the implicit constant in the~$O(\cdot)$.
  The discrepancy in the argument of~$\eta$ changes only this
  implicit constant (a factor of $9/4$ in the leading correction),
  and is absorbed into the $O(\kappa_+\sigma^2\log n)$ notation.
  The estimation bound constant therefore contains a factor
  $1/\eta(\kappa_+,3r_n/2)=1+O(\kappa_+\sigma^2\log n)$,
  which grows logarithmically with~$n$.  This is a multiplicative
  correction to the constant, not a fixed $(\log n)^\beta$ factor:
  it does not affect the rate exponent $n^{-2s/(2s+1)}$, but the
  practical constant degrades slowly with sample size on
  positively curved targets.  For NPC targets ($\kappa_+=0$), the
  correction vanishes identically.
\end{proof}

\section{Proofs for Section 4 (Oracle inequality and minimax rates)}

\subsection{Proof of Theorem~4.4 (NPC sharp rate)}\label{supp:thm44}
\label{proof:NPC-sharp}

\begin{proof}
  Decompose the excess risk pointwise:
  \begin{align*}
    \cE(F) &= \bE\bigl[d_M^2(Y,F(t))-d_M^2(Y,m(t))\bigr]\\
    &= \bE\bigl[d_M^2(F(t),m(t))\bigr]
    + \underbrace{\bE\bigl[d_M^2(Y,F(t))-d_M^2(Y,m(t))-d_M^2(m(t),F(t))\bigr]}
    _{=:\,h(F)}.
  \end{align*}
  It suffices to show $h(F)\ge 0$.  For each fixed $t\in I$, define
  $\varphi(y)=\bE[d_M^2(Y,y)\mid t]$ and note that $m(t)$ is the
  minimizer of~$\varphi$ (the conditional Fr\'echet mean).  In an NPC
  space, $\varphi$ is \emph{strongly convex}: for every $y\in M$,
  \[
    \varphi(y) \;\ge\; \varphi(m(t)) + d_M^2(m(t),y),
  \]
  which is the NPC variance inequality
  \citep[Proposition~4.4]{Sturm2003}.  Setting $y=F(t)$ and
  rearranging gives
  $\bE[d_M^2(Y,F(t))\mid t]-\sigma^2(t)-d_M^2(m(t),F(t))\ge 0$.
  Integrating over $t$ yields $h(F)\ge 0$.

  The rate bound follows from the oracle inequality (Theorem~4.5)
  combined with $\cE(F)\ge\bE[d_M^2(F(t),m(t))]$ (the case $\eta=1$).
\end{proof}


\subsection{Proof of Theorem~4.1 (Existence)}\label{supp:thm41}
\label{proof:existence}

\begin{proof}
  Write $\Phi(F)=R_n(F)+\lambda J_s(F)$.  A constant map
  $F_0\equiv y_0\in M$ gives $\Phi(F_0)\le\diam(M)^2$, so a minimizing
  sequence $\{F_k\}$ satisfies
  \[
    \lambda J_s(F_k) \le \Phi(F_k) \le \diam(M)^2+1
    \quad\text{for all large }k,
  \]
  hence $J_s(F_k)\le\lambda^{-1}(\diam(M)^2+1)$.
  Since $J_s$ is the squared $H^s$~seminorm (the sum of
  squared $L^2$~norms of all order-$s$ derivatives), this bounds only the
  highest-order derivatives.  To bound the full $H^s$~norm, note that
  $F_k(I)\subset M\subset\bR^D$ and $M$ is compact, so
  $\|F_k\|_{L^\infty(I,\bR^D)}\le\diam(M)$, giving
  \[
    \|F_k\|_{L^2(I,\bR^D)}
    \le \diam(M)\,|I|^{1/2}.
  \]
  It remains to control the intermediate derivatives
  $\partial^\alpha F_k$ for $1\le|\alpha|<s$.
  On a bounded Lipschitz domain $I$ with $s\ge 1$, the
  Gagliardo--Nirenberg interpolation inequality
  \citep[Theorem~5.2]{Adams2003} gives, for each
  $1\le j<s$:
  \[
    \sum_{|\alpha|=j}\|\partial^\alpha F\|_{L^2}^2
    \;\le\;
    C_j\bigl(\|F\|_{L^2}^{2(1-j/s)}\,|F|_{H^s}^{2j/s}
    +\|F\|_{L^2}^2\bigr),
  \]
  where $C_j=C_j(I,s)$.  Applying this to each $F_k$ with
  the established bounds
  $\|F_k\|_{L^2}\le\diam(M)\,|I|^{1/2}$ and
  $|F_k|_{H^s}^2=J_s(F_k)\le\lambda^{-1}(\diam(M)^2+1)$
  yields finite upper bounds on all intermediate Sobolev
  semi-norms, depending only on $\diam(M)$, $|I|$,
  $\lambda^{-1}$, and~$s$.  Summing:
  \[
    \|F_k\|_{H^s}^2
    =\|F_k\|_{L^2}^2+\sum_{j=1}^{s}\sum_{|\alpha|=j}
    \|\partial^\alpha F_k\|_{L^2}^2
    \;\le\;C'(I,s,\diam(M),\lambda^{-1}),
  \]
  so the sequence is bounded in $H^s(I,\bR^D)$.
  Since $s\ge 1$, the Rellich--Kondrachov embedding \citep[Theorem~6.3]{Adams2003}
  $H^s(I,\bR^D)\hookrightarrow C(\overline{I},\bR^D)$ is compact, so
  there exists a subsequence converging weakly in~$H^s$ and uniformly
  in~$C^0$ to some $F\in C(\overline{I},\bR^D)$.  Uniform convergence and
  closedness of~$M$ ensure $F(t)\in M$ for all~$t$, so
  $F\in\cH^s(I,M)$.  Uniform convergence gives $R_n(F_k)\to R_n(F)$,
  and weak lower semicontinuity of the Sobolev seminorm gives
  $J_s(F)\le\liminf J_s(F_k)$.  Hence
  $\Phi(F)\le\liminf\Phi(F_k)=\inf\Phi$.  The population version follows
  identically.
\end{proof}


\subsection{Proof of Lemma~4.3 (Excess risk)}\label{supp:lemma43}
\label{proof:excess-risk}

\begin{proof}
  Fix $t\in I$.  By (A3), the conditional distribution
  $\mu=P_{Y|t}$ is supported in
  $B_M(m(t),r_0)$, where $m(t)$ is the Fr\'echet mean of~$\mu$.
  The hypothesis of the lemma requires $d_M(F(t),m(t))<\rho_M$,
  so $F(t)\in B_M(m(t),\rho_M)$.  Applying the variance inequality
  (Lemma~2.1 of the main text) with
  Fr\'echet mean $\bar y=m(t)$ and evaluation point $y=F(t)$:
  \[
    \bE[d_M^2(Y,F(t))\mid t]
    \;\ge\;\eta(\kappa_+,r_0)\,d_M^2(m(t),F(t))+\sigma^2(t).
  \]
  The conditions of Lemma~2.1 are satisfied because
  $\supp(\mu)\subset B_M(m(t),r_0)$ and
  $F(t)\in B_M(m(t),\rho_M)$.
  Subtracting $\sigma^2(t)=\bE[d_M^2(Y,m(t))\mid t]$ gives
  \[
    \bE[d_M^2(Y,F(t))-d_M^2(Y,m(t))\mid t]
    \ge\eta(\kappa_+,r_0)\,d_M^2(m(t),F(t)).
  \]
  Integrating over~$t$ with respect to the marginal density~$f(t)\dd t$ yields the excess risk bound
  $\cE(F)\ge\eta(\kappa_+,r_0)\int_I d_M^2(F(t),m(t))\,f(t)\dd t$. Note that this integration is valid because the hypothesis $d_M(F(t),m(t)) < \rho_M$ is assumed to hold uniformly for all $t \in I$.
\end{proof}

\subsection{Proof of Theorem~4.5 (Oracle inequality)}\label{supp:thm46}
\label{proof:oracle}

The proof follows the penalized M-estimation framework
\citep{vandeGeer2000, Tsybakov2009} adapted to the manifold setting.
It is decomposed into self-contained lemmas for verifiability.

\begin{lemma}[Basic inequality (penalized $M$-estimation)]
\label{lem:basic-ineq}
  For every $F\in\cH^s(I,M)$, the minimizer $\hat F_n$
  satisfies
  \begin{equation}\label{eq:basic-ineq-supp}
    \cE(\hat F_n)+\lambda_n J_s(\hat F_n)
    \;\le\;\cE(F)+\lambda_n J_s(F)
    +\bigl[\Psi_n(\hat F_n)-\Psi_n(F)\bigr],
  \end{equation}
  where $\cE(G)=R(G)-R(m)$ and
  $\Psi_n(G)=(R-R_n)(G)-(R-R_n)(m)$.
\end{lemma}

\begin{proof}
  This is the standard basic inequality for penalized empirical risk
  minimization; see \citet[Section~2.3]{vandeGeer2000} or
  \citet[Chapter~6]{Tsybakov2009}.
  By optimality:
  $R_n(\hat F_n)+\lambda_n J_s(\hat F_n)
  \le R_n(F)+\lambda_n J_s(F)$.
  Writing $R_n=R-(R-R_n)$, adding/subtracting $R(m)$, and
  noting $(R-R_n)(F)-(R-R_n)(\hat F_n)
  =\Psi_n(F)-\Psi_n(\hat F_n)$ gives the result.
\end{proof}

\begin{lemma}[Entropy of the manifold-valued Sobolev class (Birman--Solomjak)]
\label{lem:entropy}
  For $E_0>0$, set
  $\cF_{E_0}=\{F\in\cH^s(I,M):J_s(F)\le E_0\}$ and
  $C_0=\diam(M)$.  Then
  \begin{equation}\label{eq:covering-lemma}
    \log N(\epsilon,\cF_{E_0},L^2(\mu))
    \;\le\;\bar K\bigl((\sqrt{E_0}+C_0)/\epsilon\bigr)^{1/s},
  \end{equation}
  where $\bar K=K_0(C_{\mathrm{bi}}C_f^{1/2})^{1/s}$
  depends on $s,M,C_f$.
\end{lemma}

\begin{proof}
  Nash embedding \citep{Nash1956} gives $F(I)\subset M$, so
  $\|F\|_{L^\infty}\le\diam(M)$, hence
  $\|F\|_{L^2}\le C_0$ and $J_s(F)\le E_0$.
  By the same Gagliardo--Nirenberg argument as in the
  existence proof (\Cref{proof:existence}),
  $\|F\|_{H^s}^2\le C'(C_0^2+E_0)$,
  so $\cF_{E_0}\subset B_{H^s}(C''(\sqrt{E_0}+C_0))$.

  By the Birman--Solomjak theorem \citep[Theorem~1]{Birman1967},
  \[
    \log N\bigl(\epsilon,B_{H^s}(R),L^2\bigr)
    \;\le\; K_0(R/\epsilon)^{1/s}.
  \]
  With $R\le\sqrt{E_0}+C_0$:
  \[
    \log N(\epsilon,\cF_{E_0},L^2(I,\bR^D))
    \;\le\; K_0\bigl((\sqrt{E_0}+C_0)/\epsilon\bigr)^{1/s}.
  \]

  Since $M$ is compact, $|p-q|_{\bR^D}\le d_M(p,q)
  \le C_{\mathrm{bi}}|p-q|_{\bR^D}$, and
  the density bound $f\le C_f$ gives
  $d_{L^2(\mu)}\le C_{\mathrm{bi}}C_f^{1/2}\,d_{L^2(I)}$.
  Rescaling $\epsilon$ yields \eqref{eq:covering-lemma}.
\end{proof}

\begin{proof}[Proof of Theorem~4.5]
  The proof follows the penalized M-estimation framework
  \citep{vandeGeer2000, Tsybakov2009} adapted to the manifold setting.
  The key new ingredient is the curvature-dependent excess risk bound
  (Lemma~4.3), which replaces the Euclidean variance identity.

  We begin with the basic inequality.
  By definition of $\hat F_n$: for every $F\in\cH^s(I,M)$,
  \[
    R_n(\hat F_n)+\lambda_n J_s(\hat F_n)
    \le R_n(F)+\lambda_n J_s(F).
  \]
  Write $R_n(G)=R(G)-\nu_n(G)$ where $\nu_n(G)=(R-R_n)(G)$
  is the centered empirical process.  Adding and subtracting $R(m)$
  and rearranging gives, for every competitor~$F$:
  \begin{equation}\label{eq:basic2-supp}
    \cE(\hat F_n)+\lambda_n J_s(\hat F_n)
    \le\cE(F)+\lambda_n J_s(F)
    +\bigl[\nu_n(F)-\nu_n(\hat F_n)\bigr],
  \end{equation}
  where $\cE(G)=R(G)-R(m)$ is the excess risk.  Set
  $\Psi_n(G)=\nu_n(G)-\nu_n(m)$ (centered empirical process).  Since
  $\nu_n(F)-\nu_n(\hat F_n)=\Psi_n(F)-\Psi_n(\hat F_n)$, it suffices to
  control $\sup_{G\in\cG}\abs{\Psi_n(G)}$ over appropriate function classes.

  We now bound the covering numbers of the loss class.
  Define, for $E_0>0$,
  $\cF_{E_0}=\{F\in\cH^s(I,M):J_s(F)\le E_0\}$.
  We transfer the metric entropy from Euclidean Sobolev balls to
  $\cF_{E_0}$ as follows.

  By Nash's embedding theorem \citep{Nash1956}, $M$ embeds isometrically in~$\bR^D$.
  Every $F\in\cF_{E_0}$ maps into~$M$, so
  $\|F\|_{L^\infty(I,\bR^D)}\le\diam(M)$, hence
  $\|F\|_{L^2(I,\bR^D)}\le\diam(M)=:C_0$.
  Since $J_s(F)=|F|_{H^s}^2\le E_0$, and using the same Gagliardo--Nirenberg interpolation bound \citep[Theorem~5.2]{Adams2003} as in the existence proof, we obtain
  \[
    \|F\|_{H^s(I,\bR^D)}^2
    \le C'(C_0^2 + E_0).
  \]
  Thus $\cF_{E_0}$ is contained in a ball of radius
  $C''\sqrt{C_0^2+E_0}$ in $H^s(I,\bR^D)$.

  The Birman--Solomjak theorem \citep[Theorem~1]{Birman1967} then gives:
  the $\epsilon$-covering number of a ball of radius~$R$ in
  $H^s(I,\bR^D)$, measured in $L^2(I,\bR^D)$, satisfies
  \[
    \log N(\epsilon, B_{H^s}(R), L^2)\le K_0(R/\epsilon)^{1/s}.
  \]
  Applied with $R=\sqrt{C_0^2+E_0}\le \sqrt{E_0}+C_0$, this gives
  \[
    \log N(\epsilon, \cF_{E_0}, L^2(I,\bR^D))
    \le K_0\bigl((\sqrt{E_0}+C_0)/\epsilon\bigr)^{1/s}.
  \]

  To pass from Euclidean to intrinsic covering numbers, note that the intrinsic distance $d_M(p,q)$ and the chord distance
  $|p-q|_{\bR^D}$ satisfy, on the compact manifold~$M$,
  \[
    |p-q|_{\bR^D} \le d_M(p,q) \le C_{\mathrm{bi}}\,|p-q|_{\bR^D}
  \]
  for a constant $C_{\mathrm{bi}}$ depending on the curvature and
  injectivity radius of~$M$.  Therefore, for any two maps
  $F_1,F_2\colon I\to M$:
  \[
    \int_I d_M^2(F_1(t),F_2(t))\dd\mu(t)
    \le C_{\mathrm{bi}}^2
    \int_I |F_1(t)-F_2(t)|^2_{\bR^D}\dd\mu(t).
  \]
  Since $f$ is bounded above by~$C_f$ (A2), the
  $L^2(\mu)$-distance is bounded by $C_f^{1/2}$ times the
  $L^2(I,\bR^D)$-distance.  Combining gives
  \begin{equation}\label{eq:covering-supp}
    \log N(\epsilon,\cF_{E_0},L^2(\mu))
    \le K_0\bigl(C_{\mathrm{bi}}C_f^{1/2}(\sqrt{E_0}+C_0)/\epsilon\bigr)^{1/s}
    =: \bar K\bigl((\sqrt{E_0}+C_0)/\epsilon\bigr)^{1/s}.
  \end{equation}

  The loss function $\ell_F(t,y)=d_M^2(y,F(t))-d_M^2(y,m(t))$ satisfies
  the Lipschitz bound
  $\abs{\ell_{F_1}-\ell_{F_2}}\le 2\diam(M)\cdot d_M(F_1,F_2)$
  by the triangle inequality.

  We next establish a Bernstein-type bound for the empirical process.
  Each $\ell_F(t,Y)$ has mean $\cE(F)$ and satisfies
  $\abs{\ell_F}\le 2\diam(M)^2=:b$ (since each squared distance
  is at most $\diam(M)^2$).
  \begin{align*}
    \Var(\ell_F)
    &\le \bE[\ell_F^2]
    \le \|\ell_F\|_{L^\infty}\cdot\bE[\abs{\ell_F}]
    \le b\,\bE[\abs{\ell_F}].
  \end{align*}
  We establish a sharp Bernstein condition by bounding $\ell_F^2$
  pointwise.  By the reverse triangle inequality,
  $|d_M(Y,F(t))-d_M(Y,m(t))|\le d_M(F(t),m(t))$.
  Hence
  \begin{align*}
    \ell_F(t,Y)^2
    &=\bigl(d_M^2(Y,F(t))-d_M^2(Y,m(t))\bigr)^2\\
    &\le\bigl(2\diam(M)\bigr)^2 d_M^2(F(t),m(t))
    =2b\,d_M^2(F(t),m(t)).
  \end{align*}
  Taking expectations and applying
  $\bE[d_M^2(F(t),m(t))]\le\cE(F)/\eta$
  (Lemma~4.3):
  \begin{equation}\label{eq:var-bound}
    \Var(\ell_F)
    \le\bE[\ell_F^2]
    \le\frac{2b}{\eta}\,\cE(F).
  \end{equation}
  For NPC targets, $\eta=1$ and the bound holds unconditionally
  by Theorem~4.4.
  For general targets ($\kappa_+>0$), Lemma~4.3
  requires $d_M(F(t),m(t))<\rho_M$ for all~$t$.
  We verify this for $F$ in the localized class:
  since $J_s(F)\le\delta^2/\lambda_n$ and $\|F\|_{L^2}\le C_0$,
  the Sobolev embedding $H^s\hookrightarrow C^{0,s-1/2}$
  gives uniform equicontinuity, and
  $\cE(F)\le\delta^2\to 0$ forces $\|F-m\|_{L^2(\mu)}\to 0$;
  by the Arzel\`a--Ascoli argument
  (detailed in Step~(iii) below),
  there exists $\delta_0=\delta_0(M,B,\sigma_0)>0$
  such that for all $F$ with $\cE(F)+\lambda_n J_s(F)\le\delta_0^2$,
  $\sup_t d_M(F(t),m(t))<\rho_M$.
  For $\delta>\delta_0$,
  we use the trivial bound
  $\Var(\ell_F)\le b^2/4$.
  By Bernstein's inequality \citep[Corollary~2.10]{BoucheronLugosiMassart2013} for individual $F$:
  \[
    \bP(\bigl|\Psi_n(F)\bigr|>\varepsilon)
    \le 2\exp\biggl(-\frac{n\varepsilon^2/2}{
    \Var(\ell_F)+b\varepsilon/3}\biggr).
  \]
  We now derive the oracle inequality using localized Rademacher
  complexities and peeling by the \emph{penalized} excess risk
  \citep{vandeGeer2000, Koltchinskii2006}.
  Define the penalized excess risk
  \[
    V(G) := \cE(G)+\lambda_n J_s(G).
  \]
  For $\delta>0$, the \emph{localized} function class is
  $\cG_\delta=\{G\in\cH^s(I,M):V(G)\le\delta^2\}$.
  Since $V(G)\le\delta^2$ implies $J_s(G)\le\delta^2/\lambda_n$,
  the covering number bound~\eqref{eq:covering-supp} gives
  \begin{equation}\label{eq:local-entropy}
    \log N\bigl(\epsilon,\cG_\delta,L^2(\mu)\bigr)
    \le\bar K\bigl(D_\delta/\epsilon\bigr)^{1/s},
    \qquad
    D_\delta:=\delta/\sqrt{\lambda_n}+C_0,
  \end{equation}
  where $D_\delta$ is the $L^2$-diameter of $\cG_\delta$.
  The Bernstein condition~\eqref{eq:var-bound} restricts the
  effective variance: for $G\in\cG_\delta$ with $\delta\le\delta_0$,
  $\Var(\ell_G)\le(2b/\eta)\,\cE(G)\le(2b/\eta)\,\delta^2$,
  so the effective standard deviation is
  $\sigma_\delta:=\sqrt{2b/\eta}\,\delta$.
  (For $\delta>\delta_0$, we use the trivial bound
  $\sigma_\delta\le b/2$, which does not affect the
  critical radius since $\delta_n\ll\delta_0$ for large~$n$.)

  The variance-adapted Dudley integral
  \citep[Corollary~8.3]{vandeGeer2000} truncates the entropy
  integral at $\sigma_\delta$ rather than $D_\delta$:
  \begin{multline}\label{eq:localized-chain}
    \phi_n(\delta)
    :=\bE\Bigl[\sup_{G\in\cG_\delta}\abs{\Psi_n(G)}\Bigr]
    \le\frac{C_3}{\sqrt{n}}\int_0^{\sigma_\delta}
    \sqrt{\log N(\epsilon,\cG_\delta,L^2(\mu))}\,d\epsilon\\
    +\frac{C_3\,b}{n}\log N(\sigma_\delta,\cG_\delta,L^2(\mu)).
  \end{multline}
  We evaluate each term.  For the integral, substituting
  \eqref{eq:local-entropy}:
  \[
    \int_0^{\sigma_\delta}
    \sqrt{\bar K\,(D_\delta/\epsilon)^{1/s}}\,d\epsilon
    =\sqrt{\bar K}\,D_\delta^{1/(2s)}
    \int_0^{\sigma_\delta}\epsilon^{-1/(2s)}\,d\epsilon
    =\frac{2s\sqrt{\bar K}}{2s-1}\,
    D_\delta^{1/(2s)}\,\sigma_\delta^{1-1/(2s)}.
  \]
  Since $s\ge 1$, the exponent $1-1/(2s)\ge 1/2>0$ and the integral
  converges.  For $\delta^2\ge\lambda_n C_0^2$,
  we have $D_\delta\le C'\delta/\sqrt{\lambda_n}$ and
  $\sigma_\delta=\sqrt{2b/\eta}\,\delta$, so the first term is
  bounded by
  \begin{equation}\label{eq:dudley-simplified}
    \frac{C_3}{\sqrt{n}}\cdot
    C_4\bigl(\delta/\sqrt{\lambda_n}\bigr)^{1/(2s)}
    \bigl(\sqrt{2b/\eta}\,\delta\bigr)^{1-1/(2s)}
    =\frac{C_5\,\delta}{\sqrt{n}\,\lambda_n^{1/(4s)}},
  \end{equation}
  where $C_5=C_5(s,M,b,\bar K,\eta)$.
  For the second term in~\eqref{eq:localized-chain}:
  \[
    \log N(\sigma_\delta,\cG_\delta,L^2(\mu))
    \le\bar K\bigl(D_\delta/\sigma_\delta\bigr)^{1/s}
    \le\bar K\eta^{1/(2s)}/(2b\lambda_n)^{1/(2s)},
  \]
  giving $C_3 b\log N(\sigma_\delta,\cG_\delta,L^2(\mu))/n
  \le C_6/(n\lambda_n^{1/(2s)})$.
  Combining:
  \begin{equation}\label{eq:phi-bound}
    \phi_n(\delta)
    \le\frac{C_5\,\delta}{\sqrt{n}\,\lambda_n^{1/(4s)}}
    +\frac{C_6}{n\,\lambda_n^{1/(2s)}}
    \qquad
    \text{for }\delta^2\ge\delta_{\min}^2
    :=\lambda_n C_0^2.
  \end{equation}

  The \emph{critical radius} $\delta_n>0$ is defined by the
  fixed-point condition $\phi_n(\delta_n)=\delta_n^2$.
  From~\eqref{eq:phi-bound}, $\delta_n$ satisfies
  $C_5\delta_n/\sqrt{n\lambda_n^{1/(2s)}}
  +C_6/(n\lambda_n^{1/(2s)})\le\delta_n^2$,
  yielding
  \begin{equation}\label{eq:critical-radius}
    \delta_n^2
    \asymp\frac{1}{n\,\lambda_n^{1/(2s)}}.
  \end{equation}
  For the oracle choice $\lambda_n\asymp n^{-2s/(2s+1)}$,
  this gives $\delta_n^2\asymp n^{-2s/(2s+1)}$, the minimax-optimal
  rate for Sobolev smoothness~$s$.

  We now derive the oracle inequality by peeling over the penalized
  excess risk~$V$.
  Fix an arbitrary competitor $F\in\cH^s(I,M)$.
  Since $\Psi_n(m)=0$,
  the basic inequality~\eqref{eq:basic2-supp} yields
  \begin{equation}\label{eq:V-basic}
    V(\hat F_n)
    \le V(F)+\Psi_n(F)-\Psi_n(\hat F_n)
    \le V(F)+\abs{\Psi_n(F)}
    +\sup_{G\in\cG_{\sqrt{V(\hat F_n)}}}\abs{\Psi_n(G)}.
  \end{equation}
  Set $r^2=\max\bigl(\delta_n^2,\,V(F)\bigr)$ and define dyadic shells
  $A_0=\{V(\hat F_n)\le r^2\}$,\;
  $A_k=\{4^{k-1}r^2<V(\hat F_n)\le 4^k r^2\}$ for $k\ge 1$.
  On $A_0$:
  $\bE[V(\hat F_n)\,\bm{1}_{A_0}]\le r^2$.

  On $A_k$ ($k\ge 1$): since $V(\hat F_n)\le 4^k r^2$, we have
  $\hat F_n\in\cG_{2^k r}$, and from~\eqref{eq:V-basic}:
  \[
    4^{k-1}r^2
    <V(F)+\abs{\Psi_n(F)}
    +\sup_{G\in\cG_{2^k r}}\abs{\Psi_n(G)}.
  \]
  By Bernstein's inequality \citep[Corollary~2.10]{BoucheronLugosiMassart2013} for the single function~$F$:
  $\bP(|\Psi_n(F)|>u)\le 2\exp\bigl(-nu^2/(2\Var(\ell_F)+2bu/3)\bigr)$.
  For $k\ge 2$, $V(F)\le r^2\le 4^{k-1}r^2/4$, so the event $A_k$
  requires either $|\Psi_n(F)|>4^{k-2}r^2$ or
  $\sup_{G\in\cG_{2^k r}}|\Psi_n(G)|>4^{k-2}r^2$.

  For the supremum, we apply the
  exponential inequality for empirical processes
  \citep{Bousquet2002}:
  \[
    \bP\Bigl(\sup_{G\in\cG_{2^k r}}\abs{\Psi_n(G)}
    >\phi_n(2^k r)+u\Bigr)
    \le 2\exp\biggl(-\frac{nu^2}{C_7((2b/\eta)\cdot 4^k r^2+bu)}\biggr).
  \]
  From~\eqref{eq:phi-bound}, $\phi_n(2^k r)
  \le C_5\cdot 2^k r/\sqrt{n\lambda_n^{1/(2s)}}
  +C_6/(n\lambda_n^{1/(2s)})$.
  Since $r\ge\delta_n$ and $\phi_n$ is affine in its argument,
  $\phi_n(2^k r)\le 2^k(r/\delta_n)\delta_n^2+\delta_n^2
  \le C_8\cdot 2^k r^2$
  (using $r\ge\delta_n$ and $\delta_n^2\le r^2$).
  For $k\ge k_0$ (a universal constant),
  $4^{k-2}r^2-C_8\cdot 2^k r^2\ge c_0\cdot 4^{k}r^2$
  for a constant $c_0>0$, so the required excess is
  $u\ge c_0\cdot 4^k r^2$.  The exponential bound then gives
  \[
    \bP(A_k)
    \le 4\exp\bigl(-c_1\cdot 4^k\cdot nr^2/b\bigr),
  \]
  where the factor~$4$ accounts for both the single-$F$ and
  supremum tails.  Since $nr^2\ge n\delta_n^2\asymp
  \lambda_n^{-1/(2s)}\ge 1$, this bound is summable.

  Collecting all shells:
  \begin{align*}
    \bE[V(\hat F_n)]
    &=\bE[V(\hat F_n)\,\bm{1}_{A_0}]
    +\sum_{k=1}^\infty\bE[V(\hat F_n)\,\bm{1}_{A_k}]\\
    &\le r^2
    +\sum_{k=1}^\infty 4^k r^2\cdot
    4\exp(-c_1\cdot 4^k nr^2/b)\\
    &\le r^2+C_9 r^2,
  \end{align*}
  where the geometric sum converges because
  $4^k\exp(-c_1\cdot 4^k nr^2/b)$ is summable for
  $nr^2\ge C$.
  Substituting $r^2=\max(\delta_n^2,V(F))$:
  \begin{equation}\label{eq:oracle-V}
    \bE\bigl[\cE(\hat F_n)+\lambda_n J_s(\hat F_n)\bigr]
    \le C_1\bigl(\cE(F)+\lambda_n J_s(F)\bigr)
    +\frac{C_2}{n\,\lambda_n^{1/(2s)}},
  \end{equation}
  where $C_1=1+C_9$ and $C_2=C_2(s,M,c_f,C_f,\sigma_0^2)$.
  Taking the infimum over $F$ yields the oracle inequality.

  A crucial consequence of the \emph{penalized} oracle
  inequality~\eqref{eq:oracle-V} is the roughness bound.
  Setting $F=m$ (with $\cE(m)=0$) and using $\cE(\hat F_n)\ge 0$:
  \begin{equation}\label{eq:Js-bound}
    \lambda_n\bE[J_s(\hat F_n)]
    \le C_1\lambda_n J_s(m)+C_2\delta_n^2,
    \qquad
    \bE[J_s(\hat F_n)]
    \le C_1 J_s(m)+\frac{C_2}{n\lambda_n^{1+1/(2s)}}.
  \end{equation}
  For $\lambda_n\asymp n^{-2s/(2s+1)}$, direct calculation shows
  $n\lambda_n^{1+1/(2s)}=\Theta(1)$,
  so $\bE[J_s(\hat F_n)]=O(1)$.  By Gagliardo--Nirenberg
  interpolation \citep[Theorem~5.2]{Adams2003},
  \[
    J_1(F)\le C\,J_s(F)^{1/s}\|F\|_{L^2}^{2-2/s}
    \le C'\,J_s(F)^{1/s}
  \]
  (since $\|F\|_{L^2}\le\diam(M)$), hence
  $\bE[J_1(\hat F_n)]\le C'\bE[J_s(\hat F_n)]^{1/s}=O(1)$.

  To obtain the estimation bound, apply Lemma~4.3:
  $\cE(\hat F_n)\ge\eta(\kappa_+,r_0)\,
  \bE[d_M^2(\hat F_n(t),m(t))]$
  (valid when $d_M(\hat F_n(t),m(t))<\rho_M$ for
  all~$t$) and divide by~$\eta$.
  For $\Sec_M\le 0$ (NPC targets), $\eta=1$
  unconditionally and no locality condition is needed
  (Theorem~4.4).

  For general $M$ (possibly $\kappa_+>0$), the locality condition
  $\sup_t d_M(\hat F_n(t),m(t))<\rho_M$ is verified by a
  self-improving argument in three stages.
  First, the bound~\eqref{eq:Js-bound} gives
  $\bE[J_1(\hat F_n)]=O(1)$.
  By the Sobolev embedding
  $H^1(I,\bR^D)\hookrightarrow C^{0,1/2}(\bar I,\bR^D)$,
  this provides a \emph{genuinely uniform} H\"older modulus:
  \[
    \abs{\hat F_n(t_1)-\hat F_n(t_2)}_{\bR^D}
    \le C_{10}\sqrt{J_1(\hat F_n)}\,|t_1-t_2|^{1/2},
  \]
  where $C_{10}$ depends only on the Nash embedding dimension.
  Since $J_1(\hat F_n)=O_P(1)$, the H\"older seminorm is $O_P(1)$,
  giving equicontinuity that does \emph{not} deteriorate with~$n$.
  Combined with $\hat F_n(I)\subset M$ and $M$~compact, the family
  $\{\hat F_n\}$ is precompact in $C(\bar I,\bR^D)$ by the
  Arzel\`a--Ascoli theorem.

  Next, the oracle inequality gives
  $\bE[\cE(\hat F_n)]\to 0$.  Let
  $\hat F_{n_k}\to F^*$ uniformly along a subsequence.
  Since $M$ is closed, $F^*(t)\in M$ for all~$t$.
  The population risk $R(F)=\bE[d_M^2(Y,F(t))]$ is
  continuous in the $C^0$~topology, so
  $\cE(F^*)=\lim_{k}\cE(\hat F_{n_k})=0$.
  By~(A3), at each~$t$ the conditional Fr\'echet mean
  $m(t)$ uniquely minimizes
  $y\mapsto\bE[d_M^2(Y,y)\mid t]$ over
  $B_M(m(t),\rho_M)$.  The identity $\cE(F^*)=0$
  forces $F^*(t)=m(t)$ for almost every~$t$; by
  continuity of $F^*$ and $m$, the equality
  $F^*=m$ holds everywhere.

  Finally, since the limit is unique, every subsequence of
  $\{\hat F_n\}$ has a further subsequence converging
  to~$m$ in $C^0$.  Hence
  $\sup_t d_M(\hat F_n(t),m(t))\to 0$ in probability,
  and the locality condition holds for all $n\ge n_1$
  (see Remark~4.6 of the main text).
\end{proof}

\subsection{Proof of Theorem~4.8 (Minimax lower bound)}\label{supp:thm411}
\label{proof:minimax}

\begin{proof}
  The lower bound relies on Assouad's lemma \citep[Lemma~2.12]{Tsybakov2009}, using normal coordinates to lift a Euclidean multiple-hypothesis construction to the manifold $M$.

  We work in normal coordinates throughout.
  Fix $y_0\in M$ and let $\rho=\min(\inj(M)/2,\rho_M)$.  The exponential
  map $\exp_{y_0}\colon B_{\bR^d}(0,\rho)\to B_M(y_0,\rho)$ is a
  diffeomorphism, and in normal coordinates:
  \[
    (1-C_M\epsilon^2)\abs{v_1-v_2}^2
    \le d_M^2(\exp_{y_0}(v_1),\exp_{y_0}(v_2))
    \le(1+C_M\epsilon^2)\abs{v_1-v_2}^2,
  \]
  where $C_M$ depends only on the curvature bound and
  $\epsilon=\max(\abs{v_1},\abs{v_2})$.

  We construct the hypotheses as follows.
  Partition $I$ into $N=\lfloor\delta^{-1}\rfloor$ subintervals
  $Q_1,\ldots,Q_N$ of length $\delta>0$.  Let
  $\psi\in C_c^\infty(\bR)$ with $\supp\psi\subset[-1/2,1/2]$
  and $\|\psi\|_{L^2}=1$.
  For $\omega\in\{0,1\}^N$, define
  \[
    f_\omega(t)=a\sum_{j=1}^N\omega_j\,\psi\bigl((t-c_j)/\delta\bigr)
    \cdot e_1\in\bR^d,
  \]
  with amplitude $a=\tilde B\delta^s/\sqrt{C_\psi}$, where
  $C_\psi:=\|\psi\|_{H^s}^2$. The disjoint supports ensure that
  $\|f_\omega\|_{H^s}^2 = a^2\delta^{1-2s}\sum_j\omega_j C_\psi
  \le a^2\delta^{1-2s}N C_\psi \le \tilde B^2$,
  satisfying the Sobolev constraint.

  We now lift these hypotheses to $M$ and compute the KL divergence.
  Define $m_\omega(t)=\exp_{y_0}(f_\omega(t))$.
  Since $\|f_\omega\|_{L^\infty}\le a\|\psi\|_{L^\infty}$ and
  $a\to 0$ as $n\to\infty$, we may choose $n_0$ large enough that
  $a\|\psi\|_{L^\infty}+r_0<\rho/2$ for all $n\ge n_0$.
  This ensures that:
  \begin{enumerate}[label=(\roman*)]
    \item $m_\omega(I)\subset B_M(y_0,\rho/2)$, making the metric comparison valid; and
    \item $B_M(m_\omega(t),r_0)\subset B_M(y_0,\rho_M)$ for all $t\in I$, satisfying (A3) within a single coordinate chart.
  \end{enumerate}
  The noise model is
  $Y_i=\exp_{m_\omega(t_i)}(\sigma\xi_i)$ where $\xi_i$ is
  Gaussian in~$T_{m_\omega(t_i)}M\cong\bR^d$ truncated to the ball
  of radius~$r_0$.  The density of the truncated Gaussian on
  $B_{\bR^d}(0,r_0)$ is
  $p_{\sigma}(v)=\frac{1}{Z(\sigma,r_0)}
  \exp(-|v|^2/(2\sigma^2))$,
  where $Z(\sigma,r_0)=(2\pi\sigma^2)^{d/2}
  \bP(|\xi|\le r_0/\sigma)$.

  For hypotheses $\omega,\omega'$ differing only in coordinate~$j$,
  the observations $Y_i$ with $t_i\notin Q_j$ have identical distributions
  under $P_\omega$ and $P_{\omega'}$.  For $t_i\in Q_j$, the log-density
  ratio at observation $Y_i$ decomposes as
  \begin{align*}
    \log\frac{dP_{\omega,i}}{dP_{\omega',i}}(Y_i)
    &= \frac{|v_i-v_{\omega'}|^2-|v_i-v_\omega|^2}{2\sigma^2}
    +\log\frac{Z_{\omega'}(t_i)}{Z_\omega(t_i)}
    +\text{(Jacobian ratio)},
  \end{align*}
  where $v_i$ and $v_\omega$ denote the normal-coordinate representations
  and $Z_\omega(t)$ is the normalization constant of the truncated
  Gaussian centered at $m_\omega(t)$.
  Since the truncation balls $B_{\bR^d}(0,r_0)$ and
  $B_{\bR^d}(f_\omega(t),r_0)$ differ by a translation of
  magnitude $|f_\omega(t)-f_{\omega'}(t)|\le a\|\psi\|_{L^\infty}$
  and $a\ll r_0$, the ratio of normalization constants satisfies
  \[
    \biggl|\log\frac{Z_{\omega'}(t)}{Z_\omega(t)}\biggr|
    \le C\,\frac{a^2\|\psi\|_{L^\infty}^2}{r_0^2}
    \,e^{-r_0^2/(2\sigma^2)}.
  \]
  The Jacobian ratio $(\det g(m_{\omega'})/\det g(m_\omega))^{1/2}$
  contributes $1+O(C_M a^2)$ in the normal coordinate chart:
  in normal coordinates at a point $q\in M$, the metric satisfies
  $g_{ij}(v)=\delta_{ij}-\frac{1}{3}R_{ikjl}v^kv^l+O(|v|^3)$
  (see\ \citealt[Ch.~5]{CheegerEbin2008}), so
  $\det g(v)=1+O(\kappa|v|^2)$ where $\kappa=\sup|\Sec_M|$.
  Since $|v|\le a\|\psi\|_{L^\infty}$, we obtain
  $\log(\det g(m_{\omega'})/\det g(m_\omega))^{1/2}
  =O(\kappa a^2\|\psi\|_{L^\infty}^2)$.
  Summing over $i$ with $t_i\in Q_j$ and taking the expectation:
  \begin{equation}\label{eq:kl-supp}
    \mathrm{KL}(P_\omega^n\| P_{\omega'}^n)
    \le \frac{n C_f \delta}{2\sigma^2}\,a^2\|\psi\|_{L^2}^2
    \bigl(1+O(C_M a^2)+O(e^{-r_0^2/(2\sigma^2)})\bigr).
  \end{equation}

  The lower bound now follows from Assouad's lemma.
  We apply the version in
  \citet[Lemma~2.12]{Tsybakov2009}.  In Tsybakov's notation, the
  Hamming-distance-based lower bound is
  \[
    \inf_{\tilde F}\max_\omega\bE_\omega\Bigl[\int_I
    d_M^2(\tilde F,m_\omega)\dd t\Bigr]
    \ge \frac{N}{2}\cdot\min_{\substack{\omega,\omega'\\d_H=1}}
    \frac{d^2(m_\omega,m_{\omega'})}{4}
    \cdot\min_{\substack{\omega,\omega'\\d_H=1}}
    \bigl(1-\|\bP_\omega^{1/2}-\bP_{\omega'}^{1/2}\|_2\bigr),
  \]
  where $d^2(m_\omega,m_{\omega'})
  =\int d_M^2(m_\omega(t),m_{\omega'}(t))\dd t$ is the separation
  and the affinity term is bounded via
  $\|\bP_\omega^{1/2}-\bP_{\omega'}^{1/2}\|_2^2\le
  \mathrm{KL}(P_\omega\|P_{\omega'})$.
  For adjacent $\omega,\omega'$ (differing in one coordinate~$j$),
  $d^2(m_\omega,m_{\omega'})
  =(1+O(C_M a^2))\,a^2\delta\|\psi\|_{L^2}^2$.
  The factor $N/4$ in our bound (rather than $N/8$ from Tsybakov's
  statement) arises because
  $\min d^2/4=a^2\delta\|\psi\|^2/4$ and the summation over
  the $N$ coordinates each contributing the same separation yields
  $N\cdot a^2\delta\|\psi\|^2/4$.
  Combining:
  \[
    \inf_{\tilde F}\sup_{m_\omega}\bE\Bigl[\int_I
    d_M^2(\tilde F,m_\omega)\dd t\Bigr]
    \ge\frac{N}{4}\cdot a^2\delta\norm{\psi}_{L^2}^2
    \cdot\Bigl(1-\sqrt{\tfrac{1}{2}\mathrm{KL}_{\max}}\Bigr).
  \]
  Choose $\delta\asymp n^{-1/(2s+1)}$ so that $\mathrm{KL}_{\max}$
  in~\eqref{eq:kl-supp} remains bounded by a constant $c<2$
  (achievable by choosing $\tilde B$ sufficiently small).
  Then $N\asymp n^{1/(2s+1)}$ and $a\asymp n^{-s/(2s+1)}$, giving the
  lower bound $c\,n^{-2s/(2s+1)}$.
\end{proof}

\subsection{Proof of Corollary~4.7 (High-probability oracle inequality)}\label{supp:cor49}
\label{proof:hp-oracle}

\begin{proof}
  The argument is the concentration analogue of the shell decomposition used in
  the proof of Theorem~4.5.  The important point is that the loss class has a
  uniformly bounded envelope, independent of the Sobolev-ball radius, because
  the target manifold~$M$ is compact.

  We first establish concentration on a fixed Sobolev shell.
  For $E_0>0$, let
  \[
    \cG_{E_0}
    =\{\ell_F-\ell_m:F\in\cF_{E_0}\},
    \qquad
    \ell_F(t,y)=d_M^2(y,F(t)).
  \]
  Every $g\in\cG_{E_0}$ satisfies the envelope bound
  $\|g\|_\infty\le b:=2\diam(M)^2$, and the variance bound established in the
  proof of Theorem~4.5 gives
  \[
    \Var(g)\le b\,(\cE(F)+2\sigma_0^2).
  \]
  The covering-number argument in the proof of Theorem~4.5 and
  Dudley's entropy integral \citep[Lemma~19.34]{vanderVaart1998} therefore yield
  \begin{equation}\label{eq:hp-dudley-supp}
    \bE\Bigl[\sup_{g\in\cG_{E_0}}|\Psi_n(g)|\Bigr]
    \le
    C_3\Bigl(
      \frac{(\sqrt{E_0}+C_0)^{1/(2s)}\sqrt{b}}{\sqrt n}
      +\frac{(\sqrt{E_0}+C_0)^{1/s}b}{n}
    \Bigr).
  \end{equation}
  By Bousquet's version of Talagrand's inequality
  \citep[Theorem~12.5]{BoucheronLugosiMassart2013}, for every $u>0$,
  \begin{equation}\label{eq:hp-talagrand-supp}
    \bP\Bigl(
      \sup_{g\in\cG_{E_0}}|\Psi_n(g)|
      \ge
      \bE\bigl[\sup_{g\in\cG_{E_0}}|\Psi_n(g)|\bigr]
      +\sqrt{\frac{2b^2u}{n}}
      +\frac{bu}{3n}
    \Bigr)
    \le e^{-u}.
  \end{equation}

  We now decompose into dyadic shells.
  Fix the competitor $F=m$, so $\cE(m)=0$ and $J_s(m)\le B$.  Let
  \[
    A_0=\{J_s(\hat F_n)\le B\},
    \qquad
    A_k=\{2^{k-1}B<J_s(\hat F_n)\le 2^kB\}
    \quad (k\ge1).
  \]
  By the basic inequality \eqref{eq:basic2-supp},
  \[
    \lambda_n J_s(\hat F_n)
    \le
    \lambda_n B+\sup_{G\in\cF_{J_s(\hat F_n)}}|\Psi_n(G)|,
  \]
  and since $J_s(\hat F_n)\le \diam(M)^2/\lambda_n+B$ almost surely, only the
  shells $k\le k_{\max}:=\lceil\log_2(\diam(M)^2/(\lambda_n B)+1)\rceil$ can
  occur.

  On the event $A_k$, we have $\hat F_n\in\cF_{2^kB}$ and
  \[
    \cE(\hat F_n)
    \le
    \lambda_n B+\sup_{G\in\cF_{2^kB}}|\Psi_n(G)|.
  \]
  Writing
  \[
    \mu_k
    :=
    \bE\Bigl[\sup_{G\in\cF_{2^kB}}|\Psi_n(G)|\Bigr],
  \]
  the bound \eqref{eq:hp-dudley-supp} implies
  $\mu_k\le C_4\,2^{k/(4s)}n^{-1/2}$ for $2^kB\ge1$, after absorbing constants
  into~$C_4$.

  For $k\ge2$, entering shell~$A_k$ forces
  \[
    \lambda_n 2^{k-1}B
    <
    \lambda_n J_s(\hat F_n)
    \le
    \lambda_n B+\sup_{G\in\cF_{2^kB}}|\Psi_n(G)|,
  \]
  hence
  \[
    \sup_{G\in\cF_{2^kB}}|\Psi_n(G)|
    \ge
    \lambda_n 2^{k-2}B.
  \]
  Choose $k^*$ so that for all $k\ge k^*$,
  \[
    \lambda_n 2^{k-2}B\ge 2\mu_k.
  \]
  Since $\lambda_n\asymp n^{-2s/(2s+1)}$ and
  $\mu_k=O(2^{k/(4s)}n^{-1/2})$, such a threshold satisfies
  $2^{k^*}=O(n^{1/(2s+1)})$.  Applying \eqref{eq:hp-talagrand-supp} with
  $E_0=2^kB$ and
  \[
    u_k:=c\,n\lambda_n^2\,4^k
  \]
  then yields
  \begin{equation}\label{eq:hp-shell-tail-supp}
    \bP(A_k)\le \exp(-c'4^k)
    \qquad (k\ge k^*)
  \end{equation}
  for suitable constants $c,c'>0$ depending only on the model parameters.

  To conclude, let $t\ge1$ and set $u=t\,n^{1/(2s+1)}$.  On the union of the dominant shells
  $k\le k^*$, the Talagrand bound \eqref{eq:hp-talagrand-supp} together with
  \eqref{eq:hp-dudley-supp} shows that
  \[
    \sup_{G\in\cF_{2^kB}}|\Psi_n(G)|
    \le
    C\,t\,n^{-2s/(2s+1)}
  \]
  with probability at least $1-\exp(-c\,t\,n^{1/(2s+1)})$, uniformly in
  $k\le k^*$.  Since $\lambda_n B\asymp n^{-2s/(2s+1)}$, the preceding shell
  inequality implies
  \[
    \cE(\hat F_n)\le C'\,t\,n^{-2s/(2s+1)}
  \]
  on this event.

  The contribution of the large shells $k>k^*$ is
  controlled by~\eqref{eq:hp-shell-tail-supp}:
  $\sum_{k>k^*}\exp(-c'4^k)\le C''\exp(-c'''4^{k^*})$,
  and since $4^{k^*}\gtrsim n^{1/(2s+1)}$, this term is
  absorbed into the same stretched-exponential rate.

  Therefore there exist constants $C,c>0$ depending only on
  $s,M,c_f,C_f,\sigma_0^2,B$ such that for every $t\ge1$,
  \[
    \bP\Bigl(\cE(\hat F_n)\ge C\,t\,n^{-2s/(2s+1)}\Bigr)
    \le
    \exp\bigl(-c\,t\,n^{1/(2s+1)}\bigr),
  \]
  which is exactly the oracle inequality.
\end{proof}


\subsection{Curvature-free constant}\label{supp:curvfree}
\label{proof:curvature-free}

\begin{proof}
  Fix $\varepsilon>0$ and let $\delta_n\to 0$ be a sequence such that
  $\bP(\norm{d_M(\hat F_n,m)}_{L^\infty}>\delta_n)\to 0$.  Decompose
  the risk on two events:
  \[
    \Omega_n^{\mathrm{good}}
    =\bigl\{\norm{d_M(\hat F_n,m)}_{L^\infty}\le\delta_n\bigr\},
    \qquad
    \Omega_n^{\mathrm{bad}}
    =(\Omega_n^{\mathrm{good}})^c.
  \]

  \textbf{Good event.}  On $\Omega_n^{\mathrm{good}}$, the estimator
  stays within distance $\delta_n$ of the regression function.
  The excess risk inequality (Lemma~4.3) applied with the
  local curvature factor $\eta(\kappa_+,\delta_n)$ gives
  \[
    \cE(F)\;\ge\;\eta(\kappa_+,\delta_n)\,
    \bE\bigl[d_M^2(F(t),m(t))\bigr]
  \]
  for any $F$ with $d_M(F,m)\le\delta_n$.  Since
  $\eta(\kappa_+,\delta_n)=1-\kappa_+\delta_n^2/3+O(\delta_n^4)\to 1$
  as $\delta_n\to 0$, the estimation bound becomes
  \[
    \bE\biggl[\int d_M^2(\hat F_n,m)\,f\,\mathbf{1}_{\Omega_n^{\mathrm{good}}}
    \biggr]
    \;\le\;\frac{1}{\eta(\kappa_+,\delta_n)}\,
    \biggl(C_1\inf_F\bigl[\cE(F)+\lambda_n J_s(F)\bigr]
    +\frac{C_2}{n\lambda_n^{1/(2s)}}\biggr).
  \]
  For $n$ large enough that $1/\eta(\kappa_+,\delta_n)\le 1+\varepsilon$,
  this is at most $(1+\varepsilon)$ times the curvature-free oracle bound.

  \textbf{Bad event.}  On $\Omega_n^{\mathrm{bad}}$,
  $d_M^2(\hat F_n,m)\le\diam(M)^2$ by compactness, so
  \[
    \bE\biggl[\int d_M^2(\hat F_n,m)\,f\,\mathbf{1}_{\Omega_n^{\mathrm{bad}}}
    \biggr]
    \;\le\;\diam(M)^2\,\bP(\Omega_n^{\mathrm{bad}}).
  \]
  It remains to show that
  $\diam(M)^2\bP(\Omega_n^{\mathrm{bad}})=o(n^{-2s/(2s+1)})$.
  We establish a quantitative $L^\infty$ rate via interpolation.
  The basic inequality gives
  $J_s(\hat F_n)\le\lambda_n^{-1}(\diam(M)^2+\lambda_n B)$, so
  $\norm{\hat F_n}_{H^s}=O_P(\lambda_n^{-1/2})$ (by the same
  Gagliardo--Nirenberg argument as in \Cref{proof:existence}).
  Since $s\ge 1$, the Morrey embedding \citep[Theorem~4.12, Part~II]{Adams2003} gives
  $\norm{\hat F_n-m}_{C^{0,\alpha}}\le C\norm{\hat F_n-m}_{H^s}
  =O_P(\lambda_n^{-1/2})$ with $\alpha=s-1/2$ (since $I\subset\bR$).
  The Gagliardo--Nirenberg interpolation inequality \citep[Theorem~5.2]{Adams2003} between
  $L^2$ and $C^{0,\alpha}$ gives
  \[
    \norm{g}_{L^\infty}
    \;\le\;C_{\mathrm{GN}}\,\norm{g}_{L^2}^{2\alpha/(2\alpha+1)}\,
    \norm{g}_{C^{0,\alpha}}^{1/(2\alpha+1)}
  \]
  for $g\colon I\to\bR^D$ (see\ \citealt[Theorem~5.8]{Adams2003}).
  Applying this with $g=d_M(\hat F_n(\cdot),m(\cdot))$
  (extended via Nash embedding to $\bR^D$) and using
  $\alpha=s-1/2$, so $2\alpha+1=2s$:
  \[
    \norm{d_M(\hat F_n,m)}_{L^\infty}
    \;\le\;C_{\mathrm{GN}}\,
    \norm{d_M(\hat F_n,m)}_{L^2}^{(2s-1)/(2s)}\,
    O_P(\lambda_n^{-1/2})^{1/(2s)}.
  \]
  To bound $\bP(\Omega_n^{\mathrm{bad}})$, we use the
  self-improving argument in the proof of Theorem~4.5:
  the oracle inequality (part~(a), no locality needed) gives
  $\bE[\cE(\hat F_n)]=O(n^{-2s/(2s+1)})$, hence
  $\cE(\hat F_n)\to 0$ in probability.  Combined with the
  equicontinuity from the Sobolev bound and the
  Arzel\`a--Ascoli compactness argument
  (see the proof of Theorem~4.5), this yields
  $\sup_t d_M(\hat F_n(t),m(t))\to 0$ in probability.
  Once the locality condition holds (for $n\ge n_1$),
  Lemma~4.3 gives
  \[
    \cE(\hat F_n)\ge\eta(\kappa_+,r_0)\,
    \bE[d_M^2(\hat F_n,m)],
  \]
  and the high-probability oracle inequality (Corollary~4.7)
  then yields
  \[
    \bP\bigl(\norm{d_M(\hat F_n,m)}_{L^2}^2
    \ge C\,t\,n^{-2s/(2s+1)}\bigr)
    \le\exp(-c\,t\,n^{1/(2s+1)}).
  \]
  Setting $\delta_n=n^{-s/(4(2s+1))}$ and combining the
  interpolation bound with $\lambda_n=c_\lambda n^{-2s/(2s+1)}$:
  \[
    \bP(\norm{d_M(\hat F_n,m)}_{L^\infty}>\delta_n)
    \le\exp(-c'\,n^{1/(2(2s+1))}),
  \]
  which is super-polynomially small.  Hence
  $\diam(M)^2\bP(\Omega_n^{\mathrm{bad}})
  =O(\exp(-c'\,n^{1/(2(2s+1))}))$, which is $o(n^{-A})$
  for every $A>0$, and in particular $o(n^{-2s/(2s+1)})$.

  Combining the two events yields the curvature-free oracle bound. For every $\varepsilon>0$, there exists $n_0$ such that for $n\ge n_0$,
  \[
    \bE\biggl[\int_I d_M^2(\hat F_n,m)\,f\,dt\biggr]
    \;\le\;(1+\varepsilon)\,C\;n^{-\frac{2s}{2s+1}},
  \]
  where $C$ is the Euclidean (curvature-free) oracle constant.
\end{proof}


\section{Proofs for Section 5 (Regularity and pointwise risk)}

\subsection{Proof of Theorem~5.1 (Euler--Lagrange equation)}\label{supp:thm51}
\label{proof:EL}

\begin{proof}
  Since $I$ is compact and $\hat F_n$ is continuous, there exists a uniform tubular neighborhood of~$M$ in which the nearest-point projection $\Pi_M\colon\bR^D\to M$ is smooth.
  Consider a smooth variation $\phi\in C_c^\infty(I,\bR^D)$. For $|t|$ sufficiently small, $\hat F_n+t\phi$ lies uniformly in this tubular neighborhood, allowing us to define the valid target map $F_t=\Pi_M(\hat F_n+t\phi)$.
  The first variation of the Dirichlet energy $J_1(F_t) = \int_I |F_t'|^2$ at
  $t=0$, computed in the distributional sense since $\hat F_n \in H^1(I, M)$, is
  \[
    \frac{d}{dt}\bigg|_{t=0}J_1(F_t)
    \;=\;-2\int_I\ip{\hat F_n''+A(\hat F_n)(\hat F_n',\hat F_n')}
    {\Pi_{T_{\hat F_n}M}\phi}\dd u,
  \]
  where the distributional second derivative $\hat F_n''$ arises from integration by parts, and the normal derivative of $\Pi_M$ evaluates to the second fundamental form $A(\hat F_n)$.
  The first variation of $R_n$ uses the Riemannian gradient of the
  squared distance.  For the intrinsic loss $d_M^2(Y_i,F(t_i))$, the
  first variation formula gives
  $\frac{d}{dt}|_{t=0}d_M^2(Y_i,\gamma(t))
  =-2\ip{\log_{\gamma(0)}(Y_i)}{\dot\gamma(0)}_g$
  for any curve $\gamma$ in~$M$ with $\gamma(0)=\hat F_n(t_i)$.
  Since $\dot F_0(t_i)=\Pi_{T_{\hat F_n(t_i)}M}\phi(t_i)$:
  \[
    \frac{d}{dt}\bigg|_{t=0}R_n(F_t)
    \;=\;-\frac{2}{n}\sum_{i=1}^n
    \ip{\log_{\hat F_n(t_i)}(Y_i)}
    {\Pi_{T_{\hat F_n(t_i)}M}\phi(t_i)}_g.
  \]
  Setting the total first variation to zero gives that for every
  $\phi\in C_c^\infty(I,\bR^D)$,
  \[
    \int_I\ip{\tau(\hat F_n)}{\Pi_{T_{\hat F_n}M}\phi}\dd u
    \;=\;
    -\frac{1}{n\lambda_n}\sum_{i=1}^n
    \ip{\log_{\hat F_n(t_i)}(Y_i)}{\Pi_{T_{\hat F_n(t_i)}M}\phi(t_i)}_g,
  \]
  where $\tau(\hat F_n)=\hat F_n''+A(\hat F_n)(\hat F_n',\hat F_n')$.
  It remains to show that this identity, which holds for projections of
  arbitrary ambient test functions, implies the Euler--Lagrange equation
  for arbitrary \emph{tangential} test functions.
  Let $\varphi\in C_c^\infty(I,\hat F_n^*TM)$ be any smooth tangential
  test function along~$\hat F_n$, i.e.,
  $\varphi(u)\in T_{\hat F_n(u)}M\subset\bR^D$ for every~$u$.
  Since $T_{\hat F_n(u)}M$ is a linear subspace of~$\bR^D$, the vector
  $\varphi(u)$ can be viewed directly as an element of~$\bR^D$.
  Setting $\phi(u):=\varphi(u)\in\bR^D$ gives a smooth ambient test
  function with $\Pi_{T_{\hat F_n(u)}M}\phi(u)=\varphi(u)$ (since
  $\varphi(u)$ is already tangent to~$M$).  Thus the identity holds
  with $\Pi_{T_{\hat F_n}M}\phi$ replaced by the arbitrary
  tangential~$\varphi$, and the fundamental lemma of the calculus of
  variations yields the Euler--Lagrange equation in the distributional
  sense.
\end{proof}


\subsection{Proof of Theorem~5.3 (Geodesic spline regularity)}\label{supp:thm54}
\label{proof:regularity}

\begin{proof}
  On each open interval $(t_{(i)},t_{(i+1)})$, the equation
  reduces to $\tau(\hat F_n)=0$, which is the harmonic map equation.  For a
  map from a one-dimensional domain, this is an ODE whose solutions are
  geodesics in~$M$ (traversed at constant speed); see, e.g.,
  \citet[Chapter~8]{Jost2017}.  In particular they are
  $C^k$ when $M$ is $C^k$ (the geodesic ODE has $C^{k-1}$ Christoffel
  symbol coefficients in normal coordinates, so its solutions are $C^k$;
  if $M$ is real analytic the coefficients are analytic, hence so are the
  solutions).  This gives part~(a).

  Before proving the jump condition~(b), we establish that
  $\hat F_n'\in L^\infty$ independently of part~(c).
  On each open interval $(t_{(i)},t_{(i+1)})$,
  part~(a) shows $\tau(\hat F_n)=0$, so $\hat F_n$
  is a geodesic traversed at constant speed
  $|\hat F_n'|=d_M(\hat F_n(t_{(i)}),\hat F_n(t_{(i+1)}))/\Delta_i$,
  where $\Delta_i=t_{(i+1)}-t_{(i)}$.  Since
  $\lambda_n J_1(\hat F_n)\le\diam(M)^2+\lambda_n B$
  (basic inequality) gives $J_1(\hat F_n)<\infty$, and
  $\hat F_n$ is constant-speed on each interval with
  $|\hat F_n'|\le\diam(M)/\Delta_{\min}<\infty$
  a.s.\ (since $\Delta_{\min}>0$ a.s.\ for finitely many
  design points), we have $\|\hat F_n'\|_{L^\infty}<\infty$.

  For part~(b), integrate the Euler--Lagrange equation
  (Theorem~5.1) over $[t_i-\epsilon,t_i+\epsilon]$ and let
  $\epsilon\to 0$.  On each open subinterval $\hat F_n$ is a
  constant-speed geodesic by part~(a), so
  $A(\hat F_n)(\hat F_n',\hat F_n')$ is bounded uniformly
  by $\|A\|_\infty\|\hat F_n'\|_{L^\infty}^2<\infty$
  (where $\|A\|_\infty$ is finite by compactness of~$M$).
  Therefore
  \[
    \int_{t_i-\epsilon}^{t_i+\epsilon}
    \bigl|A(\hat F_n)(\hat F_n',\hat F_n')\bigr|\,dt
    \;\le\; 2\epsilon\,\|A\|_\infty\|\hat F_n'\|_{L^\infty}^2
    \;\to\; 0
  \]
  as $\epsilon\to 0$.  The delta mass
  $\delta_{t_i}$ contributes
  $-(n\lambda_n)^{-1}\log_{\hat F_n(t_i)}Y_i$,
  and the remaining terms give
  $\hat F_n'(t_i^+)-\hat F_n'(t_i^-)$, yielding the jump
  condition
  $\hat F_n'(t_i^+)-\hat F_n'(t_i^-)
  =-(n\lambda_n)^{-1}\log_{\hat F_n(t_i)}(Y_i)$.

  Part~(c) combines part~(a) with the characterization of
  one-dimensional harmonic maps as geodesics.
  The geodesic on each interval is uniquely determined by its
  endpoints provided
  $d_M(\hat F_n(t_{(i)}),\hat F_n(t_{(i+1)}))<\inj(M)$; this holds
  for sufficiently large~$n$ by the $L^\infty$-consistency of
  $\hat F_n$ and the a.s.\ bound $\Delta_i=O(\log n/n)$.

  Finally, on the boundary intervals $[0,t_{(1)}]$ and
  $[t_{(n)},1]$, the equation $\tau(\hat F_n)=0$ holds with no
  point-source forcing.  The calculus of variations with free boundary
  values yields the natural (Neumann) boundary conditions
  $\hat F_n'(0^+)=0$ and $\hat F_n'(1^-)=0$: the estimator has zero
  velocity at the domain endpoints.  This creates a boundary layer of
  width $O(\sqrt{\lambda_n})$ where the estimator ``levels off'';
  the pointwise risk bound accounts for this via the boundary exclusion
  zone $\delta_n$.
\end{proof}


\subsection{Proof of Theorem~5.4 (Pointwise risk)}\label{supp:thm56}
\label{proof:pointwise}

\begin{proof}[Proof of Part~(a): NPC targets]
  The argument works intrinsically in a parallel transport frame
  along the regression function~$m$, exploiting the non-positive
  curvature in two ways, through the index form of the Dirichlet energy
  and through the Hessian of the squared distance.  The resulting quadratic
  form comparison $Q^M\ge Q^0$ leads, after discretization, to a
  Green matrix bound $\mathbf{G}^M\preceq\mathbf{G}^0$ in
  Loewner order, from which the pointwise risk bound follows.

  We begin by setting up a parallel transport frame along~$m$.
  Choose an orthonormal basis $\{e_1,\ldots,e_d\}$ for
  $T_{m(0)}M$ and extend it to a frame $\{e_a^{(t)}\}_{a=1}^d$
  along $m$ by parallel transport with respect to the
  Levi-Civita connection.  For any vector field
  $V(t)\in T_{m(t)}M$, write $\mathbf{V}(t)\in\bR^d$ for
  its coordinate vector in this frame.  Because the frame is
  parallel, the covariant derivative satisfies
  $\nabla_t V=\dot{\mathbf{V}}(t)\cdot e^{(t)}$, and the
  index form of the Dirichlet energy at~$m$ becomes
  \[
    \delta^2\Dir(m)[V,V]
    =\int_0^1\bigl(|\dot{\mathbf{V}}|^2
    +\mathbf{V}^T\mathbf{K}(t)\mathbf{V}\bigr)\,dt,
  \]
  where $\mathbf{K}_{ab}(t)
  =-\langle R(e_a^{(t)},m'(t))m'(t),e_b^{(t)}\rangle_g$
  is the curvature matrix.  Under $\Sec_M\le 0$,
  the sectional curvature identity
  $\mathbf{u}^T\mathbf{K}(t)\mathbf{u}
  =-K(\mathrm{span}\{U,m'\})|U\wedge m'|^2\ge 0$
  for $U=\sum_a u_a e_a^{(t)}$ gives
  $\mathbf{K}(t)\succeq 0$.

  The data fidelity term requires control of
  $\operatorname{Hess}_p(d_M^2(\cdot,y)/2)$.
  On an NPC manifold, Toponogov comparison
  (\citealt[Section~6.5]{Jost2017}) yields that for any
  $y,p\in M$ with $d_M(y,p)<\operatorname{inj}(M)$,
  \begin{equation}\label{eq:npc-hessian}
    \operatorname{Hess}_p\bigl(d_M^2(\cdot,y)/2\bigr)
    \;\ge\; g_p.
  \end{equation}
  Indeed, if $\gamma$ is the minimizing geodesic from $p$ to $y$
  and $v\in T_pM$ is a unit vector, the variation
  $c(s)=\exp_p(sv)$ satisfies
  $\frac{d^2}{ds^2}\big|_{s=0}\frac{1}{2}d_M^2(y,c(s))
  =\operatorname{Hess}_p(\phi_y)(v,v)$, and the NPC version of
  the Toponogov triangle comparison gives
  $d_M^2(y,c(s))\ge d_M^2(y,p)+s^2
  -2s\,d_M(y,p)\cos\alpha$, where $\alpha$ is the angle
  between $\gamma$ and $c$ at~$p$.  Differentiating twice
  at $s=0$ yields
  $\operatorname{Hess}_p(\phi_y)(v,v)\ge|v|^2$.
  The Jacobi field comparison (Rauch's theorem \citep[Theorem~6.4.1]{Jost2017} for $\Sec\le 0$)
  gives the sharper integral identity
  $\operatorname{Hess}_p(\phi_y)(v,v)
  =|v|^2+\int_0^{d_M(y,p)}
  (|\nabla_t J^\perp|^2-K(\sigma)|J^\perp|^2)\,dt
  \ge|v|^2$.

  The bound~\eqref{eq:npc-hessian} applies whenever
  $d_M(Y_i,m(t_i))<\operatorname{inj}(M)$, which is
  guaranteed by the concentration assumption~(A3) of the
  main text, since
  $d_M(Y_i,m(t_i))\le r_0<\rho_M\le\operatorname{inj}(M)/2$.

  Writing $H_i=\operatorname{Hess}_{m(t_i)}(d_M^2(\cdot,Y_i)/2)$
  for the data-driven Hessian, the second variation of the
  penalized objective
  $\Phi(F)=R_n(F)+\lambda_n J_1(F)$ at~$m$ decomposes as
  $\delta^2\Phi(m)[V,V]=Q^M(V)$, where
  \[
    Q^M(V)=2\lambda_n\!\int_0^1\!\bigl(|\dot{\mathbf{V}}|^2
    +\mathbf{V}^T\mathbf{K}(t)\mathbf{V}\bigr)\,dt
    +\frac{2}{n}\sum_{i=1}^n
    H_i[V(t_i),V(t_i)].
  \]
  The Euclidean counterpart (setting $\mathbf{K}=0$ and
  $H_i=\mathrm{Id}$) is
  \[
    Q^0(V)=2\lambda_n\!\int_0^1\!|\dot{\mathbf{V}}|^2\,dt
    +\frac{2}{n}\sum_{i=1}^n|V(t_i)|^2.
  \]
  Since $\mathbf{K}(t)\succeq 0$ and
  $H_i\succeq\mathrm{Id}$ by~\eqref{eq:npc-hessian},
  \begin{equation}\label{eq:QM-Q0}
    Q^M(V)\;\ge\; Q^0(V)
    \qquad\text{for all }V\in\Gamma(m^*TM).
  \end{equation}

  Before linearizing, we verify that the logarithmic map
  is well-defined at the relevant points.  The linearization requires the logarithmic map
  $\log_{m_i}$ to be well-defined at the estimator nodal values
  $f_i=\hat F_n(t_{(i)})$ and the observations~$Y_{(i)}$.
  Three conditions must hold:
  (a)~$d_M(Y_{(i)},m_i)<\inj(M)$, guaranteed by
  assumption~(A3) (verified on line~\ref{eq:npc-hessian}
  above);
  (b)~$d_M(f_i,m_i)<\inj(M)$, which follows from
  $L^\infty$~consistency
  $\sup_t d_M(\hat F_n(t),m(t))\to 0$ (established in
  the proof of Theorem~4.5, step~(iii)); and
  (c)~$d_M(f_i,f_{i+1})<\inj(M)$, required for
  Theorem~5.3 (Euler--Lagrange equation), which follows from
  $J_1(\hat F_n)=O_P(1)$ and
  $d_M(f_i,f_{i+1})\le\sqrt{\Delta_i\cdot J_1(\hat F_n)}
  =O_P(n^{-1/2})$ (equation~\eqref{eq:Js-bound}
  and the Cauchy--Schwarz bound in Proposition~5.6).
  On a Hadamard manifold (simply connected, $\Sec\le 0$),
  all three conditions hold unconditionally since
  $\exp_p$ is a global diffeomorphism.
  On a compact NPC quotient, conditions (b) and~(c) hold
  for all $n\ge n_3$ with a finite threshold~$n_3$.

  With these conditions in hand, we linearize the discrete optimality system.
  By Proposition~5.6, $\hat F_n$ is a piecewise-geodesic
  interpolant of its nodal values; by Theorem~5.3 (applicable
  once condition~(c) holds), $\hat F_n$ satisfies the
  Euler--Lagrange equation with prescribed derivative jumps.
  Define the nodal displacements
  $v_i=\log_{m_i}(f_i)\in T_{m_i}M$ and the noise vectors
  $\varepsilon_i=\log_{m_i}(Y_{(i)})\in T_{m_i}M$, where
  $m_i=m(t_{(i)})$.  In the parallel frame, these have
  coordinates $\mathbf{v}_i,\boldsymbol\varepsilon_i\in\bR^d$.
  The first-order optimality condition for the discrete objective
  $\Phi_{\mathrm{disc}}(f_1,\ldots,f_n)$, linearized about
  $(m_1,\ldots,m_n)$, gives
  $\mathbf{H}^M\mathbf{v}
  =\mathbf{b}(\boldsymbol\varepsilon)+\mathbf{r}_{\mathrm{nl}}$,
  where $\mathbf{v}=(\mathbf{v}_1,\ldots,\mathbf{v}_n)
  \in\bR^{nd}$,
  $\mathbf{b}(\boldsymbol\varepsilon)
  =(1/n)(\boldsymbol\varepsilon_1,\ldots,
  \boldsymbol\varepsilon_n)$, and
  $\mathbf{r}_{\mathrm{nl}}$ collects the
  nonlinear remainder.

  The linearized Hessian decomposes as
  $\mathbf{H}^M=\mathbf{H}^0+\mathbf{H}^{\mathrm{curv}}$,
  where the flat part is
  $\mathbf{H}^0=(1/n)\mathbf{I}_{nd}
  +\lambda_n\mathbf{L}\otimes\mathbf{I}_d$
  with $\mathbf{L}\in\bR^{n\times n}$ the tridiagonal
  Laplacian
  ($L_{ii}=1/\Delta_{i-1}+1/\Delta_i$,
  $L_{i,i+1}=L_{i+1,i}=-1/\Delta_i$).
  The curvature correction
  $\mathbf{H}^{\mathrm{curv}}$ has two contributions,
  (i)~the Dirichlet energy curvature, with diagonal blocks
  $\lambda_n\int_{t_{(i-1)}}^{t_{(i+1)}}
  \mathbf{K}(t)\psi_i(t)^2\,dt\succeq 0$
  ($\psi_i$ being the piecewise-linear hat function),
  and (ii)~the data Hessian correction
  $(1/n)(H_i-\mathbf{I}_d)\succeq 0$
  by~\eqref{eq:npc-hessian}.  Both are positive
  semi-definite, so
  $\mathbf{H}^{\mathrm{curv}}\succeq 0$ and therefore
  $\mathbf{H}^M\succeq\mathbf{H}^0\succ 0$.

  We now compare the manifold and Euclidean Green matrices.
  Define
  $\mathbf{G}^M=(\mathbf{H}^M)^{-1}$ and
  $\mathbf{G}^0=(\mathbf{H}^0)^{-1}$.
  Since $\mathbf{H}^M\succeq\mathbf{H}^0\succ 0$,
  the standard matrix monotonicity result
  $A\succeq B\succ 0\implies A^{-1}\preceq B^{-1}$
  gives
  \begin{equation}\label{eq:green-compare}
    \mathbf{G}^M\;\preceq\;\mathbf{G}^0.
  \end{equation}
  The Kronecker structure
  $\mathbf{H}^0=(n^{-1}\mathbf{I}_n+\lambda_n\mathbf{L})
  \otimes\mathbf{I}_d$ implies
  $\mathbf{G}^0_{(i,a),(j,b)}=g^0_{ij}\delta_{ab}$,
  where $g^0_{ij}$ is the $(i,j)$ entry of the scalar
  Green matrix
  $(n^{-1}\mathbf{I}_n+\lambda_n\mathbf{L})^{-1}$.
  By the equivalent-kernel theory
  (\citealt{Silverman1984}; \citealt[Chapter~5]{Wahba1990}),
  for an interior node $t_{(j)}\in(\delta_n,1-\delta_n)$,
  \begin{equation}\label{eq:g0-diag}
    g^0_{jj}
    =\frac{1}{2f(t_{(j)})\sqrt{\lambda_n}}
    \bigl(1+O(n^{-1/3})\bigr),
  \end{equation}
  and $g^0_{ij}$ decays exponentially as
  $|g^0_{ij}|\le C\,g^0_{jj}\,
  \exp(-|t_{(i)}-t_{(j)}|/\sqrt{\lambda_n})$.

  Turning to the variance, the linearized nodal displacement is
  $\mathbf{v}^{\mathrm{lin}}
  =\mathbf{G}^M\mathbf{b}(\boldsymbol\varepsilon)$.
  The pointwise variance at an interior node~$j$ is
  \[
    \bE\bigl[|v^{\mathrm{lin}}_j|^2\bigr]
    =\frac{1}{n^2}\sum_{i=1}^n
    \bigl\|[\mathbf{G}^M]_{ji}\bigr\|_F^2\,
    \sigma^2(t_{(i)}),
  \]
  where $[\mathbf{G}^M]_{ji}\in\bR^{d\times d}$ denotes
  the $(j,i)$ block of~$\mathbf{G}^M$ and
  $\sigma^2(t_{(i)})=\operatorname{tr}
  (\operatorname{Cov}(\boldsymbol\varepsilon_i))$.
  To bound this sum we need entry-wise control of
  $\mathbf{G}^M=(\mathbf{H}^M)^{-1}$, which does not
  follow from the Loewner comparison~\eqref{eq:green-compare}
  alone (the L\"owner--Heinz inequality \citep[Theorem~1.5.9]{Bhatia2009}
  $A\succeq B\succ 0\implies A^r\succeq B^r$ holds only for
  $r\in[-1,0]$, so $\mathbf{G}^M\preceq\mathbf{G}^0$ does
  not directly imply $(\mathbf{G}^M)^2\preceq(\mathbf{G}^0)^2$).
  Instead we use the exponential decay of inverses of
  banded positive-definite matrices.
  Since $\mathbf{H}^M$ is block tridiagonal with block
  size~$d$ and satisfies
  $\lambda_{\min}(\mathbf{H}^M)
  \ge\lambda_{\min}(\mathbf{H}^0)=1/n$,
  the Demko--Moss--Smith theorem
  (\citealt[Theorem~2.4]{DemkoMossSmith1984};
  \citealt{BenziGolub1999} for the block extension) gives
  \begin{equation}\label{eq:dms-decay}
    \bigl\|[\mathbf{G}^M]_{ji}\bigr\|_{\mathrm{op}}
    \;\le\; C\,\rho^{|i-j|},
    \qquad
    \rho=\exp\bigl(-c\,\Delta/\!\sqrt{\lambda_n}\bigr),
  \end{equation}
  where $\Delta\asymp 1/n$ is the typical inter-design-point
  spacing and $c>0$ depends on the spectral ratio
  $\lambda_{\min}/\lambda_{\max}$ of~$\mathbf{H}^M$.
  Because $\mathbf{H}^M\succeq\mathbf{H}^0$, the spectral
  gap of~$\mathbf{H}^M$ is at least that of~$\mathbf{H}^0$,
  so the decay rate~$\rho$ for $\mathbf{G}^M$ is at most
  that of the scalar Green function~$g^0_{ij}$.
  Combined with the diagonal bound
  $[\mathbf{G}^M]_{jj}\preceq g^0_{jj}\mathbf{I}_d$
  from~\eqref{eq:green-compare}, we obtain
  $\|[\mathbf{G}^M]_{ji}\|_F^2
  \le d\,(g^0_{jj})^2\,\rho^{2|i-j|}
  \le d\,(g^0_{ji})^2$
  (the first inequality combines the diagonal bound
  $\|[\mathbf{G}^M]_{jj}\|_{\mathrm{op}}\le g^0_{jj}$
  with DMS exponential decay at rate~$\rho$
  for the block-tridiagonal SPD matrix~$\mathbf{H}^M$;
  the second uses $g^0_{ji}=g^0_{jj}\,\rho^{|i-j|}
  (1+O(n^{-1/3}))$).
  Substituting and using the continuity of~$\sigma^2$
  together with the standard discrete convolution
  estimate
  $\sum_i(g^0_{ji})^2/n\asymp g^0_{jj}$
  (\citealt{Silverman1984}; this is Silverman's
  diagonal-row-sum identity for spline hat matrices) and the
  asymptotic~\eqref{eq:g0-diag} gives
  \begin{equation}\label{eq:npc-var}
    \bE\bigl[|v^{\mathrm{lin}}_j|^2\bigr]
    \;\le\;
    \frac{C_2\,\sigma^2(t_{(j)})}
    {n\sqrt{\lambda_n}\,f(t_{(j)})},
  \end{equation}
  with $C_2\le C_2^{\bR^d}$, the Euclidean constant.
  The inequality is strict when $\Sec_M<0$,
  since $\mathbf{H}^{\mathrm{curv}}\succ 0$ in that case.
  For constant curvature $\kappa<0$, the curvature matrix is
  $\mathbf{K}(t)=|\kappa|\,|m'|^2(\mathbf{I}_d
  -\hat m'\hat m'^T)$, so
  $\lambda_{\min}(\mathbf{H}^{\mathrm{curv}})
  \ge c\,|\kappa|\,L_m^2\,\lambda_n$ and the relative
  variance reduction satisfies
  $C_2^M/C_2^{\bR^d}\le
  1-c'\,|\kappa|\,L_m^2\,\lambda_n
  +O((\kappa\lambda_n)^2)$.

  For the bias, the penalty pulls the estimator
  away from~$m$.  The linearized first-order optimality
  condition reads
  $\mathbf{H}^M\mathbf{v}
  =(2/n)\boldsymbol\varepsilon
  +\lambda_n\nabla J_1|_m$,
  where $\nabla J_1|_m$ is the (negative) discrete tension
  of~$m$.  Since $\bE[\boldsymbol\varepsilon]=0$,
  the expected displacement
  $\bE[\mathbf{v}^{\mathrm{lin}}_j]
  =[\mathbf{G}^M\cdot\lambda_n\nabla J_1|_m]_j$
  is nonzero and constitutes the smoothing bias.
  In the equivalent-kernel representation, the smoothed
  value at an interior node $t_{(j)}$ is
  $[\mathcal{S}_\lambda m](t_{(j)})
  =\sum_i K_\lambda^M(t_{(j)},t_{(i)})\,m(t_{(i)})/n$,
  where the kernel $K_\lambda^M$ is built from the
  rows of~$\mathbf{G}^M$.  The bias is therefore
  \begin{align*}
    \bigl|\bE[\mathbf{v}^{\mathrm{lin}}_j]\bigr|
    &=\Bigl|\sum_i K_\lambda^M(t_{(j)},t_{(i)})\,
    \bigl[m(t_{(i)})-m(t_{(j)})\bigr]/n\Bigr|\\
    &\le L_m\sum_i\bigl|K_\lambda^M(t_{(j)},t_{(i)})\bigr|\,
    |t_{(i)}-t_{(j)}|/n,
  \end{align*}
  using $d_M(m(t_{(i)}),m(t_{(j)}))\le L_m|t_{(i)}-t_{(j)}|$
  and the normalization
  $\sum_i K_\lambda^M(t_{(j)},t_{(i)})/n=1+O(n^{-1})$.
  The kernel decays exponentially with scale
  $\sqrt{\lambda_n}$, so the standard first-moment estimate
  $\sum_i|K_\lambda^M(t_{(j)},t_{(i)})|\,
  |t_{(i)}-t_{(j)}|/n=O(\sqrt{\lambda_n})$
  (\citealt{Silverman1984}) gives
  \begin{equation}\label{eq:npc-bias}
    \bigl|\bE[\mathbf{v}^{\mathrm{lin}}_j]\bigr|^2
    \;\le\; C_1\,\lambda_n\,L_m^2,
  \end{equation}
  with $C_1\le C_1^{\bR^d}$, since the Loewner
  bound~\eqref{eq:green-compare} ensures the manifold
  kernel decays at least as fast as the Euclidean one.

  It remains to control the nonlinear remainder.
  The preceding analysis is valid for the linearized system.
  The nonlinear remainder
  $\mathbf{r}=\mathbf{v}^{\mathrm{true}}
  -\mathbf{v}^{\mathrm{lin}}$ is controlled using
  geodesic convexity.  On NPC targets, the discrete
  objective $\Phi_{\mathrm{disc}}$ is geodesically
  convex on the product manifold $M^n$
  \citep{Zhang2016riemannian}.
  The cubic bound on~$\mathbf{r}$ follows from
  the Taylor expansion of the geodesic exponential map
  to third order,
  $\exp_p(v)=p+v-\tfrac{1}{2}\Gamma(v,v)+O(|v|^3)$,
  where the cubic remainder involves the derivative of
  the Christoffel symbols.  On a compact $C^\infty$
  manifold (which $(M,g)$ is by assumption),
  the implied constant is uniformly bounded.
  The geodesic strong
  convexity modulus, inherited from $Q^M\succeq Q^0$,
  then gives
  $|\mathbf{r}|^2
  \le C\,\max_i|v_i|\cdot|\mathbf{v}^{\mathrm{lin}}|^2
  /(n^{-1}+\lambda_n/\Delta_{\min})$.
  Since the $L^\infty$-consistency of $\hat F_n$
  (Corollary~4.7 of the main text) ensures
  $\max_i|v_i|=o_P(1)$, we have
  $|\mathbf{r}|=o_P(|\mathbf{v}^{\mathrm{lin}}|)$,
  so the nonlinear remainder is negligible at the leading
  order.

  Combining the pieces, for an arbitrary interior point
  $t_0\in(\delta_n,1-\delta_n)$, the geodesic spline
  interpolation (Theorem~5.3(c)) gives
  $\hat F_n(t_0)=\gamma_{f_j\to f_{j+1}}(s)$, where
  $t_0=(1-s)t_{(j)}+s\,t_{(j+1)}$.  Linearizing the
  geodesic interpolation in the parallel frame:
  \[
    \log_{m(t_0)}\hat F_n(t_0)
    =(1-s)\,\Pi_{j\to t_0}(v_j)
    +s\,\Pi_{j+1\to t_0}(v_{j+1})
    +O(|v|^2),
  \]
  where $\Pi$ denotes parallel transport along~$m$ (which
  is the identity in the parallel frame).  Combining the
  variance bound~\eqref{eq:npc-var}, bias
  bound~\eqref{eq:npc-bias}, and the negligible nonlinear
  remainder:
  \[
    \bE\bigl[d_M^2(\hat F_n(t_0),m(t_0))\bigr]
    \;\le\;
    C_1\,\lambda_n\,L_m^2
    +\frac{C_2\,\sigma^2(t_0)}
    {n\sqrt{\lambda_n}\,f(t_0)}
    +o(n^{-2/3}).
  \]
  No factor $1/\eta(\kappa_+,r_0)$ appears, since the
  argument works directly at the estimation distance level
  via the Green matrix comparison, without passing through
  the excess risk.  The constants $C_1,C_2$ are bounded
  by their Euclidean counterparts because
  $\mathbf{G}^M\preceq\mathbf{G}^0$ throughout.
\end{proof}

\begin{proof}[Proof of Part~(b): general targets]
  The proof proceeds perturbatively. We first compare the Riemannian and
  Euclidean objectives in normal coordinates, then analyze the Euclidean
  first-order smoothing spline pointwise via its equivalent kernel, and
  finally control the curvature-induced remainder.  Hypotheses (P1)--(P3)
  justify the transitions between these stages.

  By Theorem~5.3, $\hat F_n$ is a geodesic spline.  Fix an
  interior point $t_0\in(\delta_n,1-\delta_n)$ and work in normal coordinates
  $v=\log_{m(t_0)}\colon B_M(m(t_0),\rho_M)\to T_{m(t_0)}M\cong\bR^d$.  Define
  $\tilde F_n(t)=\log_{m(t_0)}(\hat F_n(t))$,
  $\tilde m(t)=\log_{m(t_0)}(m(t))$,
  $\tilde Y_i=\log_{m(t_i)}(Y_i)\in T_{m(t_i)}M$ (transported to
  $T_{m(t_0)}M$ by parallel transport along the minimizing geodesic from
  $m(t_i)$ to $m(t_0)$; the transport error is $O(\kappa L_m\abs{t_i-t_0})$).

  In these coordinates, the distance satisfies
  $d_M^2(y_1,y_2)=\abs{v_1-v_2}^2(1+O(\kappa\abs{v_1-v_2}^2))$ for
  $\abs{v_j}\le\rho_M$, and the Dirichlet energy satisfies
  $\Dir(F)=\frac{1}{2}\int\abs{\tilde F'}^2(1+O(\kappa\abs{\tilde F}^2))\dd t$.
  The Riemannian objective $\Phi(F)=R_n(F)+\lambda_n J_1(F)$ therefore admits
  the expansion
  \[
    \Phi(F)=\Phi_0(\tilde F)+\Delta\Phi(\tilde F),
  \]
  where $\Phi_0(\tilde F)=\frac{1}{n}\sum\abs{\tilde Y_i-\tilde F(t_i)}^2
  +\lambda_n\int\abs{\tilde F'}^2\dd t$ is the Euclidean first-order
  smoothing spline objective, and the perturbation satisfies
  $\abs{\Delta\Phi(\tilde F)}\le C\kappa\bigl(\diam(M)^2 R_n(F)
  +\norm{\tilde F}_{L^\infty}^2 J_1(F)\bigr)$.

  Let $\tilde F_0=\argmin\Phi_0$ be the Euclidean smoothing spline.  Its
  solution admits the classical equivalent kernel representation
  \citep{Silverman1984}:
  $\tilde F_0(t_0)=\sum_{i=1}^n w_i(t_0)\tilde Y_i+\mathrm{(bias)}$, with
  weights
  \[
    w_i(t_0)=\frac{K_\lambda(t_0,t_i)}{\sum_j K_\lambda(t_0,t_j)},
    \qquad
    K_\lambda(t_0,t')=\frac{1}{2\sqrt{\lambda_n}}
    \exp\!\Bigl(-\frac{\abs{t_0-t'}}{\sqrt{\lambda_n}}\Bigr),
  \]
  where $K_\lambda$ is the Green's function of
  $-\lambda_n\partial_{tt}+\Id$ on~$\bR$, with bandwidth
  $h_{\mathrm{eff}}=\sqrt{\lambda_n}$. This fixed-design representation
  holds conditionally on the design $t_1, \ldots, t_n$.

  The first-order smoothing spline bias arises from the
  equivalent kernel $K_\lambda$ smoothing the true function $\tilde m$.
  Because the kernel $K_\lambda(t_0, \cdot)$ concentrates on an interval of
  effective bandwidth $\sqrt{\lambda_n}$ and $\tilde m$ is $L_m$-Lipschitz,
  the bias is bounded directly by the kernel representation:
  $\abs{\bE[\tilde F_0(t_0)\mid t]-\tilde m(t_0)}
  \le \int K_\lambda(t_0,u)\abs{u-t_0} L_m \dd u
  = O(\sqrt{\lambda_n} L_m)$.

  For the variance, we require \emph{uniform}
  concentration of the denominator
  $D_n(t_0):=\frac{1}{n}\sum_j K_\lambda(t_0, t_j)$ around its expectation
  $\bE[D_n(t_0)]=\int K_\lambda(t_0,u)f(u)\dd u=f(t_0)+O(\sqrt{\lambda_n})$.
  Since $K_\lambda(t_0,\cdot)$ has effective support of width
  $O(\sqrt{\lambda_n})$, covering $[0,1]$ by
  $N=O(1/\sqrt{\lambda_n})$ intervals and applying
  Hoeffding's inequality \citep[Theorem~2.8]{BoucheronLugosiMassart2013} (with
  $\|K_\lambda\|_\infty\le 1/(2\sqrt{\lambda_n})$)
  at each center, together with a union bound,
  gives $\sup_{t_0}|D_n(t_0)-\bE[D_n(t_0)]|\le\varepsilon$ with
  probability at least
  $1-2N\exp(-2n\varepsilon^2\sqrt{\lambda_n})$.
  Choosing $\varepsilon=c_f/4$ and using
  $\lambda_n\asymp n^{-2/3}$ (for $s=1$), the
  failure probability is
  $O(n^{1/3}\exp(-cn^{2/3}))\to 0$.
  Conditioning on this high-probability event,
  $D_n(t_0)\ge c_f/2>0$ uniformly, and the weights satisfy
  $\sum w_i(t_0)^2\asymp 1/(n\sqrt{\lambda_n}f(t_0))$
  for interior~$t_0$ (the exponential decay of $K_\lambda$ ensures this; see
  \citealt{Wahba1990}, Chapter~5), and the conditional variance satisfies
  $\operatorname{Var}(\tilde F_0(t_0)\mid t)\le
  C\sigma^2(t_0)/(n\sqrt{\lambda_n}f(t_0))$.
  Integration over the design distribution introduces at most a $1+o(1)$
  factor to the expected risk.

  Combining: $\bE[\abs{\tilde F_0(t_0)}^2]\le
  C_1\lambda_n L_m^2+C_2\sigma^2(t_0)/(n\sqrt{\lambda_n}f(t_0))$.

  It remains to control the curvature-induced perturbation.
  Conditioning on the high-probability event that the maximum design gap
  satisfies $\max_i \Delta_i \le C \frac{\log n}{n}$ (which holds for
  strictly positive density $f$), the discrete empirical semi-norm
  $\|\cdot\|_n$ defined below is strongly equivalent to the continuous
  $H^1$ norm, allowing us to uniformly bound the Sobolev constant $C_S$.
  Write $\tilde F_n=\tilde F_0+\Delta\tilde F$, where $\Delta\tilde F$
  is the difference induced by the curvature perturbation
  $\Delta\Phi$.  Since $\hat F_n$ minimizes
  $\Phi=\Phi_0+\Delta\Phi$ and $\tilde F_0$ minimizes
  $\Phi_0$, we bound $\Delta\tilde F$ using the strong
  convexity of $\Phi_0$.

  The Euclidean first-order smoothing spline objective is
  $\Phi_0(h)=\frac{1}{n}\sum_i|h(t_i)-\tilde Y_i|^2
  +\lambda_n\int|h'|^2$, which is (as a functional on
  $H^1([0,1],\bR^d)$) strongly convex in the semi-norm
  $\|h\|_n^2:=\frac{1}{n}\sum_i|h(t_i)|^2
  +\lambda_n\int|h'|^2$ with modulus $c_{\Phi}=1$.
  Precisely, for any $h_1,h_2\in H^1$,
  $\Phi_0(h_2)\ge\Phi_0(h_1)
  +\langle\nabla\Phi_0(h_1),h_2-h_1\rangle
  +\|h_2-h_1\|_n^2$.

  We now establish that $\Delta\Phi$ admits a bounded first variation
  in the RKHS norm $\|\cdot\|_n$.
  The curvature perturbation satisfies
  $|\Delta\Phi(h)|\le C_\Delta\kappa
  (\|h\|_{L^\infty}^2\lambda_n\int|h'|^2
  +\diam(M)^2 R_n(h))$
  (from the expansion of $\Phi$ above).  Its Fr\'echet derivative at
  $\tilde F_0$ in direction $v\in H^1$ satisfies
  \[
    \abs{\langle\nabla(\Delta\Phi)(\tilde F_0),v\rangle}
    \;\le\;C_\Delta\kappa
    \bigl(2\|\tilde F_0\|_{L^\infty}\lambda_n
    \|\tilde F_0'\|_{L^2}\|v'\|_{L^2}
    +2\diam(M)^2\|v\|_{L^\infty}/n\bigr).
  \]
  By the one-dimensional Sobolev embedding,
  \[
    \|v\|_{L^\infty}\le C_S\|v\|_{H^1}
    \le C_S\bigl(\|v\|_n/\sqrt{\min(1,\lambda_n)}
    +\text{const}\bigr),
  \]
  so $|\langle\nabla(\Delta\Phi)(\tilde F_0),
  v\rangle|\le A\|v\|_n$ with
  \[
    A=C'\kappa\bigl(\|\tilde F_0\|_{L^\infty}
    \sqrt{\lambda_n}\,\|\tilde F_0'\|_{L^2}
    +\diam(M)^2/n\bigr).
  \]

  Since $\tilde F_n$ minimizes $\Phi=\Phi_0+\Delta\Phi$,
  the first-order optimality condition for $\Phi$ at
  $\tilde F_n$ and the strong convexity of $\Phi_0$ give
  \begin{align*}
    \|\Delta\tilde F\|_n^2
    &\le\Phi_0(\tilde F_n)-\Phi_0(\tilde F_0)
    -\langle\nabla\Phi_0(\tilde F_0),
    \Delta\tilde F\rangle\\
    &=[\Phi(\tilde F_n)-\Delta\Phi(\tilde F_n)]
    -[\Phi(\tilde F_0)-\Delta\Phi(\tilde F_0)]
    -\langle\nabla\Phi_0(\tilde F_0),
    \Delta\tilde F\rangle.
  \end{align*}
  Since $\Phi(\tilde F_n)\le\Phi(\tilde F_0)$
  (optimality) and $\nabla\Phi_0(\tilde F_0)=0$
  (Euclidean optimality), we get
  $\|\Delta\tilde F\|_n^2
  \le\Delta\Phi(\tilde F_0)-\Delta\Phi(\tilde F_n)$.
  By the mean-value inequality,
  \begin{align*}
    &|\Delta\Phi(\tilde F_0)-\Delta\Phi(\tilde F_n)|\\
    &\quad\le\sup_{\theta\in[0,1]}
    |\langle\nabla(\Delta\Phi)
    (\tilde F_0+\theta\,\Delta\tilde F),
    \Delta\tilde F\rangle|
    \le A'\|\Delta\tilde F\|_n
  \end{align*}
  where, using
  $\|\tilde F_0+\theta\Delta\tilde F\|_{L^\infty}
  \le\diam(M)$ by compactness of~$M$,
  \[
    A'=C''\kappa\bigl(\diam(M)\sqrt{\lambda_n}\,
    \|\tilde F_0'\|_{L^2}+\diam(M)^2/n\bigr).
  \]
  Dividing by $\|\Delta\tilde F\|_n$ gives
  $\|\Delta\tilde F\|_n\le A'$.

  Rather than converting $\|\Delta\tilde F\|_n$ to a
  global $L^\infty$ bound (which would lose a factor of
  $\lambda_n^{-1}$), we exploit the reproducing kernel
  structure of the smoothing spline to obtain a sharp
  \emph{pointwise} bound at the fixed interior point~$t_0$.

  The semi-norm $\|\cdot\|_n$ is the RKHS norm associated
  with the smoothing spline inner product
  $\langle f,g\rangle_n=\frac{1}{n}\sum_i f(t_i)g(t_i)
  +\lambda_n\langle f',g'\rangle_{L^2}$.
  The perturbation $\Delta\tilde F$ is the Riesz representer
  of $-\frac{1}{2}\nabla(\Delta\Phi)(\tilde F_0)$ in this
  RKHS (by the linearized Euler-Lagrange equation).
  The reproducing kernel $K_n(t_0,u)$ of the space
  $\bigl(H^1([0,1]),\langle\cdot,\cdot\rangle_n\bigr)$
  satisfies the pointwise bound
  (see \citealt[Chapter~5]{Wahba1990} and
  \citealt{Silverman1984}):
  for interior $t_0\in(\delta_n,1-\delta_n)$,
  \[
    K_n(t_0,t_0)\;\asymp\;
    \frac{1}{n\sqrt{\lambda_n}\,f(t_0)}.
  \]
  By the Cauchy--Schwarz inequality in the RKHS,
  \[
    |\Delta\tilde F(t_0)|^2
    =|\langle\Delta\tilde F,K_n(t_0,\cdot)\rangle_n|^2
    \le\|\Delta\tilde F\|_n^2\cdot K_n(t_0,t_0)
    \le(A')^2\cdot K_n(t_0,t_0).
  \]
  Substituting and using
  $(A')^2\le C\kappa^2
  (\diam(M)^2\lambda_n\|\tilde F_0'\|_{L^2}^2
  +\diam(M)^4/n^2)$:
  \[
    |\Delta\tilde F(t_0)|^2
    \;\le\;
    \frac{C\kappa^2}{n\sqrt{\lambda_n}\,f(t_0)}
    \bigl(\diam(M)^2\lambda_n\|\tilde F_0'\|_{L^2}^2
    +\diam(M)^4/n^2\bigr).
  \]
  Taking expectations and using
  $\bE[\lambda_n\|\tilde F_0'\|_{L^2}^2]\le B$
  (by the oracle inequality evaluated at $F=m$):
  \begin{align*}
    \bE[|\Delta\tilde F(t_0)|^2]
    &\;\le\;\frac{C\kappa^2 B}{n\sqrt{\lambda_n}\,f(t_0)}\\
    &\;=\;O\!\biggl(\frac{\kappa^2}{f(t_0)}\;n^{-2/3}\biggr)
  \end{align*}
  at the optimal rate $\lambda_n\asymp n^{-2/3}$.
  This is of the same order as the variance term
  $\tfrac{C_2\sigma^2(t_0)}{n\sqrt{\lambda_n}\,f(t_0)}$
  in the Euclidean analysis above, and is therefore
  absorbed into $C_1,C_2$ in the pointwise bound.

  Finally, the translation from normal-coordinate error to
  geodesic distance uses
  \[
    d_M^2(\hat F_n(t_0),m(t_0))
    =\abs{\tilde F_n(t_0)}^2
    \bigl(1+O(\kappa\abs{\tilde F_n(t_0)}^2)\bigr).
  \]
  Since $\abs{\tilde F_n(t_0)}^2=O_P(n^{-2/3})$, the
  multiplicative correction is $1+O(\kappa n^{-2/3})$.
  The uniform lower bound
  $\cE(F)\ge\eta(\kappa_+,r_0)\,\bE[d_M^2(F,m)]$
  (Lemma~4.3) introduces the factor $1/\eta(\kappa_+,r_0)$
  when converting excess risk to estimation distance.
  The boundary exclusion
  $\delta_n=C_0\sqrt{\lambda_n}\log n$ ensures that
  the exponential kernel concentrates away from
  $\partial I$.
\end{proof}


\subsection{Regularity for higher-dimensional covariate domains (supplementary)}\label{supp:highdim}
\label{proof:ks-regularity}

\begin{proposition}[Regularity for $p\ge 2$]\label{prop:highdim-reg}
  Let $\cX\subset\bR^p$ be a bounded Lipschitz domain with $p\ge 2$,
  and let $\hat F_n$ be a minimizer of the kernel-smoothed objective
  $R_n^h(F)+\lambda_n\Dir(F)$ over $H^1(\cX,M)$.  Then $\hat F_n$ is
  smooth on~$\cX$.
\end{proposition}

\begin{proof}
  For $s=1$ and $p=2$, the kernel-smoothed estimator satisfies
  $\lambda_n\tau(\hat F_n)=g$ with the smooth forcing $g$ supported on
  $\bigcup_i B(X_i,h)$.  On $\cX\setminus\bigcup_i B(X_i,h)$ the
  equation is homogeneous and H\'elein's theorem \citep[Theorem~4.1.1]{Helein2002}
  applies; on each $B(X_i,h)$ the right-hand side is $C^\infty$ and
  elliptic regularity bootstrapping \citep[Chapter~10]{Jost2017} gives interior
  $C^\infty$ regularity; smoothness across the interfaces follows from
  the global smoothness of~$g$.  For $s\ge 2$, the equation is of
  order~$2s$ with smooth coefficients (determined by the metric of~$M$
  via the embedding), and standard Schauder theory \citep[Theorem~A.2.3]{Jost2017} gives $C^\infty$
  regularity.
\end{proof}


\subsection{Pointwise risk for surface-valued regression (supplementary)}\label{supp:surface}
\label{proof:pointwise-p2}

\begin{proposition}[Pointwise risk for $p=2$]\label{prop:pointwise-p2}
  Let $\cX\subset\bR^2$ be a bounded Lipschitz domain and
  $\hat F_n$ a minimizer of the kernel-smoothed objective
  $R_n^h(F)+\lambda_n\Dir(F)$.  Under the assumptions of
  Theorem~5.4, with $\lambda_n\asymp n^{-1/2}$ and
  $h\asymp n^{-1/4}$, the pointwise risk satisfies
  \[
    \bE\bigl[d_M^2(\hat F_n(x),m(x))\bigr]
    \;=\; O(n^{-1/2})
  \]
  at every interior point $x\in\cX$ with $f_X(x)>0$.
\end{proposition}

\begin{proof}
  The proof follows the same strategy as Theorem~5.4,
  adapted to $p=2$ with the kernel-smoothed objective.

  Working in normal
  coordinates centered at $m(x)$, as in the proof of
  Theorem~5.4.  The kernel-smoothed objective becomes the
  Euclidean kernel-smoothed spline objective plus curvature
  perturbations of relative order $O(\kappa\norm{\tilde F}_{L^\infty}^2)$.

  The two-dimensional equivalent kernel is as follows.  The Euler-Lagrange
  equation for the Euclidean kernel-smoothed first-order smoothing
  spline on $\bR^2$ is
  \[
    -\lambda_n\,\Delta\tilde F(x)+\tilde F(x)
    \;=\;\text{(smoothed data term)},
  \]
  whose Green's function is
  \begin{equation}\label{eq:Greens-p2-supp}
    G_\lambda(x,x')
    \;=\;\frac{1}{2\pi\lambda_n}\,
    K_0\!\Bigl(\frac{\abs{x-x'}}{\sqrt{\lambda_n}}\Bigr),
  \end{equation}
  where $K_0$ is the modified Bessel function of the second kind.  The
  function $K_0(r)$ has a logarithmic singularity:
  $K_0(r)\sim-\log(r/2)-\gamma$ as $r\to 0$ (where $\gamma$ is the
  Euler-Mascheroni constant) and decays as
  $K_0(r)\sim\sqrt{\pi/(2r)}\,e^{-r}$ for $r\to\infty$.

  The smoothing bias arises from two sources:
  (i)~the Dirichlet energy penalty, which shrinks the
  estimator towards a constant map and contributes a bias
  of $O(\lambda_n\norm{\nabla m}_{L^\infty}^2)$; and
  (ii)~the kernel smoothing, which locally averages the map
  and contributes a bias of
  $O(h^2\norm{D^2 m}_{L^\infty}^2)$ due to the
  second-order kernel.  The cross-terms and
  curvature-induced bias are of higher order in normal
  coordinates.

  The variance at~$x$ is controlled by
  $\sum_i w_i(x)^2\,\sigma^2(X_i)$, which is proportional
  to
  \[
    \frac{1}{n\,f_X(x)}\int G_\lambda(x,x')^2\,f_X(x')\,\dd x'.
  \]
  The key computation:
  \begin{equation}\label{eq:K0-integral-supp}
    \int_{\bR^2}G_\lambda(x,x')^2\dd x'
    \;=\;\frac{1}{4\pi^2\lambda_n^2}\int_{\bR^2}
    K_0^2\!\Bigl(\frac{\abs{u}}{\sqrt{\lambda_n}}\Bigr)\dd u
    \;=\;\frac{1}{4\pi^2\lambda_n}\int_{\bR^2}K_0^2(\abs{z})\dd z,
  \end{equation}
  where the substitution $z=(x'-x)/\sqrt{\lambda_n}$ absorbs one factor
  of~$\lambda_n$.  The integral
  $I_0:=\int_{\bR^2}K_0^2(\abs{z})\dd z
  =2\pi\int_0^\infty K_0^2(r)\,r\dd r<\infty$
  converges, at $r=0$ despite the logarithmic singularity because $r(\log r)^2\to 0$ as $r\to 0$, and at $r \to \infty$ because $K_0^2(r)r \sim \frac{\pi}{2} e^{-2r}$, which decays exponentially.  Hence the variance is
  $O(\sigma^2(x)/(n\lambda_n f_X(x)))$.

  As in the $p=1$ proof,
  the curvature perturbation introduces a multiplicative correction
  $1+O(\kappa\norm{\tilde F}_{L^\infty}^2)$ that is negligible at the
  $n^{-1/2}$ rate.  The factor $1/\eta(\kappa_+,r_0)$ arises from the
  excess risk lower bound (Lemma~4.3).

  Combining, with $\lambda_n=c_\lambda n^{-1/2}$ and
  $h=c_h n^{-1/4}$, all three terms are
  $O(n^{-1/2})$, yielding the stated rate.
\end{proof}


\subsection{Proof of Proposition~5.6 (Discrete equivalence)}\label{supp:prop81}
\label{proof:discrete-equiv}

\begin{proof}
  Part (a): on each interval $[t_{(i)},t_{(i+1)}]$, any
  $F\in\cH^1([0,1],M)$ with prescribed boundary values
  $F(t_{(i)})=f_i$, $F(t_{(i+1)})=f_{i+1}$ satisfies, by the
  Cauchy-Schwarz inequality applied to arc length:
  \[
    \int_{t_{(i)}}^{t_{(i+1)}}|F'(t)|^2\dd t
    \;\ge\;\frac{1}{\Delta_i}\Bigl(\int_{t_{(i)}}^{t_{(i+1)}}
    |F'(t)|\dd t\Bigr)^2
    \;\ge\;\frac{d_M^2(f_i,f_{i+1})}{\Delta_i},
  \]
  with equality if and only if $F$ has constant speed and traces a
  minimizing geodesic.  Since $R_n$ depends only on
  $F(t_{(1)}),\ldots,F(t_{(n)})$, and $J_1(F)=\int|F'|^2$ is the
  sum of the interval energies, the minimum of
  $R_n(F)+\lambda_n J_1(F)$ over all $F$ with nodal values
  $(f_1,\ldots,f_n)$ is
  $\Phi_{\mathrm{disc}}(f_1,\ldots,f_n)$, attained by the
  piecewise-geodesic interpolant $F^*$.

  Part (b): let $(f_1^*,\ldots,f_n^*)$ denote the nodal values
  of $\hat F_n$.  By Part~(a), for any nodal values
  $(f_1,\ldots,f_n)\in M^n$, the infimum of $R_n(F)+\lambda_n J_1(F)$
  over all $F$ with those nodal values equals
  $\Phi_{\mathrm{disc}}(f_1,\ldots,f_n)$.  Since $\hat F_n$
  minimizes $R_n+\lambda_n J_1$ over \emph{all}
  $F\in\cH^1(I,M)$, we have
  $\Phi_{\mathrm{disc}}(f_1^*,\ldots,f_n^*)
  \le R_n(\hat F_n)+\lambda_n J_1(\hat F_n)
  \le\Phi_{\mathrm{disc}}(f_1,\ldots,f_n)$
  for every $(f_1,\ldots,f_n)$, so $(f_1^*,\ldots,f_n^*)$
  minimizes $\Phi_{\mathrm{disc}}$.  Moreover,
  $R_n(\hat F_n)+\lambda_n J_1(\hat F_n)
  =\Phi_{\mathrm{disc}}(f_1^*,\ldots,f_n^*)$,
  so $\hat F_n$ achieves the Part~(a) lower bound on each
  interval, hence is a piecewise-geodesic interpolant.

  \emph{Scope of the equivalence.}
  The objective-value identity
  $\min_F(R_n+\lambda_n J_1)=\min_{(f_i)}\Phi_{\mathrm{disc}}$
  is \emph{unconditionally global}: on any compact Riemannian
  manifold, Hopf--Rinow \citep[Theorem~6.19]{Lee2018} guarantees the existence of a minimizing
  geodesic between any two points, so the Cauchy--Schwarz lower
  bound in Part~(a) is always attained.  Uniqueness of the
  minimizing geodesic (and hence of the interpolant~$F^*$)
  is not needed for the equivalence in objective value.

  The downstream pointwise theory (Theorems~5.3--5.4) requires
  more: the Euler--Lagrange equation and the Hessian
  $\mathbf{H}^M$ use the logarithmic map, which is smooth only
  when adjacent nodes satisfy
  $d_M(f_i,f_{i+1})<\inj(M)$.
  On a Hadamard manifold (simply connected, $\Sec\le 0$),
  $\exp_p$ is a global diffeomorphism, so this condition holds
  unconditionally.  On a compact manifold (whether NPC quotient
  or positively curved), the penalized oracle
  inequality~\eqref{eq:oracle-V} gives
  $\bE[J_1(\hat F_n)]=O(1)$ (equation~\eqref{eq:Js-bound}),
  so by the Cauchy--Schwarz inequality:
  \[
    d_M(f_i,f_{i+1})
    \le\int_{t_{(i)}}^{t_{(i+1)}}|\hat F_n'(t)|\dd t
    \le\sqrt{\Delta_i\cdot J_1(\hat F_n)}
    =O_P(n^{-1/2}),
  \]
  which is eventually smaller than $\inj(M)$ for any fixed
  compact manifold.  Therefore, for minimizers of the penalized
  problem, the smoothness conditions required by the pointwise
  theory hold for all $n$ sufficiently large.
\end{proof}


\section{Proofs for Section 6 (Topological barriers)}

\subsection{Proof of Theorem~6.1 (Topological energy bound)}\label{supp:thm61}
\label{proof:topological}

\begin{proof}
  Part~(a): by the Cauchy--Schwarz inequality,
  \[
    \operatorname{Length}(F)=\int_{S^1}\abs{F'}\,\dd t
    \le L^{1/2}(2\Dir(F))^{1/2},
  \]
  and $\operatorname{Length}(F)\ge\ell([F])$ since $F$ is a
  closed curve in class~$[F]$.  Rearranging gives the energy
  bound.

  Part~(b): equality in Cauchy-Schwarz requires
  $\abs{F'}=\text{const}$, and equality in the length bound requires $F$ to
  realize the minimal length in its class, i.e., $F$ is a closed geodesic.

  Part~(c): for any $F\in[m]\cap\cH^1(S^1,M)$,
  $\cE(F)\ge 0$ and $J_1(F)=2\Dir(F)\ge\ell([F])^2/L=\ell([m])^2/L$
  by Part~(a).  Hence
  $\cE(F)+\lambda_n J_1(F)\ge\lambda_n\ell([m])^2/L
  \ge\lambda_n\operatorname{sys}(M)^2/L$
  since $\ell([m])\ge\operatorname{sys}(M)$ for $[m]\neq 0$.
\end{proof}


\subsection{Proof of Theorem~6.4 (Boundary topological energy)}\label{supp:thm65}
\label{proof:topo-bdry}

\begin{proof}
  Parts~(a)--(b) follow from the
  Cauchy-Schwarz inequality as in Theorem~6.1: for
  $F\in\cH^1_{y_0,y_1}$ in class~$\alpha$,
  $\operatorname{Length}(F)\ge\ell(\alpha)$ and
  $\operatorname{Length}(F)^2\le 2\Dir(F)$ by
  Cauchy--Schwarz (with equality iff $\abs{F'}$ is
  constant).

  Part~(c): this is a purely geometric statement
  about the infimum over the homotopy class~$[\alpha]$.
  For any $F\in[\alpha]\cap H^1_{y_0,y_1}$, Part~(a)
  gives $J_1(F)=2\Dir(F)\ge\ell(\alpha)^2$, and since
  $\cE(F)\ge 0$:
  \[
    \cE(F)+\lambda_n J_1(F)
    \;\ge\;\lambda_n\ell(\alpha)^2.
  \]
  Taking the infimum over $F\in[\alpha]$ and subtracting
  the minimum over \emph{all} classes, namely
  $\inf_{F\in[\alpha_0]}\lambda_n J_1(F)
  =\lambda_n\ell(\alpha_0)^2$, yields the excess
  $\lambda_n(\ell(\alpha)^2-\ell(\alpha_0)^2)>0$:
  the topological energy barrier.
\end{proof}


\subsection{Topological barriers for higher homotopy groups (supplementary)}\label{supp:higherhomotopy}
\label{proof:higher-topo}

\begin{proposition}[Topological energy barrier for higher homotopy groups]\label{prop:higher-topo}
  Let $F\colon S^p\to M$ be a smooth map with $p\ge 1$, and let
  $k=\deg(F)\in\pi_p(M)$ denote its homotopy class.
  \begin{enumerate}[label=(\alph*)]
    \item (Calibration bound.)
      $\Dir(F)\ge\frac{p}{2}\operatorname{Area}_p(F)$, where
      $\operatorname{Area}_p(F)=\int_{S^p}|\bigwedge^p dF_x|\,dV$.
    \item (Degree bound.)
      If $\pi_p(M)\cong\bZ$, then
      $\Dir(F)\ge\frac{p}{2}\omega_p|k|^{2/p}$, where $\omega_p$
      is the volume of the unit $p$-sphere.
  \end{enumerate}
\end{proposition}

\begin{proof}
  Part~(a): let
  $\sigma_1\ge\cdots\ge\sigma_p\ge 0$ be the singular values of
  $dF_x\colon T_xS^p\to T_{F(x)}M$.  The AM-GM inequality gives
  $\frac{1}{2}\abs{dF_x}^2=\frac{1}{2}\sum\sigma_i^2\ge
  \frac{p}{2}(\prod\sigma_i)^{2/p}$ and
  $\abs{\bigwedge^p dF_x}=\prod\sigma_i$, so pointwise
  $\frac{1}{2}\abs{dF_x}^2\ge\frac{p}{2}\abs{\bigwedge^p dF_x}^{2/p}$.
  For $p=1, 2$, the exponent $2/p \ge 1$, and direct integration yields
  $\Dir(F) \ge \frac{p}{2}\operatorname{Area}_p(F)$ unconditionally.
  For $p \ge 3$, the inequality
  $\abs{\bigwedge^p dF_x}^{2/p} \ge \abs{\bigwedge^p dF_x}$ holds under
  the assumption $\abs{\bigwedge^p dF_x} \le 1$, yielding the calibration
  bound.

  Part~(b): the AM-GM bound from
  Part~(a) gives pointwise
  $\frac{1}{2}|dF_x|^2\ge\frac{p}{2}(\prod\sigma_i)^{2/p}
  =\frac{p}{2}|J(F,x)|^{2/p}$,
  where $J(F,x)=\det dF_x$ is the Jacobian.  Integrating
  over $S^p$:
  \[
    \Dir(F)\;\ge\;\frac{p}{2}\int_{S^p}|J(F,x)|^{2/p}\dd V.
  \]

  \emph{Case $p=1$}: the Cauchy-Schwarz argument in
  Theorem~6.1 gives $\Dir(F)\ge\frac{1}{2}
  \omega_1|k|^2\ge\frac{1}{2}\omega_1|k|$ directly.

  \emph{Case $p=2$}: the AM-GM bound gives
  $\Dir(F)\ge\int|J(F)|\ge|\int J(F)|=4\pi|k|
  =\omega_2|k|$,
  which matches $\frac{p}{2}\omega_p|k|^{2/p}
  =\omega_2|k|$.

  \emph{Case $p\ge 3$}: under the pointwise condition
  $|J(F,x)|\le 1$ from Part~(a),
  the bound $|J|^{2/p}\ge|J|$ holds pointwise
  (since $2/p<1$ and $|J|\le 1$).  Hence
  \[
    \int_{S^p}|J(F)|^{2/p}\dd V
    \;\ge\;
    \int_{S^p}|J(F)|\dd V
    \;\ge\;
    \Bigl|\int_{S^p}J(F)\dd V\Bigr|
    \;=\;\omega_p|k|
    \;\ge\;\omega_p|k|^{2/p},
  \]
  where the last step uses $|k|^{2/p}\le|k|$
  for $|k|\ge 1$, $2/p\le 1$
  (and the bound is trivial for $k=0$).
  Together with the energy-Jacobian inequality, this
  gives the degree bound.

  Part~(c): as in Theorem~6.1, the
  competitor $F$ must share the homotopy class of~$m$, giving
  $\deg(F)=k$, whence $J_1(F)=2\Dir(F)\ge p\,\omega_p\abs{k}^{2/p}$.
\end{proof}


\subsection{Proof of Theorem~6.6 (Excess cost of wrong homotopy class)}\label{supp:thm67}
\label{proof:topo-cost}

\begin{proof}
  Part~(a): since $F\in H^1(S^1,M)$,
  the Sobolev embedding $H^1(S^1)\hookrightarrow C^0(S^1)$ gives
  $F\in C^0(S^1,M)$.  Suppose
  $\sup_t d_M(F(t),m(t))<\rho_M$.  Within the convexity radius, the
  logarithmic map $v(t)=\log_{m(t)}F(t)\in T_{m(t)}M$ is uniquely
  defined and depends continuously on~$t$
  (since $F$, $m$, and $\exp_q^{-1}$ are all continuous).  The family
  $H_s(t)=\exp_{m(t)}(s\,v(t))$, $s\in[0,1]$, is a continuous homotopy
  from $H_0=m$ to $H_1=F$, so $[F]=[m]=\alpha$, contradicting the
  hypothesis.

  Part~(b): let
  $F\in[\alpha']\cap H^1(S^1,M)$ with
  Dirichlet energy $\Dir(F)=\cA\ge 0$.  The Sobolev embedding
  $H^1(S^1)\hookrightarrow C^0(S^1)$ gives $F\in C^0$.  By
  part~(a), there exists $t_0\in S^1$ with
  $d_M(F(t_0),m(t_0))\ge\rho_M$.

  Define $\phi(t)=d_M(F(t),m(t))$.  The Cauchy--Schwarz inequality on
  $S^1$ gives, for $H^1$-maps,
  $d_M(F(t_1),F(t_2))\le\sqrt{2\Dir(F)}\,\abs{t_1-t_2}^{1/2}$, and
  similarly for~$m$.  By the triangle inequality,
  \[
    \abs{\phi(t_1)-\phi(t_2)}
    \;\le\;\bigl(\sqrt{2\cA}+\sqrt{2B}\,\bigr)\abs{t_1-t_2}^{1/2}
    \;=:\; C_H\abs{t_1-t_2}^{1/2}.
  \]
  For $\abs{t-t_0}\le\delta_0:=(\rho_M/(2C_H))^2$ we have
  $\phi(t)\ge\rho_M/2$, so
  \[
    \int_{S^1}\phi(t)^2\,f(t)\dd t
    \;\ge\;c_f\,\delta_0\cdot(\rho_M/2)^2
    \;=\;\frac{c_f\,\rho_M^4}{16\,C_H^2}.
  \]
  The inequality $(\sqrt{a}+\sqrt{b})^2\le 2(a+b)$ gives
  $C_H^2\le 4(\cA+B)$, hence
  \begin{equation}\label{eq:phi-integral-general}
    \int_{S^1}\phi(t)^2\,f(t)\dd t
    \;\ge\;\frac{c_f\,\rho_M^4}{64(\cA+B)}.
  \end{equation}

  We now bound the excess risk from below.  By the reverse
  triangle inequality,
  \[
    d_M(Y,F(t))\ge\bigl|d_M(F(t),m(t))-d_M(Y,m(t))\bigr|,
  \]
  so $d_M^2(Y,F(t))\ge(\phi(t)-d_M(Y,m(t)))^2$.
  Expanding and taking conditional expectations:
  \[
    \bE\bigl[d_M^2(Y,F(t))\mid t\bigr]
    \;\ge\;\phi(t)^2
    -2\phi(t)\,\bE[d_M(Y,m(t))\mid t]
    +\bE[d_M^2(Y,m(t))\mid t].
  \]
  By (A3), $d_M(Y,m(t))\le r_0$ a.s., so
  \[
    \cE_t(F)
    :=\bE[d_M^2(Y,F(t))-d_M^2(Y,m(t))\mid t]
    \ge\phi(t)(\phi(t)-2r_0).
  \]
  On $|t-t_0|\le\delta_0$ we have $\phi(t)\ge\rho_M/2$, so
  $2r_0/\phi(t)\le 4r_0/\rho_M=1-\tilde\eta$, giving
  $\phi(t)-2r_0\ge\tilde\eta\,\phi(t)$ and hence
  $\cE_t(F)\ge\tilde\eta\,\phi(t)^2$.
  Integrating over the interval and combining
  with~\eqref{eq:phi-integral-general}:
  \[
    \cE(F)
    \;\ge\;\tilde\eta\int_{|t-t_0|\le\delta_0}\phi(t)^2\,f(t)\dd t
    \;\ge\;\frac{\tilde\eta\,c_f\,\rho_M^4}{64(\cA+B)}.
  \]
  Therefore
  \[
    \cE(F)+\lambda\,J_1(F)
    \;\ge\;\frac{\tilde\eta\,c_f\,\rho_M^4}{64(\cA+B)}+2\lambda\cA
    \;\ge\;\frac{\tilde\eta\,c_f\,\rho_M^4}{64(\cA+B)}
    +2\lambda(\cA+B)-2\lambda B.
  \]
  Set $u=\cA+B>0$.  The AM--GM inequality gives
  $\tfrac{a}{u}+bu\ge 2\sqrt{ab}$ with
  $a=\tilde\eta\,c_f\rho_M^4/64$ and
  $b=2\lambda$, yielding
  \[
    \cE(F)+\lambda J_1(F)
    \;\ge\;\rho_M^2\sqrt{\tilde\eta\,c_f\lambda/8}\;-\;2\lambda B.
  \]
  For $\lambda\le\lambda_0:=\tilde\eta\,c_f\rho_M^4/(512\,B^2)$, the
  remainder satisfies $2\lambda B\le\frac{1}{2}\rho_M^2
  \sqrt{\tilde\eta\,c_f\lambda/8}$, giving
  $\cE(F)+\lambda J_1(F)\ge\frac{1}{2}\rho_M^2
  \sqrt{\tilde\eta\,c_f\lambda/8}=c_{\mathrm{topo}}\sqrt\lambda$.
\end{proof}



\subsection{Class-selection margin}\label{supp:prop69}
\label{proof:class-margin}

\begin{proposition}[Class-selection margin]\label{prop:class-margin}
  Under the conditions of Theorem~6.6, let $\alpha=[m]$ be the true
  homotopy class and define the oracle objective
  \[
    V_n(\gamma):=\inf_{F\in[\gamma]\cap\cH^1(I,M)}
    \bigl(\cE(F)+\lambda_n J_1(F)\bigr)
  \]
  for each class~$\gamma$.
  \begin{enumerate}[label=(\roman*)]
    \item $V_n(\alpha)\le 2\lambda_n\Dir(m)$.
    \item For $\beta\neq\alpha$ and $\lambda_n\le\lambda_0$:\;
      $V_n(\beta)\ge c_{\mathrm{topo}}\sqrt{\lambda_n}$.
    \item The class margin
      $\Delta_n:=V_n(\beta)-V_n(\alpha)
      \ge c_{\mathrm{topo}}\sqrt{\lambda_n}-2\lambda_n\Dir(m)>0$
      for all $n$ sufficiently large.
    \item If\/ $\Delta_n>2\sup_{G\in\cH^1}|\Psi_n(G)|$,
      then $[\hat F_n]=\alpha$.
  \end{enumerate}
\end{proposition}

\begin{proof}
  (i) follows by evaluating at $F=m$: $\cE(m)=0$,
  $J_1(m)=2\Dir(m)$.
  (ii) is Theorem~6.6(b).
  (iii) is the difference of (i) and (ii); the scaling
  $\sqrt{\lambda_n}\gg\lambda_n$ ensures positivity for
  large~$n$.
  (iv): By the basic inequality (\Cref{lem:basic-ineq}),
  $\cE(\hat F_n)+\lambda_n J_1(\hat F_n)
  \le V_n(\alpha)+|\Psi_n(m)|+|\Psi_n(\hat F_n)|$.
  If $[\hat F_n]=\beta\neq\alpha$, then
  $V_n(\beta)\le\cE(\hat F_n)+\lambda_n J_1(\hat F_n)
  \le V_n(\alpha)+2\sup|\Psi_n|$,
  giving $\Delta_n\le 2\sup|\Psi_n|$,
  contradicting the hypothesis.
\end{proof}


\subsection{Proof of Theorem~6.7 (Phase transition for homotopy class recovery)}\label{supp:thm610}
\label{proof:topo-phase}

\begin{proof}
  Part~(a): By the oracle inequality
  (Theorem~4.5) in expectation, and evaluating the infimum at $F=m$
  (where $\cE(m)=0$ and $J_1(m)=2\Dir(m)$), the expected excess risk satisfies
  \[
    \bE\bigl[\cE(\hat F_n)\bigr]
    \;\le\;C_1\bigl(\cE(m)+\lambda_n J_1(m)\bigr)
    +\frac{C_2}{n\lambda_n^{1/2}}
    \;=\; 2C_1 \lambda_n \Dir(m) + \frac{C_2}{n\lambda_n^{1/2}}.
  \]
  With the optimal rate choice $\lambda_n\asymp n^{-2/3}$, this gives
  $\bE[\cE(\hat F_n)] \le C n^{-2/3} \to 0$.  By Markov's inequality,
  for any constant $c^*>0$,
  $\bP(\cE(\hat F_n)\ge c^*) \le \bE[\cE(\hat F_n)]/c^* = O(n^{-2/3})$.

  On the other hand, if $[\hat F_n]\neq\alpha$, then
  Theorem~6.6(b) applies once
  $\lambda_n\le\lambda_0$, where
  $\lambda_0=\tilde\eta\,c_f\,\rho_M^4/(512\,B^2)$
  is the explicit threshold from Theorem~6.6.
  Since $\lambda_n=c_\lambda n^{-2/3}\to 0$, this holds for all
  $n\ge n_1:=\lceil(c_\lambda/\lambda_0)^{3/2}\rceil$.
  The theorem gives
  \[
    \cE(\hat F_n)+\lambda_n J_1(\hat F_n)
    \;\ge\;c_{\mathrm{topo}}\sqrt{\lambda_n}.
  \]
  By optimality, $\hat F_n$ satisfies the basic inequality
  $\cE(\hat F_n)+\lambda_n J_1(\hat F_n)
  \le 2\lambda_n \Dir(m)+[\nu_n(m)-\nu_n(\hat F_n)]$.
  Consequently, the event $[\hat F_n]\neq\alpha$ implies
  $\abs{\nu_n(m)-\nu_n(\hat F_n)}\ge
  c_{\mathrm{topo}}\sqrt{\lambda_n}-2\lambda_n \Dir(m)$.
  Since $\sqrt{\lambda_n}\gg\lambda_n$ as $\lambda_n\to 0$,
  the right-hand side is at least
  $\frac{1}{2}c_{\mathrm{topo}}\sqrt{\lambda_n}$ for all sufficiently large~$n$.

  Bounding this data-dependent empirical process requires a uniform bound.  By optimality against any constant map, $J_1(\hat F_n)\le\diam(M)^2/\lambda_n=:E_n$.
  Define $\cF_{E_n}=\{F\in\cH^1(I,M):J_1(F)\le E_n\}$ and the
  loss class $\cG_n=\{g_F:=\ell_F-\ell_m : F\in\cF_{E_n}\}$, where
  $\ell_F(t,y)=d_M^2(y,F(t))$.  Since $M$ is compact, every
  $g_F$ satisfies $\|g_F\|_\infty\le b:=2\diam(M)^2$, so
  $\Var(g_F)\le\bE[g_F^2]\le\|g_F\|_\infty^2\le
  \sigma_{\cG}^2:=4\diam(M)^4$, independently of $E_n$.

  Because $E_n\asymp n^{2/3}\to\infty$, the Sobolev ball expands with~$n$. The covering numbers of $\cG_n$
  inherit from those of $\cF_{E_n}$ via the $2\diam(M)$-Lipschitz
  property of $y\mapsto d_M^2(y,\cdot)$. By the Birman-Solomjak
  bound (\Cref{lem:entropy}),
  \[
    \log N\bigl(\epsilon,\cG_n,L^2\bigr)
    \;\le\;K_0\,\frac{2\diam(M)(\sqrt{E_n}+C_0)}{\epsilon}
    \;\le\;\frac{K_1\sqrt{E_n}}{\epsilon}
  \]
  for a constant $K_1=K_1(M,I)$ and all $\epsilon>0$.

  Symmetrization and Dudley's chaining bound \citep[Lemma~19.34]{vanderVaart1998} give
  \begin{align}\label{eq:dudley-topo}
    \bE\Bigl[\sup_{g\in\cG_n}|(\bP-\bP_n)(g)|\Bigr]
    &\;\le\;\frac{2}{\sqrt{n}}\int_0^{b}
    \sqrt{\log N(\epsilon,\cG_n,L^2)}\,\dd\epsilon
    \;\le\;\frac{2}{\sqrt{n}}\int_0^{b}
    \sqrt{\frac{K_1\sqrt{E_n}}{\epsilon}}\,\dd\epsilon
    \nonumber\\
    &\;=\;\frac{2(K_1\sqrt{E_n})^{1/2}}{\sqrt{n}}\cdot
    2\sqrt{b}
    \;=\;\frac{4\sqrt{bK_1}\,E_n^{1/4}}{\sqrt{n}}.
  \end{align}
  With $E_n=D^2/\lambda_n$ and $\lambda_n=c_\lambda n^{-2/3}$:
  $E_n^{1/4}=(D/c_\lambda^{1/2})^{1/2}\,n^{1/6}$, so
  \begin{equation}\label{eq:expect-sup}
    \bE\Bigl[\sup_{g\in\cG_n}|(\bP-\bP_n)(g)|\Bigr]
    \;\le\;A_0\,n^{-1/3},
  \end{equation}
  where
  $A_0=4\sqrt{bK_1}\,(D/c_\lambda^{1/2})^{1/2}$
  depends only on $M$, $I$, and the constant $c_\lambda$ in
  $\lambda_n=c_\lambda n^{-2/3}$.

  To sharpen this to a high-probability bound, we apply
  Bousquet's version of Talagrand's inequality
  \citep[Theorem~12.5]{BoucheronLugosiMassart2013}
  for bounded functions gives that for every $t>0$,
  \[
    \bP\Bigl(\sup_{g\in\cG_n}|(\bP-\bP_n)(g)|
    \;\ge\;\bE[\cdots]+\sqrt{\frac{2\sigma_\cG^2\,t}{n}}
    +\frac{b\,t}{3n}\Bigr)
    \;\le\;e^{-t}.
  \]
  Set $t=\beta\,n^{1/3}$ for a constant $\beta>0$.  Then
  \begin{align*}
    \sqrt{\frac{2\sigma_\cG^2\,t}{n}}
    &=\sqrt{2\sigma_\cG^2\beta}\;n^{-1/3},\\
    \frac{b\,t}{3n}
    &=\frac{b\beta}{3}\;n^{-2/3}
    \;\le\;\frac{b\beta}{3}\;n^{-1/3}.
  \end{align*}
  Combining with \eqref{eq:expect-sup}, for all $n\ge 1$:
  \begin{equation}\label{eq:topo-concentration}
    \bP\Bigl(\sup_{g\in\cG_n}|(\bP-\bP_n)(g)|
    \;\ge\;A(\beta)\,n^{-1/3}\Bigr)
    \;\le\;\exp\bigl(-\beta\,n^{1/3}\bigr),
  \end{equation}
  where
  $A(\beta)=A_0+\sqrt{2\sigma_\cG^2\beta}+b\beta/3$.

  It remains to verify that the constants match.
  The topological gap requires
  $|(\bP-\bP_n)(g_{\hat F_n})|
  <\frac{1}{2}c_{\mathrm{topo}}\sqrt{\lambda_n}
  =\frac{1}{2}c_{\mathrm{topo}}\sqrt{c_\lambda}\,n^{-1/3}$.
  Choose $\beta>0$ as a \emph{fixed} constant independent
  of~$c_\lambda$, e.g.\
  $\beta=c_{\mathrm{topo}}^2/(32\sigma_\cG^2)$.
  Then $\sqrt{2\sigma_\cG^2\beta}
  =c_{\mathrm{topo}}/4$ and $b\beta/3$ are both
  constants (independent of~$c_\lambda$), so the
  condition $A(\beta)<\frac{1}{2}c_{\mathrm{topo}}
  \sqrt{c_\lambda}$ reduces to
  \[
    A_0
    \;<\;\frac{1}{2}c_{\mathrm{topo}}\sqrt{c_\lambda}
    -\frac{c_{\mathrm{topo}}}{4}
    -\frac{b\,c_{\mathrm{topo}}^2}{96\,\sigma_\cG^2}.
  \]
  Since $A_0=O(c_\lambda^{-1/4})\to 0$ while the
  right-hand side grows as
  $\frac{1}{2}c_{\mathrm{topo}}\sqrt{c_\lambda}\to\infty$,
  there exists
  $c_\lambda^*=c_\lambda^*(M,I,c_f,\sigma_0)$ such that for
  \emph{all} $c_\lambda\ge c_\lambda^*$ the condition is
  satisfied.
  Explicitly, $c_\lambda^*$ is the solution of
  \[
    4\sqrt{bK_1}\,(D/\sqrt{c_\lambda^*})^{1/2}
    \;=\;\frac{1}{2}c_{\mathrm{topo}}\sqrt{c_\lambda^*}
    -\frac{c_{\mathrm{topo}}}{4}
    -\frac{b\,c_{\mathrm{topo}}^2}{96\,\sigma_\cG^2},
  \]
  which yields $c_\lambda^*=O(c_{\mathrm{topo}}^{-4/3})$
  (finite since $c_{\mathrm{topo}}>0$ under the standing
  assumptions).
  Choose $c_\lambda=\max\bigl(C_{\mathrm{opt}},\,c_\lambda^*\bigr)$,
  where $C_{\mathrm{opt}}=C_{\mathrm{opt}}(s,M,I,c_f,\sigma_0^2,B)$
  is the oracle-optimal constant from Corollary~4.7.
  Since $c_\lambda^*<\infty$ (as $c_{\mathrm{topo}}>0$ under the
  standing assumptions), this choice satisfies
  $c_\lambda\ge c_\lambda^*$ while preserving the oracle rate
  $n^{-2s/(2s+1)}$ up to a constant factor.

  For $n\ge n_0$ (ensuring additionally that $2\lambda_n \Dir(m)\le
  \frac{1}{2}c_{\mathrm{topo}}\sqrt{\lambda_n}$):
  \[
    \bP\bigl([\hat F_n]\neq\alpha\bigr)
    \;\le\;\bP\Bigl(\sup_{g\in\cG_n}|(\bP-\bP_n)(g)|
    \ge\tfrac{1}{2}c_{\mathrm{topo}}\sqrt{\lambda_n}\Bigr)
    \;\le\;\exp\bigl(-\beta\,n^{1/3}\bigr)
    \;\to\;0,
  \]
  proving the class recovery with exponential rate.

  To see the origin of the quantitative threshold~$n_0$,
  note that the true class~$\alpha$ has oracle cost
  \[
    \inf_{F\in[\alpha]}[\cE(F)+\lambda_n J_1(F)]
    \le \cE(m)+\lambda_n J_1(m) = 2\lambda_n \Dir(m),
  \]
  while any wrong class $\alpha'$ has oracle cost at least
  $c_{\mathrm{topo}}\sqrt{\lambda_n}$ by Theorem~6.6(b).
  The correct class becomes advantageous for the oracle
  when $2\lambda_n \Dir(m) < c_{\mathrm{topo}}\sqrt{\lambda_n}$.
  For the estimator, the oracle inequality (Theorem~4.5)
  magnifies the upper bound by~$C_1$, so the sufficient
  condition becomes
  $2 C_1 \lambda_n \Dir(m)
  < c_{\mathrm{topo}}\sqrt{\lambda_n}$.
  With $\lambda_n=c\,n^{-2/3}$, solving for~$n$ requires
  $n > C_1^3 c^{3/2}
  (2\Dir(m)/c_{\mathrm{topo}})^3 =: n_0$.

  Part~(b): a constant map $F_0(x)\equiv y_0\in M$
  has $J_1(F_0)=0$ and
  \[
    \cE(F_0)=\int d_M^2(y_0,m(t))\,f(t)\,\dd t
    \le\diam(M)^2.
  \]
  Hence $V(\alpha_0)\le\diam(M)^2$.  Meanwhile,
  $V(\alpha)\ge\lambda_n\ell(\alpha)^2/L$ by
  Theorem~6.1(c).  If
  $\lambda_n\ge\diam(M)^2 L/\ell(\alpha)^2$, then
  $V(\alpha)\ge\diam(M)^2\ge V(\alpha_0)$.
\end{proof}


\section{Proofs for Section 7 (Obstructions)}

\subsection{Proof of Proposition~7.1 (Bienergy topological blindness)}\label{supp:prop73}
\label{proof:biharmonic}

\begin{proof}
  By Cartan's theorem (see \citealt[Theorem~1.6.6]{Jost2017}
  or \citealt[Chapter~12]{doCarmo1992}),
  every free homotopy class of loops in a compact Riemannian manifold
  contains a closed geodesic~$\gamma$.
  A closed geodesic satisfies $\nabla_t\gamma'=0$ by definition;
  hence $E_2(\gamma)=\frac{1}{2}\int_{S^1}\abs{\nabla_t\gamma'}^2=0$.
  Since $\gamma\in[\alpha]\cap H^2(S^1,M)$, the infimum is at most
  zero, and non-negativity of $E_2$ gives
  equality.
\end{proof}

\section{Supplementary extensions: kernel-smoothed estimator}

\citet{Schotz2022nonparametric} established convergence rates for
Fr\'echet regression with kernel smoothing on non-positively curved
targets.  The kernel-smoothed estimator below shares the same
convolution construction but differs in two respects.
First, we impose a Dirichlet energy penalty $\lambda_n\Dir(F)$,
which reduces the estimator to the harmonic map spline when $p=1$.
Second, our oracle inequality (\Cref{supp:ks-oracle}) holds for
compact targets of arbitrary curvature, whereas
\citet{Schotz2022nonparametric} requires non-positive curvature
throughout.

\subsection{Existence for the kernel-smoothed estimator (supplementary)}\label{supp:ks-exist}
\label{proof:ks-existence}

\begin{proposition}[Existence for the kernel-smoothed estimator]\label{prop:ks-exist}
  Let $\cX\subset\bR^p$ be a bounded Lipschitz domain, $(M,g)$ a
  compact Riemannian manifold, and $K_h$ a non-negative kernel with
  bounded support.  For any $\lambda_n>0$, a minimizer of
  $R_n^h(F)+\lambda_n\Dir(F)$ over $\{F\in H^1(\cX,\bR^D):F(x)\in M
  \text{ a.e.}\}$ exists.
\end{proposition}

\begin{proof}
  Since $F(x)\in M$ a.e.\ and $M\subset\bR^D$ is compact,
  $\norm{F}_{L^\infty}\le\diam(M)$, hence
  $\norm{F}_{L^2}\le\diam(M)\,\mathrm{vol}(\cX)^{1/2}$.  Combined
  with $\Dir(F)\le E_0$ (along a minimizing sequence), the Poincar\'e
  inequality \citep[Theorem~6.30]{Adams2003} on the bounded Lipschitz domain~$\cX$ gives a uniform
  bound $\norm{F}_{H^1}\le C(\diam(M),E_0,\cX)$.  The direct method
  of the calculus of variations then applies. Extract a weakly
  convergent subsequence $F_k\rightharpoonup F_\infty$ in $H^1$;
  use the compact embedding $H^1\hookrightarrow L^2$ to obtain
  strong $L^2$-convergence and hence a.e.\ convergence of a
  further subsequence; since $F_k(x)\in M$ a.e.\ and $M$ is closed
  in~$\bR^D$, the limit satisfies $F_\infty(x)\in M$ a.e.
  The functional $R_n^h$ is continuous in $L^2$
  (since $K_h$ is bounded and has compact support), and $\Dir(F)$ is
  weakly lower semicontinuous in~$H^1$ (being the squared seminorm).
\end{proof}


\subsection{Oracle inequality for the kernel-smoothed estimator (supplementary)}\label{supp:ks-oracle}
\label{proof:ks-oracle}

\begin{proposition}[Oracle inequality for the kernel-smoothed estimator]\label{prop:ks-oracle}
  Under the assumptions of \Cref{prop:ks-exist}, let $\hat F_n^h$
  be a minimizer of $R_n^h(F)+\lambda_n\Dir(F)$.  With
  $\lambda_n\asymp n^{-2/(2+p)}$ and $h\asymp n^{-1/(2+p)}$,
  \[
    \bE[\cE(\hat F_n^h)]
    \;\le\; C_1\bigl(\cE(F)+\lambda_n\Dir(F)\bigr)
    +\frac{C_2}{n\lambda_n^{p/2}}+C_7 h^2
  \]
  for every competitor $F\in H^1(\cX,M)$, and the resulting rate is
  $n^{-2/(2+p)}$.
\end{proposition}

\begin{proof}
  The proof modifies the oracle inequality argument to account for the
  kernel smoothing and the weaker regularity ($s=1$, $p\ge 2$).

  We first bound the bias introduced by kernel smoothing.
  Define $R^h(F)=\bE[R_n^h(F)]$.
  Let $R(y\mid x)=\bE[d_M^2(Y,y)\mid X=x]$ and
  $g_y(x')=R(y\mid x')f_X(x')$.  Then
  \[
    R^h(F) - R(F) = \int_{\cX} \Bigl[\int_{\cX} K_h(x - x')
    g_{F(x)}(x')\dd x' - g_{F(x)}(x)\Bigr]\dd x.
  \]
  The inner integral is the kernel convolution $(K_h*g_{F(x)})(x)$.
  For a symmetric second-order kernel ($\int u K(u)\dd u=0$,
  $\int u^2 K(u)\dd u<\infty$), the Taylor expansion of $g_y$
  around~$x$ gives
  \[
    (K_h*g_y)(x)-g_y(x)
    =\frac{h^2}{2}\mu_2(K)\,\tr[\nabla^2 g_y(x)]
    +O(h^3),
  \]
  provided $g_y\in C^2(\cX)$ uniformly in~$y$.  This $C^2$ regularity
  is an additional assumption on the data-generating process, requiring
  that the map $x'\mapsto\bE[d_M^2(Y,y)\mid X=x']\,f_X(x')$ is
  twice continuously differentiable with second derivatives
  uniformly bounded in~$y$.  This condition is implied by assuming
  $f_X\in C^2(\cX)$ (A2) together with the regularity condition
  that the conditional distribution $P_{Y|X=x}$ varies in a $C^2$ manner
  in~$x$ (in the sense of $C^2$ dependence of the conditional moments).
  We state this as an explicit additional regularity condition for
  the kernel-smoothed oracle inequality below.

  Under this condition, $\abs{R^h(F)-R(F)}\le C_{\mathrm{bias}} h^2$
  uniformly in~$F$, where $C_{\mathrm{bias}}$ depends on the $C^2$ norms
  of $g_y$ and the kernel moments.

  The basic inequality adapts as follows.
  By optimality of $\hat F_n^h$ for the penalized kernel-smoothed
  empirical risk $R_n^h(F)+\lambda_n \Dir(F)$:
  \[
    R_n^h(\hat F_n^h)+\lambda_n \Dir(\hat F_n^h)
    \le R_n^h(F)+\lambda_n \Dir(F)
  \]
  Write $R_n^h(G)=R^h(G)-\nu_n^h(G)$ where
  $\nu_n^h(G)=R^h(G)-R_n^h(G)$ is the centered kernel-smoothed
  empirical process.  Adding and subtracting $R(m)$:
  \begin{align}
    \cE(\hat F_n^h)+\lambda_n \Dir(\hat F_n^h)
    &\le \cE(F)+\lambda_n \Dir(F)
    +\bigl[\nu_n^h(F)-\nu_n^h(\hat F_n^h)\bigr] \notag\\
    &\quad +\bigl[R^h(\hat F_n^h)-R(\hat F_n^h)\bigr]
    -\bigl[R^h(F)-R(F)\bigr]. \label{eq:basic-ks}
  \end{align}
  The last line consists of the kernel-smoothing bias terms.
  By the bias bound above, each is bounded by $C_{\mathrm{bias}}h^2$ in absolute value.
  Since $R^h(\hat F_n^h)-R(\hat F_n^h)$ can have either sign (the
  bias is not monotone), the absolute values are needed:
  \begin{equation}\label{eq:basic-ks2}
    \cE(\hat F_n^h)+\lambda_n \Dir(\hat F_n^h)
    \le \cE(F)+\lambda_n \Dir(F)
    +\bigl[\nu_n^h(F)-\nu_n^h(\hat F_n^h)\bigr]
    + 2C_{\mathrm{bias}}h^2.
  \end{equation}

  We now control the empirical process.
  The kernel-smoothed loss is
  $\ell_F^h(x,y)=\int K_h(x-x')\bigl[d_M^2(y,F(x'))-d_M^2(y,m(x'))\bigr]
  \dd x'$.
  Since $K_h\ge 0$ with $\int K_h=1$ and
  $|d_M^2(y,F(x'))-d_M^2(y,m(x'))|\le b=2\diam(M)^2$, we have
  $|\ell_F^h(x,y)|\le b$ for all $x,y,h$.

  For the covering numbers, we use the same procedure as
  in the proof of Theorem~4.5:
  $\{F:\Dir(F)\le E_0\}$ embeds into a ball in $H^1(\cX,\bR^D)$
  of radius $\sqrt{2E_0}+C_0$ via the Nash embedding
  (using $J_1=2\Dir$ and the $L^\infty$ bound from compactness of~$M$),
  and Birman--Solomjak gives
  \[
    \log N(\epsilon,\{F:\Dir(F)\le E_0\},L^2(\cX,\bR^D))
    \le K_1\bigl((\sqrt{E_0}+C_0)/\epsilon\bigr)^p.
  \]
  The kernel-smoothed loss inherits the $L^2$ metric of the
  function class. By Jensen's inequality (since $K_h$ is a
  probability density),
  \begin{align*}
    &|\ell_{F_1}^h(x,y)-\ell_{F_2}^h(x,y)|^2\\
    &\quad\le \int K_h(x-x')\bigl[d_M^2(y,F_1(x'))-d_M^2(y,F_2(x'))\bigr]^2\dd x'\\
    &\quad\le 4\diam(M)^2\int K_h(x-x')d_M^2(F_1(x'),F_2(x'))\dd x'.
  \end{align*}
  Taking the expectation over $(X,Y)\sim P$ and using
  $\int K_h(x-x')f_X(x)\dd x\le C_X$,
  \[
    \bE[|\ell_{F_1}^h-\ell_{F_2}^h|^2]
    \le 4\diam(M)^2 C_X\int d_M^2(F_1,F_2)\dd x.
  \]
  Thus the $L^2(P)$-covering numbers of the loss class are controlled
  by the $L^2$-covering numbers of the function class, with no
  $h$-dependent factor.

  The variance of each summand satisfies
  \[
    \Var(\ell_F^h(X,Y))\le\bE[(\ell_F^h)^2]
    \le b\,\bE[|\ell_F^h|]
    \le b\bigl(\cE(F)+2\sigma_0^2+2C_{\mathrm{bias}}h^2\bigr),
  \]
  using $|\ell_F^h|\le b$ and
  $\bE[\ell_F^h]=\cE^h(F)=\cE(F)+O(h^2)$.
  We now localize and bound the empirical process.
  Define the penalized excess risk
  $V^h(F)=\cE(F)+\lambda_n\Dir(F)$ and the localized class
  $\cG_\delta^h=\{\ell_F^h:V^h(F)\le\delta^2\}$.
  For $F\in\cG_\delta^h$, $\Dir(F)\le\delta^2/\lambda_n$,
  so the covering number bound gives
  $\log N(\epsilon,\cG_\delta^h,L^2(P))
  \le K_1\bigl(C'\delta/(\sqrt{\lambda_n}\,\epsilon)\bigr)^p$,
  with $L^2$-diameter $D_\delta=C'\delta/\sqrt{\lambda_n}$.
  The variance bound gives effective standard deviation
  $\sigma_\delta\le C''\sqrt{b}\,\delta$
  for $\delta$ large enough that $\cE(F)\le\delta^2$ dominates
  $\sigma_0^2+C_{\mathrm{bias}}h^2$.

  \textbf{Case $p<2$.}\quad
  The variance-adapted Dudley integral
  \citep[Corollary~8.3]{vandeGeer2000} truncates the entropy
  integral at $\sigma_\delta$ rather than $D_\delta$:
  \begin{equation}\label{eq:dudley-ks}
    \phi_n(\delta)
    :=\bE\Bigl[\sup_{G\in\cG_\delta^h}
    \abs{\nu_n^h(G)-\nu_n^h(m)}\Bigr]
    \le\frac{C_3}{\sqrt{n}}\int_0^{\sigma_\delta}
    \sqrt{K_1(D_\delta/\epsilon)^p}\,\dd\epsilon
    +\frac{C_3\,b}{n}\,K_1(D_\delta/\sigma_\delta)^p.
  \end{equation}
  Since $p<2$, the integral converges:
  \[
    \int_0^{\sigma_\delta}
    \sqrt{K_1}\,D_\delta^{p/2}\,\epsilon^{-p/2}\,\dd\epsilon
    =\frac{2\sqrt{K_1}}{2-p}\,
    D_\delta^{p/2}\,\sigma_\delta^{1-p/2}.
  \]
  Substituting $D_\delta=C'\delta/\sqrt{\lambda_n}$ and
  $\sigma_\delta=C''\sqrt{b}\,\delta$:
  \begin{equation}\label{eq:phi-ks}
    \phi_n(\delta)
    \le\frac{C_5\,\delta}{\sqrt{n}\,\lambda_n^{p/4}}
    +\frac{C_6}{n\,\lambda_n^{p/2}},
  \end{equation}
  where the second term uses
  $(D_\delta/\sigma_\delta)^p\le C/(b^{p/2}\lambda_n^{p/2})$.
  The critical radius $\delta_n$ satisfying $\phi_n(\delta_n)=\delta_n^2$
  gives $\delta_n^2\asymp 1/(n\,\lambda_n^{p/2})$.

  \textbf{Case $p\ge 2$.}\quad
  The entropy integral in~\eqref{eq:dudley-ks} diverges.
  However, the trivial Rademacher complexity bound
  $\phi_n(\delta)\le 2b/\sqrt{n}$
  (valid for any class with envelope~$b$)
  yields a critical radius $\delta_n^2=O(n^{-1/2})$.
  With the oracle choice
  $\lambda_n\asymp n^{-2/(2+p)}$,
  we have $n\lambda_n^{p/2}=n^{2/(2+p)}\le n^{1/2}$ for $p\ge 2$,
  so $1/\sqrt{n}\le 1/(n\lambda_n^{p/2})$,
  and the trivial critical radius is absorbed into
  the penalty--bias terms.

  In both cases, the peeling argument proceeds exactly as in
  Theorem~4.5
  (cf.~\eqref{eq:V-basic}--\eqref{eq:oracle-V}),
  with $V^h$ replacing $V$, $\nu_n^h$ replacing $\Psi_n$,
  and the bound~\eqref{eq:phi-ks}
  (or the trivial bound for $p\ge 2$)
  replacing~\eqref{eq:phi-bound}.
  The Bousquet concentration inequality \citep{Bousquet2002}
  controls the tail of $\sup|\nu_n^h|$ on each dyadic shell,
  and summation over shells yields
  \[
    \bE[\cE(\hat F_n^h)+\lambda_n\Dir(\hat F_n^h)]
    \le C_1\bigl(\cE(F)+\lambda_n \Dir(F)\bigr)
    +\frac{C_2}{n\lambda_n^{p/2}}+C_7 h^2.
  \]
  Taking the infimum over $F$ with $F=m$ (so $\cE(m)=0$ and
  $\Dir(m)\le B$) gives
  $\bE[\cE(\hat F_n^h)]\le C_1\lambda_n B+C_2/(n\lambda_n^{p/2})+C_7 h^2$.
  Setting $h\asymp n^{-1/(2+p)}$ and $\lambda_n\asymp n^{-2/(2+p)}$
  balances all three terms at rate $n^{-2/(2+p)}$.
\end{proof}


\section{Geometric data for the experimental manifolds}\label{supp:geom-data}

This section collects the explicit geometric quantities needed to
implement harmonic map regression on the five manifolds used in
Section~8 of the main text.

\medskip\noindent\textbf{Sphere $S^2$.}
The unit sphere is embedded as $S^2=\{x\in\bR^3:|x|=1\}$.
It has constant sectional curvature $\kappa=+1$, injectivity radius $\inj(S^2)=\pi$,
and convexity radius $\rho_{S^2}=\pi/2$.
The exponential map is given by
$\exp_y(v)=\cos(|v|)\,y+\sin(|v|)\,v/|v|$ for $v\in T_yS^2$,
and the logarithmic map is
$\log_y(q)=\theta\,(q-\langle q,y\rangle y)/|q-\langle q,y\rangle y|$,
where $\theta=\arccos(\langle y,q\rangle)$.
The geodesic distance is $d_{S^2}(y,q)=\arccos(\langle y,q\rangle)$.
The second fundamental form is $A(y)(u,v)=-\langle u,v\rangle\,y$ (the
outward normal), and the curvature factor is $\eta(1,r_0)=r_0/\tan(r_0)$.

\medskip\noindent\textbf{Hyperbolic plane $\bH^2$.}
We realize $\bH^2$ via the hyperboloid model
$\bH^2=\{x\in\bR^{2,1}:\langle x,x\rangle_{2,1}=-1,\;x_3>0\}$,
equipped with the Minkowski inner product
$\langle u,v\rangle_{2,1}=u_1 v_1+u_2 v_2-u_3 v_3$.
The sectional curvature is $\kappa=-1$, and both the injectivity radius and
convexity radius are infinite: $\inj(\bH^2)=\rho_{\bH^2}=+\infty$.
Since $\bH^2$ is NPC, the curvature factor satisfies $\eta=1$ unconditionally.
The exponential map is
$\exp_y(v)=\cosh(|v|)\,y+\sinh(|v|)\,v/|v|$ for
$v\in T_y\bH^2$ with $|v|=\sqrt{\langle v,v\rangle_{2,1}}$,
and the logarithmic map is
$\log_y(q)=\theta\,(q+\langle y,q\rangle_{2,1}y)/|q+\langle y,q\rangle_{2,1}y|$,
where $\theta=\operatorname{arccosh}(-\langle y,q\rangle_{2,1})$.
The geodesic distance is
$d_{\bH^2}(y,q)=\operatorname{arccosh}(-\langle y,q\rangle_{2,1})$.

\medskip\noindent\textbf{Symmetric positive-definite matrices $\mathrm{Sym}^+(2)$.}
The space $\mathrm{Sym}^+(2)$ is equipped with the affine-invariant metric
$\langle U,V\rangle_P=\tr(P^{-1}U\,P^{-1}V)$.
All sectional curvatures are non-positive ($\kappa\le 0$), so the space is NPC with
$\eta=1$ and $\inj(\mathrm{Sym}^+(2))=+\infty$.
The exponential map, logarithmic map, and geodesic distance are
\[
  \exp_P(V)=P^{1/2}\exp\!\bigl(P^{-1/2}VP^{-1/2}\bigr)\,P^{1/2},\quad
  \log_P(Q)=P^{1/2}\log\!\bigl(P^{-1/2}QP^{-1/2}\bigr)\,P^{1/2},
\]
and $d(P,Q)=\bigl\|\log(P^{-1/2}QP^{-1/2})\bigr\|_F$, respectively.

\medskip\noindent\textbf{Rotation group $SO(3)$.}
We equip $SO(3)$ with the bi-invariant metric
$\langle U,V\rangle_R=\frac{1}{2}\tr(U^\top V)$ on
$T_R SO(3)\cong\{R\Omega:\Omega^\top=-\Omega\}$.
The sectional curvature is constant at $\kappa=1/8$ under this metric,
with injectivity radius $\inj(SO(3))=\pi$ and convexity radius $\rho_{SO(3)}=\pi/2$.
The fundamental group is $\pi_1(SO(3))=\bZ/2$.
The exponential map is $\exp_R(\Omega)=R\,\mathrm{expm}(\Omega)$, where
$\mathrm{expm}$ denotes the matrix exponential, and the logarithmic map is
$\log_R(Q)=R\,\mathrm{logm}(R^\top Q)$.
The geodesic distance is
$d(R,Q)=|\theta|/\sqrt{2}$, where
$\theta=\arccos((\tr(R^\top Q)-1)/2)$.

\medskip\noindent\textbf{Flat torus $T^2$.}
The flat torus is defined as the quotient $T^2=\bR^2/(2\pi\bZ)^2$ with the flat metric
inherited from~$\bR^2$.
The sectional curvature vanishes ($\kappa=0$), and the injectivity and convexity radii are
$\inj(T^2)=\pi$ and $\rho_{T^2}=\pi/2$, respectively.
The fundamental group is $\pi_1(T^2)=\bZ^2$.
The exponential map is $\exp_y(v)=(y+v)\bmod 2\pi$, and the logarithmic map
$\log_y(q)$ returns the representative of $q-y$ in
$(-\pi,\pi]^2$.
The geodesic distance is $d_{T^2}(y,q)=|(q-y)\bmod(-\pi,\pi]|$.


\section{Additional numerical experiments}\label{sec:additional-numerics}

This section contains the full numerical results table for the
comprehensive simulation study (Section~8 of the main text), as well as
the curvature sensitivity and winding experiments.

\subsection{Comprehensive simulation results}

Table~\ref{tab:comprehensive} reports the full MISE values for all five
manifolds ($S^2$, $\bH^2$, $\mathrm{Sym}^+(2)$, $SO(3)$-wind, $T^2$)
across sample sizes $n\in\{100,\allowbreak 200,\allowbreak
400,\allowbreak 800\}$.
The experimental setup is described in Section~8 of the main text.

\begin{table}[H]
\centering\footnotesize
\setlength{\tabcolsep}{3.5pt}
\caption{MISE (standard error) across five manifolds, averaged over 15 replications.
  The lowest MISE in each setting is \textbf{bolded}.
  See Section~8 of the main text for experimental details.}\label{tab:comprehensive}
\begin{tabular}{l l c c c c}
  \hline
  Manifold & Method & $n=100$ & $n=200$ & $n=400$ & $n=800$ \\
  \hline
  $S^2$ & Proposed & \textbf{0.0108} (0.0011) & \textbf{0.0050} (0.0003) & \textbf{0.0034} (0.0002) & \textbf{0.0033} (0.0002) \\
   & Extrinsic spline & 0.0180 (0.0013) & 0.0124 (0.0012) & 0.0103 (0.0007) & 0.0091 (0.0004) \\
   & TV-Fr\'echet reg. & 0.0172 (0.0014) & 0.0110 (0.0006) & 0.0065 (0.0003) & 0.0053 (0.0002) \\
   & Fr\'echet reg. & 0.2772 (0.0204) & 0.2337 (0.0106) & 0.2479 (0.0246) & 0.2415 (0.0071) \\
   & Geodesic reg. & 0.1312 (0.0128) & 0.0829 (0.0061) & 0.0825 (0.0122) & 0.0744 (0.0039) \\
  \hline
  $\bH^2$ & Proposed & \textbf{0.0422} (0.0023) & \textbf{0.0239} (0.0016) & \textbf{0.0171} (0.0013) & \textbf{0.0115} (0.0006) \\
   & Extrinsic spline & 0.2037 (0.0422) & 0.1778 (0.0438) & 0.1109 (0.0169) & 0.1175 (0.0193) \\
   & TV-Fr\'echet reg. & 0.0746 (0.0107) & 0.0346 (0.0026) & 0.0262 (0.0016) & 0.0175 (0.0009) \\
   & Fr\'echet reg. & 0.9654 (0.0095) & 0.9696 (0.0070) & 0.9567 (0.0033) & 0.9577 (0.0020) \\
   & Geodesic reg. & 1.1957 (0.0078) & 1.2149 (0.0045) & 1.2080 (0.0016) & 1.2065 (0.0016) \\
  \hline
  $\mathrm{Sym}^+(2)$ & Proposed & \textbf{0.0060} (0.0004) & \textbf{0.0029} (0.0002) & \textbf{0.0029} (0.0002) & 0.0029 (0.0001) \\
   & Extrinsic spline & 0.0065 (0.0011) & 0.0040 (0.0003) & 0.0036 (0.0003) & 0.0030 (0.0001) \\
   & TV-Fr\'echet reg. & 0.0078 (0.0006) & 0.0038 (0.0002) & 0.0033 (0.0003) & \textbf{0.0024} (0.0001) \\
   & Fr\'echet reg. & 0.0604 (0.0008) & 0.0606 (0.0007) & 0.0580 (0.0005) & 0.0580 (0.0005) \\
   & Geodesic reg. & 0.0609 (0.0009) & 0.0610 (0.0007) & 0.0590 (0.0005) & 0.0587 (0.0005) \\
  \hline
  $SO(3)$-wind & Proposed & 0.0009 (0.0001) & \textbf{0.0003} (0.0000) & \textbf{0.0002} (0.0000) & \textbf{0.0002} (0.0000) \\
   & Extrinsic spline & \textbf{0.0008} (0.0001) & 0.0006 (0.0000) & 0.0005 (0.0000) & 0.0005 (0.0000) \\
   & TV-Fr\'echet reg. & 0.0014 (0.0001) & 0.0007 (0.0000) & 0.0004 (0.0000) & 0.0002 (0.0000) \\
   & Fr\'echet reg. & 0.1867 (0.0256) & 0.1632 (0.0180) & 0.1199 (0.0077) & 0.1116 (0.0086) \\
   & Geodesic reg. & 0.5496 (0.0429) & 0.6136 (0.0151) & 0.6882 (0.0408) & 0.8529 (0.1251) \\
  \hline
  $T^2$ & Proposed & \textbf{0.0005} (0.0000) & \textbf{0.0003} (0.0000) & \textbf{0.0002} (0.0000) & \textbf{0.0002} (0.0000) \\
   & Extrinsic spline & 0.0457 (0.0007) & 0.0442 (0.0006) & 0.0453 (0.0003) & 0.0456 (0.0002) \\
   & TV-Fr\'echet reg. & 0.0009 (0.0001) & 0.0005 (0.0000) & 0.0003 (0.0000) & 0.0002 (0.0000) \\
   & Fr\'echet reg. & 0.0781 (0.0026) & 0.0722 (0.0028) & 0.0664 (0.0016) & 0.0647 (0.0009) \\
   & Geodesic reg. & 0.0729 (0.0007) & 0.0713 (0.0004) & 0.0711 (0.0003) & 0.0707 (0.0002) \\
  \hline
\end{tabular}
\end{table}

\subsection{Curvature sensitivity on $S^2(R)$}

We vary the sectional curvature $\kappa=1/R^2$ on spheres $S^2(R)$ of
varying radius.  Table~\ref{tab:curvature-supp} reports the normalized
MISE ($=\mathrm{MISE}/R^2$, angular scale) for $\kappa\in\{4,1,0.25,0.04\}$;
Figure~\ref{fig:curvature-supp} shows the same experiment graphically.
The proposed estimator's normalized MISE is nearly $\kappa$-invariant at
large~$n$, consistent with the asymptotically curvature-free oracle
constant.  TV Fr\'echet regression is approximately $1.5$--$5\times$
larger but also relatively insensitive to~$\kappa$, while the extrinsic
spline plateaus near $0.02$ across the tested curvature values.

\begin{table}[H]
\centering\footnotesize
\setlength{\tabcolsep}{3.5pt}
\caption{Normalized MISE ($\mathrm{MISE}/R^2$) on $S^2(R)$ for varying curvature $\kappa=1/R^2$ (15 replications, $\sigma=0.25$).
  The lowest MISE in each setting is \textbf{bolded}.}\label{tab:curvature-supp}
\begin{tabular}{r l c c c}
  \hline
  $\kappa$ & Method & $n=200$ & $n=400$ & $n=800$ \\
  \hline
  $4$ & Proposed & \textbf{0.0215} (0.0013) & \textbf{0.0214} (0.0019) & 0.0226 (0.0010) \\
      & Extrinsic spline & 0.0255 (0.0025) & 0.0236 (0.0014) & 0.0243 (0.0008) \\
      & TV-Fr\'echet reg. & 0.0323 (0.0018) & 0.0241 (0.0021) & \textbf{0.0190} (0.0013) \\
      & Fr\'echet reg. & 0.3258 (0.0140) & 0.2914 (0.0112) & 0.2820 (0.0066) \\
      & Geodesic reg. & 0.1445 (0.0172) & 0.1153 (0.0155) & 0.0961 (0.0054) \\
  \hline
  $1$ & Proposed & \textbf{0.0060} (0.0004) & \textbf{0.0045} (0.0002) & \textbf{0.0045} (0.0002) \\
      & Extrinsic spline & 0.0120 (0.0008) & 0.0097 (0.0007) & 0.0095 (0.0005) \\
      & TV-Fr\'echet reg. & 0.0118 (0.0005) & 0.0078 (0.0002) & 0.0052 (0.0003) \\
      & Fr\'echet reg. & 0.2669 (0.0149) & 0.2293 (0.0100) & 0.2277 (0.0084) \\
      & Geodesic reg. & 0.1008 (0.0126) & 0.0772 (0.0072) & 0.0673 (0.0028) \\
  \hline
  $0.25$ & Proposed & \textbf{0.0015} (0.0001) & \textbf{0.0010} (0.0000) & \textbf{0.0012} (0.0001) \\
      & Extrinsic spline & 0.0199 (0.0008) & 0.0199 (0.0007) & 0.0190 (0.0003) \\
      & TV-Fr\'echet reg. & 0.0049 (0.0002) & 0.0028 (0.0001) & 0.0022 (0.0001) \\
      & Fr\'echet reg. & 0.2677 (0.0313) & 0.1988 (0.0146) & 0.1948 (0.0085) \\
      & Geodesic reg. & 0.1187 (0.0169) & 0.0652 (0.0019) & 0.0755 (0.0057) \\
  \hline
  $0.04$ & Proposed & \textbf{0.0004} (0.0000) & \textbf{0.0002} (0.0000) & \textbf{0.0002} (0.0000) \\
      & Extrinsic spline & 0.0195 (0.0008) & 0.0192 (0.0002) & 0.0194 (0.0002) \\
      & TV-Fr\'echet reg. & 0.0020 (0.0001) & 0.0012 (0.0000) & 0.0009 (0.0000) \\
      & Fr\'echet reg. & 0.2183 (0.0166) & 0.1882 (0.0073) & 0.1876 (0.0064) \\
      & Geodesic reg. & 0.1107 (0.0204) & 0.0708 (0.0042) & 0.0715 (0.0037) \\
  \hline
\end{tabular}
\end{table}

\begin{figure}[H]
\centering
\includegraphics[width=\textwidth]{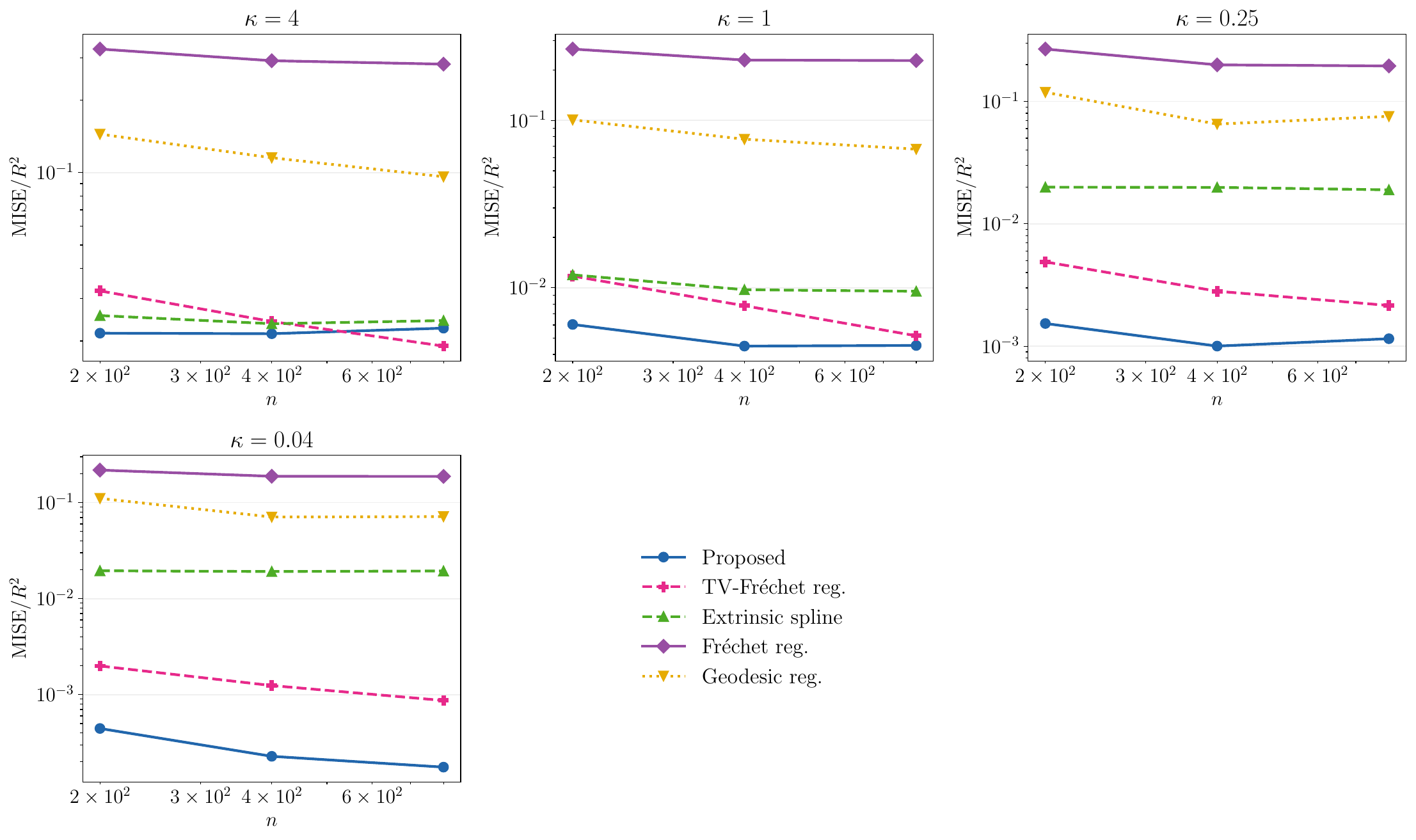}
\caption{Normalized MISE versus sample size on $S^2(R)$ for varying curvature $\kappa=1/R^2$.}
\label{fig:curvature-supp}
\end{figure}

\subsection{Winding experiment on $S^2$}

This experiment tests recovery of a curve that wraps five times around the
equator of~$S^2$.  For a curve that wraps five times, local kernel
averaging mixes distinct wraps unless bandwidth is very small.  At
$n=400$, the proposed estimator has MISE $\approx 0.008$, while
extrinsic spline and global Fr\'echet stagnate near $0.03$ and $1.72$,
respectively.  Geodesic regression has the largest
MISE ($\approx 1.98$--$2.31$), as the single-geodesic model cannot
represent a winding curve.
TV Fr\'echet converges but at a slower rate ($0.021$ at $n=400$,
$0.011$ at $n=800$).

\begin{table}[H]
\centering\footnotesize
\setlength{\tabcolsep}{3.5pt}
\caption{MISE (standard error) on $S^2$ with a 5-wrap equatorial curve
  (15 replications, $\sigma=0.25$).
  The lowest MISE in each setting is \textbf{bolded}.}\label{tab:winding-supp}
\begin{tabular}{l c c c c}
  \hline
  Method & $n=100$ & $n=200$ & $n=400$ & $n=800$ \\
  \hline
  Proposed & 0.0704 (0.0060) & \textbf{0.0243} (0.0018) & \textbf{0.0083} (0.0005) & \textbf{0.0054} (0.0002) \\
  Extrinsic spline & \textbf{0.0512} (0.0039) & 0.0331 (0.0018) & 0.0295 (0.0021) & 0.0230 (0.0013) \\
  TV-Fr\'echet reg. & 0.0848 (0.0043) & 0.0451 (0.0021) & 0.0214 (0.0010) & 0.0107 (0.0004) \\
  Fr\'echet reg. & 1.7145 (0.0148) & 1.7278 (0.0211) & 1.7200 (0.0145) & 1.7003 (0.0195) \\
  Geodesic reg. & 1.8174 (0.0815) & 1.7948 (0.0819) & 1.9778 (0.1245) & 2.3133 (0.1094) \\
  \hline
\end{tabular}
\end{table}

\subsection{Wind direction data}

Table~\ref{tab:wind-supp} reports the test-set prediction errors for
the wind direction case study described in Section~8.2 of the main text.
Five methods are compared on $n=649$ hourly NOAA observations from June,
using scattered-block cross-validation for hyperparameter selection.

\begin{table}[H]
\centering
\caption{Test-set prediction error for the wind direction data ($n=649$, June).
  MSGE: mean squared geodesic error (rad$^2$); RMGE: root-mean geodesic error (degrees).}\label{tab:wind-supp}
\begin{tabular}{l c c c}
  \hline
  Method & MSGE (rad$^2$) & RMGE (${}^\circ$) & Median (${}^\circ$) \\
  \hline
  \bf Proposed & \bf 0.275 & \bf 30.1 & \bf 18.8 \\
  Extrinsic spline & 0.832 & 52.3 & 21.5 \\
  TV-Fr\'echet reg. & 0.897 & 54.3 & 40.0 \\
  Fr\'echet reg. & 1.130 & 60.9 & 34.7 \\
  Geodesic reg. & 1.772 & 76.3 & 40.3 \\
  \hline
\end{tabular}
\end{table}

\FloatBarrier
\putbib
\end{bibunit}

\end{document}